\documentclass[12pt]{amsart}
\usepackage{amsmath,amssymb,amsbsy,amsfonts,latexsym,amsopn,amstext, amsxtra,euscript,amscd,bm}
\usepackage{comment}

\allowdisplaybreaks
\usepackage{units}
\usepackage{float}
\usepackage{titlesec}
\restylefloat{table}
\usepackage[english]{babel}                    
\usepackage{url}
\usepackage[colorlinks,linkcolor=blue,anchorcolor=blue,citecolor=blue]{hyperref}
\usepackage{color}
\usepackage[utf8]{inputenc}

\usepackage{listings}
\usepackage{xcolor}

\usepackage{blindtext}
\usepackage{multicol}

\titleformat{\section}
  {\normalfont\large\bfseries\centering}{\thesection}{1em}{}

\titleformat{\subsection}
  {\normalfont\bfseries}{\thesubsection}{1em}{}

\definecolor{codegreen}{rgb}{0,0.6,0}
\definecolor{codegray}{rgb}{0.5,0.5,0.5}
\definecolor{codepurple}{rgb}{0.58,0,0.82}
\definecolor{backcolour}{rgb}{0.95,0.95,0.92}

\lstdefinestyle{mystyle}{
    backgroundcolor=\color{backcolour},   
    commentstyle=\color{codegreen},
    keywordstyle=\color{magenta},
    numberstyle=\tiny\color{codegray},
    stringstyle=\color{codepurple},
    basicstyle=\ttfamily\footnotesize,
    breakatwhitespace=false,         
    breaklines=true,                 
    captionpos=b,                    
    keepspaces=true,                 
    numbers=left,                    
    numbersep=5pt,                  
    showspaces=false,                
    showstringspaces=false,
    showtabs=false,                  
    tabsize=2
}

\usepackage[margin=1.2in]{geometry} 

\begin{document}

\newcommand{\Mod}[1]{\ (\mathrm{mod}\ #1)}

\newtheorem{problem}{Problem}
\newtheorem{theorem}{Theorem}
\newtheorem{lemma}[theorem]{Lemma}
\newtheorem{claim}[theorem]{Claim}
\newtheorem{corollary}[theorem]{Corollary}
\newtheorem{prop}[theorem]{Proposition}
\newtheorem{definition}{Definition}
\newtheorem{question}[theorem]{Question}
\newtheorem{conjecture}{Conjecture}

\theoremstyle{remark}
\newtheorem*{note}{Note}
\newtheorem*{remark}{Remark}
\newtheorem*{remarks}{Remarks}
\def\cA{{\mathcal A}}
\def\cB{{\mathcal B}}
\def\cC{{\mathcal C}}
\def\cD{{\mathcal D}}
\def\cE{{\mathcal E}}
\def\cF{{\mathcal F}}
\def\cG{{\mathcal G}}
\def\cH{{\mathcal H}}
\def\cI{{\mathcal I}}
\def\cJ{{\mathcal J}}
\def\cK{{\mathcal K}}
\def\cL{{\mathcal L}}
\def\cM{{\mathcal M}}
\def\cN{{\mathcal N}}
\def\cO{{\mathcal O}}
\def\cP{{\mathcal P}}
\def\cQ{{\mathcal Q}}
\def\cR{{\mathcal R}}
\def\cS{{\mathcal S}}
\def\cT{{\mathcal T}}
\def\cU{{\mathcal U}}
\def\cV{{\mathcal V}}
\def\cW{{\mathcal W}}
\def\cX{{\mathcal X}}
\def\cY{{\mathcal Y}}
\def\cZ{{\mathcal Z}}

\def\A{{\mathbb A}}
\def\B{{\mathbb B}}
\def\C{{\mathbb C}}
\def\D{{\mathbb D}}
\def\E{{\mathbb E}}
\def\F{{\mathbb F}}
\def\G{{\mathbb G}}
\def\I{{\mathbb I}}
\def\J{{\mathbb J}}
\def\K{{\mathbb K}}
\def\L{{\mathbb L}}
\def\M{{\mathbb M}}
\def\N{{\mathbb N}}
\def\O{{\mathbb O}}
\def\P{{\mathbb P}}
\def\Q{{\mathbb Q}}
\def\R{{\mathbb R}}
\def\T{{\mathbb T}}
\def\U{{\mathbb U}}
\def\V{{\mathbb V}}
\def\W{{\mathbb W}}
\def\X{{\mathbb X}}
\def\Y{{\mathbb Y}}
\def\Z{{\mathbb Z}}

\def\ep{{\mathbf{e}}_p}
\def\em{{\mathbf{e}}_m}
\def\eq{{\mathbf{e}}_q}

\def\scr{\scriptstyle}
\def\\{\cr}
\def\({\left(}
\def\){\right)}
\def\[{\left[}
\def\]{\right]}
\def\<{\langle}
\def\>{\rangle}
\def\fl#1{\left\lfloor#1\right\rfloor}
\def\rf#1{\left\lceil#1\right\rceil}
\def\le{\leqslant}
\def\ge{\geqslant}
\def\eps{\varepsilon}
\def\mand{\qquad\mbox{and}\qquad}

\def\sssum{\mathop{\sum\ \sum\ \sum}}
\def\ssum{\mathop{\sum\, \sum}}
\def\ssumw{\mathop{\sum\qquad \sum}}

\def\vec#1{\mathbf{#1}}
\def\inv#1{\overline{#1}}
\def\num#1{\mathrm{num}(#1)}
\def\dist{\mathrm{dist}}

\def\fA{{\mathfrak A}}
\def\fB{{\mathfrak B}}
\def\fC{{\mathfrak C}}
\def\fU{{\mathfrak U}}
\def\fV{{\mathfrak V}}

\newcommand{\bflambda}{{\boldsymbol{\lambda}}}
\newcommand{\bfxi}{{\boldsymbol{\xi}}}
\newcommand{\bfrho}{{\boldsymbol{\rho}}}
\newcommand{\bfnu}{{\boldsymbol{\nu}}}

\def\GL{\mathrm{GL}}
\def\SL{\mathrm{SL}}

\def\Hba{\overline{\cH}_{a,m}}
\def\Hta{\widetilde{\cH}_{a,m}}
\def\Hb1{\overline{\cH}_{m}}
\def\Ht1{\widetilde{\cH}_{m}}

\def\flp#1{{\left\langle#1\right\rangle}_p}
\def\flm#1{{\left\langle#1\right\rangle}_m}
\def\dmod#1#2{\left\|#1\right\|_{#2}}
\def\dmodq#1{\left\|#1\right\|_q}

\def\Zm{\Z/m\Z}

\def\Err{{\mathbf{E}}}

\newcommand{\comm}[1]{\marginpar{%
\vskip-\baselineskip 
\raggedright\footnotesize
\itshape\hrule\smallskip#1\par\smallskip\hrule}}

\def\xxx{\vskip5pt\hrule\vskip5pt}

\def\cc#1{\textcolor{red}{#1}}

\newenvironment{nouppercase}{%
  \let\uppercase\relax%
  \renewcommand{\uppercasenonmath}[1]{}}{}
  

\title{An explicit version of Chen's theorem and the linear sieve}
\author{Matteo Bordignon, Daniel R. Johnston and Valeriia Starichkova}
\address{KTH Royal Institute of Technology, Stockholm \newline and \newline Charles University, Faculty of Mathematics and Physics, Department of Algebra, Sokolovská 83, 186 00 Praha 8, Czech Republic Department of Mathematics}
\email{matteobordignon91@gmail.com}
\address{The University of New South Wales Canberra, School of Science}
\email{daniel.johnston@unsw.edu.au}
\address{The University of New South Wales Canberra, School of Science}
\email{v.starichkova@unsw.edu.au}
\date\today
\thanks{The basis for this work was done as part of the thesis the first author wrote during the length of their PhD at the University of New South Wales Canberra. It was also partially supported by an Australian Mathematical Society Lift-off Fellowships of the first and the third author, by OP RDE project No.
CZ.$02.2.69/0.0/0.0/18\_053/0016976$ International mobility of research, technical and administrative staff at the Charles University, and by Australian RC Discovery Project DP240100186.}
\date{\today
}


\begin{abstract}
Drawing inspiration from the work of Nathanson and Yamada we prove an effective and explicit version of Chen's theorem. By contrast, existing proofs of Chen's theorem are ineffective due to their use of the Siegel-Walfisz theorem. Our main result is that every even integer larger than $\exp (\exp (32.7))$ can be written as the sum of a prime and the product of at most two primes. We also prove that all even integers $N\ge 4$ can be written as the sum of a prime and the product of at most $e^{29.3}$ primes. The main idea will be to follow a proof of Chen's theorem due to Nathanson, being more careful with the treatment of potential Siegel zeros in order to obtain an effective and explicit result. In following this framework we also prove an explicit version of the linear sieve, which substantially improves upon the previous best one by Nathanson. \newline
\newline
 \noindent \textbf{MSC classes:} 11N36, 11P32 (Primary) 11M20, 11N13 (Secondary)  
 \end{abstract}

\begin{nouppercase}
\maketitle
\end{nouppercase}

\section{Introduction}
One of the most famous problems in number theory is Goldbach's conjecture.
\begin{conjecture}[Goldbach]
For any even integer $N\ge 4$ there exist two primes $p_1$ and $p_2$, such that
\begin{equation*}
N=p_1+p_2.
\end{equation*}
\end{conjecture}
This conjecture was verified for all even $N\le 4\cdot 10^{18}$ by Oliveira e Silva \cite{TOS}. However, a complete proof appears to be out of reach for the present state of mathematics. There are two results that are arguably the nearest approximations to Goldbach's conjecture: Goldbach's weak conjecture and Chen's theorem. Goldbach's weak conjecture, also known as the ternary Goldbach problem, is a proved result.
\begin{theorem}[Vinogradov--Helfgott]
For any odd number $N \ge 7$ there exist three primes $p_1, p_2$ and $p_3$, such that
\begin{equation*}
N=p_1+p_2+p_3.
\end{equation*}
\end{theorem}
In particular, Vinogradov proved in \cite{Vinogradov} that all odd numbers larger than some constant $C$ can be written as a sum of three prime numbers. According to \cite[p. 201]{Chu}, the first explicit value of $C$ was established by Borodzkin in his unpublished doctoral dissertation and he later, in \cite{Bro}, improved the result to $C=\exp (\exp (16.038))$. After a series of further improvements, the final push to prove Goldbach's weak conjecture was done by Helfgott in \cite{Helfgott}.

In this paper we instead focus on obtaining an explicit version of Chen's theorem, first proved in 1966 by Chen \cite{Chen2, Chen1}.
\begin{theorem}[Chen]
All sufficiently large even numbers can be written as the sum of a prime and another number that is the product of at most two primes (a semi-prime).
\end{theorem}
A lot of work has been done to improve Chen's result. Simpler proofs were given in \cite{Halberstam}, \cite{Ross} and \cite{Nathanson}. Further, Chen's theorem was quantitatively improved, in the counting of the number of ways in which large enough even numbers can be written as the sum of a prime and a semi-prime, by Chen himself in \cite{Chen} and \cite{Chen3}, and further in \cite{Cai}, \cite{Wu}, \cite{Cai1} and \cite{Wu1}. Many generalizations of Chen's theorem have also been obtained. Generalizations with bounds on the prime and/or semi-prime were given in \cite{Lu1}, \cite{Cai2}, \cite{Cai3}, \cite{Li} and \cite{Cai4}. Lu and Cai \cite{Lu}, proved a version in which the prime and semi-prime are in certain arithmetic progressions. Hinz \cite{Hinz} proved a version for totally real algebraic number fields. Car \cite{Car} proved a version for $F_q[X]$, the ring of polynomials with one variable over a finite field of $q$ elements.

It is interesting to note that while a lot of effort was put into making Vinogradov's proof of Goldbach's weak conjecture completely explicit, not much effort was put into making Chen's theorem explicit. The only attempt was made by Yamada in \cite{Yamada1}, but several mistakes can be found in the proof; see \cite[(87) \& (104)]{Yamada1}, where a $\log $ term appears to be missing, and, notably, that no proof is given of the explicit version of the linear sieve that is used and that this version is inconsistent with the versions in \cite{JR},  \cite{Iwaniec} and \cite{Nathanson}. The aim of this paper is thus to obtain the first complete explicit version of Chen's theorem. Our main result is as follows.
\begin{theorem}
\label{Theo:B.}
Let $\pi_2(N)$ denote the number of representations of a given even integer $N$ as the sum of a prime and a semi-prime. If $N> \exp (\exp (32.7))$, then
\begin{equation}\label{pi2lower}
    \pi_2(N)> 2\cdot 10^{-4}\cdot\frac{U_N N}{\log^2N},
\end{equation}
where, with $\gamma$ the Euler--Mclaurin constant,
\begin{equation}
U_N:=2 e^{-\gamma}\prod_{p>2}\left( 1-\frac{1}{(p-1)^2}\right)\prod_{\substack{p>2\\p|N}}\frac{p-1}{p-2}.\label{UNeq}
\end{equation}
\end{theorem}
Here we note that directly correcting the mistakes in \cite{Yamada1} would lead to a much worse lower bound than the $\exp (\exp (36))$ claimed in the paper.  While the lower bound $\exp (\exp (32.7))$ that we prove appears only to be a modest improvement on the one in the incomplete work of Yamada \cite{Yamada1}, this is mainly due to this constant being around the optimal that is possible to obtain with the present method, for more details see \S \ref{section:B}.  \newline
Using \eqref{pi2lower}, we will prove the following corollary, which is essentially a stronger form of Chen's theorem.
\begin{corollary}\label{distinctcor}
    Every even integer $N>\exp(\exp(32.7))$ can be represented as the sum of a prime and a square-free number with at most two prime factors.
\end{corollary}
In particular, Corollary \ref{distinctcor} implies that we can take the prime factors in Chen's theorem to be distinct. 

We also prove the following, which is a simple consequence of Theorem \ref{Theo:B.} and a result of Dudek \cite{Dudek1}, where he proves that all integers larger than two can be written as the sum of a prime and a square-free number.
\begin{theorem}
\label{Theo:B1.}
All even integers $N\ge 4$ can be written as the sum of a prime and the product of at most $e^{29.3}$ primes.
\end{theorem}

Theorem \ref{Theo:B1.} makes explicit a result of R\'enyi \cite{Renyi}. We also remark that the proof of Theorem \ref{Theo:B1.} is quite wasteful, meaning the number $e^{29.3}$ can certainly be lowered with more work. The second and third authors are currently writing a follow up article in this direction.

For the proof of Theorem \ref{Theo:B.} we will draw inspiration from the works of Nathanson in \cite{Nathanson} and Yamada in \cite{Yamada1}. In particular, Nathanson \cite[Theorem 10.1]{Nathanson} gives a proof of Chen's theorem, of which Yamada \cite[Theorem 1.1]{Yamada1} made a partial attempt to make explicit. The technique uses an explicit version of the linear sieve to obtain upper and lower bounds for the number of certain sifted integers, combined with explicit versions of the prime number theorem for primes in arithmetic progression and an upper bound for an exceptional zero of a Dirichlet $L$-function. Note that previous proofs of Chen's theorem, including Nathanson's \cite{Nathanson}, are ineffective. Hence, to obtain an explicit version a modified approach is required. That is, one cannot simply repeat each step of these past proofs whilst keeping track of the error terms. The main problem is making the error term in the linear sieve explicit, accounting for the possible existence of an exceptional zero. We address this problem by splitting the argument into two cases: one when the exceptional modulus is ‘small' and one when it is  ‘large'. In the case when the modulus is ‘small' it is possible to absorb the Siegel zero into the error term concerning primes in arithmetic progression. In the second case, a variant of the inclusion-exclusion principle is used, avoiding the Siegel zero but introducing more complicated error terms.

For reference, all of the notation used at the paper is given at the end, in Section \ref{sectnot}. Otherwise, an outline of the paper is as follows. In Section \ref{section:LS} we prove a new explicit version of the linear sieve and introduce other preliminary results and definitions. In Section \ref{section:sur} we state several lemmas from existing literature that will be frequently used in the later sections. For our application of the linear sieve, the remainder term essentially corresponds with the error term in the prime number theorem for arithmetic progressions. Thus, in Section \ref{sec:PNT} we prove a collection of explicit results regarding primes in arithmetic progressions. Here, we use a recent result for the prime number theorem for primes in arithmetic progressions given by the first author in \cite{Bordignon4} and an improved upper bound for the exceptional zero given by the first author in \cite{Bordignon1, Bordignon2}. In Section \ref{sec:ps}, we then use the preceding results to set up all the required preliminaries for sieving. In Sections \ref{section:1}, \ref{section:2} and \ref{section:3} we obtain upper and lower bounds for the sifted integer sets. In Section \ref{section:B} we prove Theorem \ref{Theo:B.} and in Section \ref{section:B1} we conclude by proving Corollary \ref{distinctcor} and Theorem \ref{Theo:B1.}.

\section*{Acknowledgements}
We would like to thank our supervisor Tim Trudgian for his help in developing this paper and his insightful comments. We would also like to thank Leo Goldmakher, Bryce Kerr and Kevin O’Bryant for their helpful comments and suggestions.

\section{An explicit version of the linear sieve}\label{sievesect}
\label{section:LS} 
Nathanson's proof of Chen's theorem \cite[\S 10]{Nathanson}, which we roughly follow, requires multiple applications of lower and upper bounds of the linear sieve. Therefore, in order to obtain the best explicit result, we prove the following theorem, which is an improved and more general version of Nathanson's linear sieve bounds in \cite[Theorem 9.7]{Nathanson}. Notably, our result is written in a general form, and can thus be used in other applications beyond Chen's Theorem.
\begin{theorem}[The linear-sieve, explicit version]
\label{theo:JR}
Let $A=\{a(n)\}_{n=1}^{\infty}$ be an arithmetic function such that
\begin{equation*}
a(n)\ge 0 \quad \text{for all} \quad n \quad \text{and} \quad 
|A|=\sum_{n=1}^{\infty} a(n)< \infty.
\end{equation*}
Let $\mathbb{P}$ be a set of prime numbers and for $z\ge 2$, let
\begin{equation*}
P(z):=\prod_{\substack{p \in \mathbb{P}\\ p<z}}p.
\end{equation*}
Let 
\begin{equation*}
S(A,\mathbb{P}, z):=\sum_{\substack{n=1\\ (n, P(z))=1}}^{\infty}a(n).
\end{equation*}
For every $n\ge 1$, let $g_n(d)$ be a multiplicative function such that
\begin{equation*}
0\le g_n(p)< 1 \quad \text{for all} \quad p \in \mathbb{P}.
\end{equation*}
Define $|A_d|$ and $r(d)$ by
\begin{equation}\label{eq:Adrd}
    |A_d|:=\sum_{\substack{n=1\\ d|n}}^{\infty} a(n)= \sum_{n=1}^{\infty} a(n)g_n(d)+r(d).
\end{equation}
Let $\mathbb{Q}\subseteq \mathbb{P}$, and $Q$ be the product of its primes. Suppose that, for some $\varepsilon$ satisfying $0<\varepsilon\le 1/74$, the inequality
\begin{equation}
\label{eq:cond1}
\prod_{\substack{p\in \mathbb{P}/\mathbb{Q}\\ u \le p < z}}(1-g_n(p))^{-1}<(1+\varepsilon) \frac{\log z}{\log u},
\end{equation}
holds for all $n$ and $1<u<z$. Then, for any $D\ge z$ we have the upper bound
\begin{equation}\label{suppereq}
    S(A,\mathbb{P},z)<(F(s)+\varepsilon C_{1}(\varepsilon) e^2h(s))X_A+R,
\end{equation}
and for any $D\ge z^2$ we have the lower bound
\begin{equation}\label{slowereq}
    S(A,\mathbb{P},z)>(f(s)-\varepsilon C_{2}(\varepsilon) e^2h(s))X_A-R,
\end{equation}
where
\begin{equation*}
s:=\frac{\log D}{\log z},
\end{equation*}
\begin{equation} \label{eq-def-h(s)}
h(s):= \begin{cases} 
      e^{-2} & 1\le s \le 2, \\
      e^{-s} & 2\le s \le 3, \\
      3s^{-1}e^{-s} & s\ge 3,
      \end{cases} 
\end{equation}
$F(s)$ and $f(s)$ are the two functions defined by the following delay differential equations (see \cite{Cai1}):
\begin{equation} \label{def-f(s)-F(s)}
    \begin{cases}
        F(s) = \frac{2 e^{\gamma}}{s}, \quad f(s) = 0 &\text{for } 0 < s \le 2,\\
        (sF(s))' = f(s - 1), \quad (sf(s))' = F(s-1) &\text{for } s \ge 2,
    \end{cases}
\end{equation}
$C_1(\varepsilon)$ and $C_2(\varepsilon)$ come from Table \ref{tab:fF},
\begin{equation} \label{eq: def-X_A}
    X_A:=\sum_{n=1}^{\infty} a(n) \prod_{p\mid P(z)}(1-g_n(p))=|A|\prod_{p\mid P(z)}(1-g_n(p)),
\end{equation}
and the remainder term is
\begin{equation}\label{Rdef}
    R:=\sum_{\substack{d|P(z)\\ d<QD}}|r(d)|.
\end{equation}
If there is a multiplicative function $g(d)$ such that $g_n(d)=g(d)$ for all $n$, then
\begin{equation}\label{vzdef}
X_A=V(z)|A|, \quad \text{where} \quad 
V(z):=\prod_{p|P(z)}(1-g(p)).
\end{equation}
\end{theorem}
\begin{table}
\centering
    \begin{tabular}{ | l | l | l |}
    \hline
    $\varepsilon^{-1}$ & $C_1(\varepsilon) $ & $C_2(\varepsilon)$ \\  
    \hline
    $74$& $631$ & $630$ \\
    \hline
    $76$& $559$ & $559$ \\
    \hline
    $78$& $504$ & $504$ \\
    \hline
    $80$& $461$ & $461$ \\
    \hline
    $85$& $386$ & $386$ \\
    \hline
    $90$& $336$ & $337$ \\
    \hline
    \end{tabular}
\begin{tabular}{ | l | l | l |}
    \hline
    $\varepsilon^{-1}$ & $C_1(\varepsilon) $ & $C_2(\varepsilon)$  \\
    \hline
    $95$& $302$ & $302$ \\
    \hline
    $100$& $276$ & $277$ \\
    \hline
    $120$& $218$& $219$\\
    \hline
    $140$& $189$& $190$ \\
    \hline
    $160$& $172$ & $173$ \\
    \hline
    $180$ & $161$ & $162$ \\
    \hline
    \end{tabular}
    \begin{tabular}{ | l | l | l |}
    \hline
    $\varepsilon^{-1}$ & $C_1(\varepsilon) $ & $C_2(\varepsilon)$ \\  
    \hline
    $200$& $153$ & $154$ \\
    \hline
    $300$& $133$ & $134$ \\
    \hline
    $400$& $125$ & $126$ \\
    \hline
    $500$& $121$ & $122$\\
    \hline
    $600$& $118$ & $119$\\
    \hline
    $700$& $116$ & $117$\\
    \hline
    \end{tabular}
    \begin{tabular}{ | l | l | l |}
    \hline
    $\varepsilon^{-1}$ & $C_1(\varepsilon) $ & $C_2(\varepsilon)$ \\  
    \hline
    $800$& $115$ & $116$\\
    \hline
    $900$& $114$ & $115$\\
    \hline
    $1000$& $113$ & $114$\\
    \hline
    $2000$& $109$ & $110$\\
    \hline
    $10000$& $106$ & $108$\\
    \hline
    $100000$& $106$ & $107$\\
    \hline
    \end{tabular}   
\caption{Values for $C_1(\varepsilon)$ and $C_2(\varepsilon)$.}  
\label{tab:fF}  
\end{table}
To read Table \ref{tab:fF} it is useful to remember that $C_1(\varepsilon)$ and $C_2(\varepsilon)$ are decreasing in $\varepsilon^{-1}$. This will be made evident from the proof.
We also note that we chose to stop at $\varepsilon^{-1}=100000$ as any larger value would give the same upper bound (to the nearest integer) for $C_1(\varepsilon)$ and $C_2(\varepsilon)$.

We observe that for $\varepsilon=1/200$ and all $s\ge 1$
\begin{equation}\label{conste200eq}
    154\varepsilon e^2h(s)  \le 0.77 ~\text{and} ~153\varepsilon  e^2h(s)\le 0.765.
\end{equation}
This gives a uniform upper bound for the `constants' appearing in Theorem~\ref{theo:JR}. In \cite{Nathanson} the equivalent upper bounds for $\varepsilon C_i(\varepsilon) e^2 h(s)$ are $\approx 2210$, and thus ours in \eqref{conste200eq} are around $3000$ times smaller. 
\subsection{Introduction}\label{introsub}
To begin with, we note that the functions $f$ and $F$ from Theorem \ref{theo:JR} may be equivalently defined as follows \cite[Theorem 9.4]{Nathanson}:
\begin{align}
    &F(s):= 1 + \sum_{\substack{n = 1 \\ n \text{ odd}}} f_n(s), ~\text{ for } s \ge 1, \label{bigFdeffn}\\
    &f(s):= 1 - \sum_{\substack{n = 1 \\ n \text{ odd}}} f_n(s), ~\text{ for } s \ge 2. \label{lilFdeffn}
\end{align}
Here,
\begin{equation}
\label{eq:f11}
    sf_1(s):=\begin{cases} 
      3-s,& 1\le s \le 3, \\
      0 & s>3.
      \end{cases} 
\end{equation} 
Then, if $n\ge 2$ is even and $s\ge 2$, or if $n\ge 3$ is odd and $s\ge 3$,
\begin{equation}
\label{eq:f22}
sf_n(s):=\int_s^{\infty}f_{n-1}(t-1)\mathrm{d}t.
\end{equation}
Finally, if $n$ is odd and $1\le s \le 3$, then
\begin{equation}
\label{eq:odd}
    sf_n(s):=3f_n(3)=\int_3^{\infty}f_{n-1}(t-1)\mathrm{d}t.
\end{equation}

A key part of the proof of Theorem \ref{theo:JR} is based on finding accurate upper bounds for the functions $f_n(s)$. We will focus on improving Nathanson's bound on $f_n(s)$ in \cite[Chapter 9]{Nathanson}, combining his analytic approach with a more computational one.
To this aim we introduce an elementary method, namely, approximating an integral by Riemann sums, to obtain an upper bound for the function $f_n(s)$ for `small' $n$. The chosen upper bound function is $h(s)$ from \eqref{eq-def-h(s)} which was also used by Nathanson and indeed appears to be a numerically good approximation for $f_n(s)$. Our aim will be to find a value $c_n$ such that 
\begin{equation}
\label{eq:fh}
    f_n(s)\le 2e^2 (c_n)^{n-1}h(s)
\end{equation}
when $n$ is odd and $s\ge 1$, or if $n$ is even and $s\ge 2$. Note that by the definitions \eqref{eq:f11} and \eqref{eq-def-h(s)}, we can take $c_1=1$. Thus, it suffices to find values for $c_n$ for $n\ge 2$.

Our computational approach, which we detail in Section \ref{EXPS}, yields the following values for $c_n$ when $2\le n\le 500$. 
\begin{lemma}
\label{lemma:fsmall}
The bound \eqref{eq:fh} holds with $c_n$ as in Table \ref{tab:cn1} below.
\begin{table}[H]
\caption{Valid values of $c_n$ for $2\le n\le 500$}  
\label{tab:cn1}  
\centering
    \begin{tabular}{ | l | l | }
    \hline
    $n$ & $c_n$   \\  
    \hline
    $2$& $0.33$    \\  
    \hline
    $3$& $0.39$ \\
    \hline
    $4$& $0.45$ \\
    \hline
    $5$& $0.51$  \\
    \hline
    $6$& $0.54$  \\
    \hline
    \end{tabular}
     \begin{tabular}{ | l | l | }
    \hline
    $n$ & $c_n$    \\
    \hline
    $7$ & $0.57$  \\
    \hline
    $8$& $0.58$  \\
    \hline
     $9-10$ &  $0.61$  \\  
    \hline
    $11-12$ & $0.63$  \\
    \hline
    $13$& $0.64$ \\
    \hline
    \end{tabular}
     \begin{tabular}{ | l | l | }
    \hline
     $n $ & $c_n$  \\ 
     \hline
    $14$ &$0.65$ \\
    \hline
    $15-18$ &$0.66$ \\
    \hline
    $19-20$ &$0.67$ \\
    \hline
    $21-26$ &$0.68$\\
    \hline
    $27-34$ & $0.69$ \\  
    \hline
    \end{tabular}
    \begin{tabular}{ | l | l | }
    \hline
    $n$ & $c_n$   \\  
    \hline
    $35-46$ & $0.7$\\
    \hline
    $47-82$ & $0.71$\\
    \hline
    $83-345$ & $0.72$\\
    \hline
    $346-500$ & $0.73$\\
    \hline
    &\\
    \hline
    \end{tabular}
\end{table}
\end{lemma}
Certainly, it is possible to compute these values of $c_n$ to more decimal places and larger $n$ if required. We then extend Lemma~\ref{lemma:fsmall} to all $n$ by using the following analytic result of Nathanson.
\begin{lemma}[{\cite[Lemma 9.7]{Nathanson}}]\label{lemma:f_n}
    For all $n\ge 2$ we have that \eqref{eq:fh} holds with
    \begin{equation*}
        c_n=0.9607.
    \end{equation*}
\end{lemma}
The rest of the proof of Theorem \ref{theo:JR} is laid out as follows. In Section \ref{EXPS} we prove Lemma~\ref{lemma:fsmall}. Then, in Section \ref{AA} we provide some useful bounds relating to the function $h(s)$. Finally, in Section \ref{ELS} we finish the proof of Theorem \ref{theo:JR}. Explicit versions of the inequality \eqref{eq:cond1} are then included in a supplementary section \ref{subsec: expl-epsilon}.

\subsection{Numerical approximation of $f_n(s)$}
\label{EXPS}
Using the definition \eqref{eq:f11}--\eqref{eq:odd} of $f_n(s)$ we obtain
\begin{equation}
\label{eq:f2}
sf_2(s)= \begin{cases} 
      s-3\log(s-1)+3 \log 3-4,~& 2\le s \le 4, \\
      0,~& s\ge 4,
      \end{cases} 
\end{equation}
but for larger $n$ the solution is more complicated. From \eqref{eq:f2} and the definition \eqref{eq-def-h(s)} of $h(s)$ we find $c_2=0.33$ works in \eqref{eq:fh}. For $n\ge 3$, we introduce a simple computational framework to bound $f_n(s)$ above by $h(s)$. We start by observing from \eqref{eq:f22} and \eqref{eq:odd}, that for $n\ge 4$ even and $2\le s \le 4$,
\begin{equation}
\label{eq:f4}
sf_n(s)=3f_{n-1}(3)\log\left(\frac{3}{s-1}\right)+\int_4^{\infty}f_{n-1}(t-1)\mathrm{d}t.
\end{equation}
Therefore, by \eqref{eq:odd} and \eqref{eq:f4}, we need to approximate 
\begin{equation*}
    \int_{\max(3,s)}^\infty f_{n-1}(t-1)\mathrm{d}t
\end{equation*}
when $s\ge 1$ and $n\ge 3$ is odd, and
\begin{equation*}
    \int_{\max(4,s)}^\infty f_{n-1}(t-1)\mathrm{d}t
\end{equation*}
when $s\ge 2$ and $n\ge 4$ is even. Now, since $f_n(s)$ is decreasing in $s$, and $f_n(t)=0$ for $t\ge n+2$, we can use the Riemann sum approximation (with interval length $1/1000$):
\begin{align}
    \int_{\sigma}^\infty f_{n-1}(t-1)\mathrm{d}t&\le\sum_{i=0}^{1000(n-1)}\frac{f_{n-1}\left((\sigma-1)+\frac{i}{1000}\right)}{1000}\label{sigmaapprox},
\end{align}
with $\sigma=\max(3,s)$ or $\sigma=\max(4,s)$. Here, we chose an interval length of $1/1000$ so that the bound \eqref{sigmaapprox} was sharp enough for our purposes whilst still easy to calculate on a modern computer. In particular, we ran some longer computations with smaller intervals and only found a marginal improvement in our results. 

After bounding the relevant integral in \eqref{sigmaapprox} one then obtains an upper bound for $f_n(s)$ by \eqref{eq:f22}, \eqref{eq:odd} or \eqref{eq:f4}. For our purposes, we used \eqref{sigmaapprox} recursively to approximate $f_n(1+i/1000)$ for $3\le n\le 499$ odd and $0\le i\le 1000(n+1)$, and $f_n(2+i/1000)$ for $4\le n\le 500$ even and $0\le i\le 1000n$.

Finally, we can use our bounds for $f_n(s)$ to compute $c_n$ for $3\le n\le 500$. More precisely, we let $x_i=1+i/1000$ if $n$ is odd, and $x_i=2+i/1000$ if $n$ is even. Since $f_n(s)$ and $h(s)$ are decreasing, $f_n\left(x_{i}\right)/h\left(x_{i+1}\right)$ is an upper bound for $f_n(s)/h(s)$ for all $s \in [x_i, x_{i+1}]$. Computing the maximum such bound over all intervals $[x_i,x_{i+1}]$ to 2 decimal places (rounded up) then gives the values for $c_n$ in Table \ref{tab:cn1} and thereby proves Lemma \ref{lemma:fsmall}.
\subsection{Bounds relating to $h(s)$}
\label{AA}
In this section we prove some useful bounds relating to $h(s)$. Compared to Nathanson, we split our results into more regions for $s$. This piecewise approach ultimately yields better values for $C_1(\varepsilon)$ and $C_2(\varepsilon)$ in Table~\ref{tab:fF}. 

To begin with, we give the following lemma, which readily follows from the definition \eqref{eq-def-h(s)} of $h(s)$.
\begin{lemma}\label{lem:hsm1}
    Let $\gamma_3=4e/3$ and for any $2\le s_0\le 2.8$, let $\gamma_{s_0}=e^{s_0-2}$. Then
    \begin{align}
        h(s-1)&\le
        \begin{cases}
            \gamma_{s_0}\cdot h(s),&\text{if}\ 2\le s\le s_0+0.2,\\
            \gamma_3\cdot h(s),&\text{if}\ s\ge 3.\\
        \end{cases}
        \label{eq:h11}
    \end{align}
    Moreover, for $1\le s\le 3$
    \begin{equation}
        \frac{3}{s}h(2)\le 3h(s).\label{eq:h12}
    \end{equation}
\end{lemma}
Next, for $s\ge 2$, we define
\begin{equation}
\label{eq:H11}
H(s):=\int_s^{\infty} h(t-1) \mathrm{d}t \quad\text{and} \quad\alpha:=\frac{H(2)}{2h(2)}=\frac{e^2H(2)}{2}=0.96068\ldots.
\end{equation}
The following lemma is then an improved version of \cite[Lemma 9.6]{Nathanson} which bounds $H(s)$ in different ranges.
\begin{lemma}
\label{lemma:forg}
Let $\kappa_{2} = 0.9607$, $\kappa_{2.2} = 0.9557$, $\kappa_{2.4} = 0.9457$, $\kappa_{2.6} = 0.9261$, $\kappa_{2.8} = 0.8914$ and $\kappa_3 = 0.8349$. Then for $s_0 \in \{2, 2.2, 2.4, 2.6, 2.8,3\}$, we have
\begin{align}
    H(s)&\le \kappa_{s_0}\cdot s\cdot h(s),&\text{if}\ s\ge s_0. \label{eq: H(s)-upper}
\end{align}
In addition, for $\tilde{\kappa}=0.9214$,
\begin{align}    
    H(3)&\le \tilde{\kappa}\cdot s\cdot h(s),&\text{if}\ 1\le s\le 3.\label{eq: H(3)-upper}
\end{align}
\end{lemma}
\begin{proof}
We start by proving \eqref{eq: H(s)-upper}. Firstly, for $s \ge 4$, we show the stronger bound 
\begin{equation*}
    H(s)\le 0.81\cdot s\cdot h(s).
\end{equation*}
This follows from a repeated application of integration by parts:
\begin{align*}
    \frac{1}{3}H(s)&=\int_{s-1}^{\infty} \frac{e^{-t}}{t} \mathrm{d}t\\
    &= e^{1-s} \left( \frac{1}{s-1} - \frac{1}{(s-1)^2} + \frac{2}{(s-1)^3} \right) - \int_{s-1}^\infty\frac{6e^{-t}}{t^4}\mathrm{d}t\\ 
    &\le e \left( \frac{1}{3} - \frac{1}{3^2} + \frac{2}{3^3} \right) e^{-s}\\
    &\le 0.81 e^{-s}.
\end{align*}
That is,
\begin{equation*}
    H(s)\le 0.81\cdot 3e^{-s}=0.81\cdot s\cdot h(s)
\end{equation*}
as claimed. Next we consider $3\le s\le 4$. In this case
\begin{align*}
    \frac{1}{3}e^{s}H(s)&=e^{s}\left(\frac{1}{3}\int_{s-1}^3e^{-t}\mathrm{d}t+\int_{3}^\infty\frac{e^{-t}}{t}\mathrm{d}t\right)\\
    &\le \frac{1}{3}\left(e-e^{s-3}\right)+0.01305\cdot e^{s},
\end{align*}
which is maximised at $s=3$. Thus, 
\begin{equation*}
    \frac{1}{3}e^{-s}H(s)\le\frac{1}{3}\left(e-1\right)+e^3\cdot 0.01305\le 0.8349,
\end{equation*}
so that $H(s)\le \kappa_3\cdot s\cdot h(s)$ as required.

In the case $2\le s\le 3$, we have $h(s) = e^{-s}$ and 
\begin{align*}
    H(s) = \int_{s-1}^2 e^{-2} \mathrm{d}t + \int_{2}^3 e^{-t} \mathrm{d}t + 3 \int_{3}^{\infty} e^{-t} t^{-1} \mathrm{d}t = (4-s) e^{-2} - e^{-3} + 3\int_{3}^{\infty} e^{-t} t^{-1} \mathrm{d}t,
\end{align*}
hence
\begin{align}\label{Hs23eq}
    \frac{e^s}{s} H(s) = \frac{e^s}{s} \left((4-s) e^{-2} - e^{-3} + 3\int_{3}^{\infty} e^{-t} t^{-1} \mathrm{d}t\right).
\end{align}
Standard calculus arguments reveal that the right-hand side of \eqref{Hs23eq} is decreasing. Hence, for every $s_0 \in \{2, 2.2, 2.4, 2.6, 2.8\}$ we can substitute $s_0$ into \eqref{Hs23eq} to obtain that $H(s)\le \kappa_{s_0}\cdot s\cdot h(s)$ for $s \ge s_0$.

Finally, let us prove \eqref{eq: H(3)-upper}. If $2\le s\le 3$,
\begin{equation*}
H(3)=H(2)-e^{-2}=(2\alpha-1)e^{-2}\le \frac{e(2\alpha-1)}{3}s\cdot h(s),
\end{equation*}
since $se^{-s}$ is decreasing for $s\ge 1$. Then, if $1\le s\le 2$,
\begin{equation*}
H(3)=(2\alpha-1)e^{-2}\le (2\alpha-1)sh(s).
\end{equation*}
The desired result then follows upon noting that
\begin{equation*}
    \max\left\{\frac{e(2\alpha-1)}{3},2\alpha-1\right\}=2\alpha-1\le 0.9214= \tilde{\kappa}.\qedhere
\end{equation*}
\end{proof}

\begin{remarks}\ 
\begin{enumerate}
    \item[1.] In \cite[Lemma 9.7]{Nathanson}, the bound $H(s)\le \alpha\cdot s\cdot h(s)$ is used to give $c_n=\alpha\le 0.9607$. That is, Lemma \ref{lemma:f_n}.
    \item[2.] One could further split the regions of $s$ for which we bound $h(s-1)$ and $H(s)$. This would slightly improve our final numerics, but further complicate our ensuing arguments. We have chosen not to pursue such an optimisation as it is unlikely to have any impact on our final application to Chen's theorem.
\end{enumerate}
\end{remarks}

\subsection{Explicit version of the linear sieve}
\label{ELS}
We now list some more definitions related to the linear sieve. Let $\mathbb{P}$ be a set of primes and $g(d):\mathbb{N}\rightarrow \mathbb{C}$ a multiplicative function. For $2\le z \le D$, with $D \in \mathbb{R}^+$, we define $V(z)$ as in \eqref{vzdef} and let
\begin{equation*} \label{eq: def-y_n}
    y_n=y_n( D, p_1, \ldots, p_n):=\left( \frac{D}{p_1 \ldots p_n}\right)^{\frac{1}{2}}.
\end{equation*}
We wish to obtain an upper bound for
\begin{equation} \label{eq: def-T-n-D-z}
T_n(D,z):= \sum_{\substack{p_1 \ldots p_n \in \mathbb{P}\\y_n\le p_n < \ldots<p_1<z\\p_m<y_m\forall m<n,~m\equiv n\: (\text{mod }2)}}g(p_1\ldots p_n)V(p_n).
\end{equation} 
In particular, an upper bound on $T_n(D,z)$ is the core ingredient used to obtain the lower and upper bounds on $S(A,\mathbb{P},z)$ in Theorem \ref{theo:JR}. To estimate $T_n(D,z)$ we will utilise our bounds on $f_n(s)$, $h(s)$ and $H(s)$ obtained in the previous sections.

Our main lemma is as follows, which improves on \cite[Theorem 9.5]{Nathanson}. Compared to the work of Nathanson, we obtain a significant improvement by reducing the uniformity in the parameters, which improves the overall accuracy in the induction step.
\begin{lemma}
\label{lemma:expf_n}
Let $z\ge 2$, and $D > 0$ be real such that 
\begin{equation*} 
s:=\frac{\log D}{\log z}\ge 
\begin{cases} 
1 ~\text{if} ~ n  ~\text{is odd,} \\
2 ~\text{if}~n ~\text{is even} .
\end{cases} 
\end{equation*}
Let $\mathbb{P}$ be a set of primes and $g(d)$ be a multiplicative function such that
\begin{equation*}
0 \le g(p) < 1 ~\text{for all} ~p \in \mathbb{P}
\end{equation*} 
and
\begin{equation}
\label{eq:prod11}
\frac{V(u)}{V(z)}:=\prod_{\substack{p\in \mathbb{P}\\ u \le p < z}}(1-g(p))^{-1} \le K \frac{\log z}{\log u},
\end{equation}
for all $u$ such that $1 < u < z$ and $K$ such that
\begin{equation}
\label{eq:epsilonK}
1< K < 1+\varepsilon
\end{equation}
for some choice of $\varepsilon>0$. 
Then
\begin{equation*}
T_n(D,z)<V(z)\left( f_n(s)+\varepsilon\tau_n e^2h(s)\right),
\end{equation*}
where $\tau_1=3$ and for $n\ge 2$
\begin{equation} 
\label{eq:tau}
\tau_n:=
\begin{cases} 
    \tau_{n-1}\cdot \max\{\xi_2,\xi_{2.2},\xi_{2.4},\xi_{2.6},\xi_{2.8},\xi_3\} & \text{if}~n\ \text{even}, \\
    \tau_{n-1}\cdot\max\{\xi_3,\tilde{\xi}\} & \text{if}~n\ \text{odd},
\end{cases} 
\end{equation}
and
\begin{align*}
    \xi_{s_0}&:=\kappa_{s_0}+\left(\gamma_{s_0}+\kappa_{s_0}\right)\varepsilon+2\gamma_{s_0}\frac{(c_{n-1})^{n-2}}{\tau_{n-1}}+2\frac{(c_n)^{n-1}}{\tau_{n-1}}\\
    \tilde{\xi}&:=\tilde{\kappa}+(5+4\varepsilon)\varepsilon+6(1+\varepsilon)\frac{(c_{n-1})^{n-2}}{\tau_{n-1}}+2(2+\varepsilon)\frac{(c_n)^{n-1}}{\tau_{n-1}}
\end{align*}
with $\gamma_{s_0}$, $\kappa_{s_0}$ and $\tilde{\kappa}$ as defined in Lemmas \ref{lem:hsm1} and \ref{lemma:forg}, and $c_n$ as defined in Section \ref{introsub}. 
\end{lemma}
\begin{proof}
We start by defining 
\begin{equation}
    \label{eq:hnf}
    h_n(s):=\varepsilon\tau_n e^2h(s).
\end{equation}
We thus want to prove 
\begin{equation}
\label{eq:T_n}
    T_n(D,z)<V(z)\left( f_n(s)+h_n(s)\right).
\end{equation}
We proceed by induction on $n$. Let $n=1$. By \cite[Lemma 9.3]{Nathanson} with $\beta =2$, we have $T_1(D,z)=0$ for $s>3$. Since the right-hand side of \eqref{eq:T_n} is positive, it follows that the inequality holds for $s>3$. If $1\le s \le 3$ then $sf_1(s)=3-s$ and 
\begin{equation*}
T_1(D,z)=V(D^{1/3})-V(z),
\end{equation*}
by \cite[(9.13)]{Nathanson}. Hence, using \eqref{eq:prod11},
\begin{equation*}
\frac{T_1(D,z)}{V(z)}\le \left(\frac{3}{s}-1\right)+\frac{3}{s}(K-1)<f_1(s)+h_1(s).
\end{equation*}
This proves the lemma for $n=1$. Now, let $n\ge 2$ and assume that the lemma holds for $n-1$. We begin with the case where $s\ge 3$. Using \cite[Lemma 9.8]{Nathanson} and the induction hypothesis for $n-1$ as done in \cite[Theorem 9.5]{Nathanson}
\begin{align}\label{Tneq}
    \frac{T_n(D,z)}{V(z)}< & (K-1)(f_{n-1}(s-1)+h_{n-1}(s-1)) \\&+\frac{K}{s}\int_s^{\infty} (f_{n-1}(t-1)+h_{n-1}(t-1))\mathrm{d}t.\notag
\end{align}
We now bound each term in \eqref{Tneq} in terms of $h_{n-1}(s)$. Firstly, by \eqref{eq:h11} and \eqref{eq:fh},
\begin{align}
\label{eq:Tn3}
\begin{split}
(K-1)f_{n-1}(s-1)&< \varepsilon\cdot 2e^2(c_{n-1})^{n-2}h(s-1)\\&\le \varepsilon\cdot 2\gamma_3\cdot e^2(c_{n-1})^{n-2}h(s)=2\gamma_3\frac{(c_{n-1})^{n-2}}{\tau_{n-1}} h_{n-1}(s).\end{split}
\end{align}
Then, again by \eqref{eq:h11}
\begin{equation}\label{eq:Tn4}
    (K-1)h_{n-1}(s-1)<\varepsilon\cdot\gamma_3\cdot h_{n-1}(s).
\end{equation}
Next, by \eqref{eq:f22} we have
\begin{equation*}
\frac{K}{s}\int_s^{\infty} f_{n-1}(t-1)\mathrm{d}t=Kf_n(s).
\end{equation*}
To express this in terms of $h_{n}(s)$, we note that by \eqref{eq:fh}
\begin{equation*}
    (K-1)f_n(s)<\varepsilon\cdot 2e^2(c_{n})^{n-1}h(s)=2 \frac{(c_{n})^{n-1}}{\tau_{n-1}}h_{n-1}(s)
\end{equation*}
so
\begin{equation}\label{eq:Tn1}
Kf_n(s)< f_n(s)+2 \frac{(c_{n})^{n-1}}{\tau_{n-1}}h_{n-1}(s).
\end{equation}
Finally, by the definition of $H(s)$ and Lemma~\ref{lemma:forg},
\begin{equation*}
    \int_s^{\infty} h(t-1)\mathrm{d}t=H(s)\le \kappa_3\cdot s\cdot h(s),
\end{equation*}
and thus 
\begin{equation}
\label{eq:Tn2}
    \frac{K}{s}\int_s^{\infty} h_{n-1}(t-1)\mathrm{d}t\le \kappa_3 K h_{n-1}(s)<\kappa_3h_{n-1}(s)+\varepsilon\cdot\kappa_3 h_{n-1}(s).
\end{equation}
Combining \eqref{eq:Tn3}, \eqref{eq:Tn4}, \eqref{eq:Tn1} and \eqref{eq:Tn2} 
\begin{align*}
    \frac{T_n(D,z)}{V(z)}<f_n(s)+\xi_3\cdot h_{n-1}(s)\le f_n(s)+h_n(s),
\end{align*}
as required. The case for $n\ge 2$ even and $2\le s\le 3$ is similar. In particular, for any $s_0\in\{2,2.2,2.4,2.6,2.8\}$ repeating the above argument gives
\begin{equation*}
    \frac{T_n(D,z)}{V(z)}<f_n(s)+\xi_{s_0}\cdot h_{n-1}(s)
\end{equation*}
if $s_0\le s\le s_0+0.2$. Now, let $n\ge 3$ be odd and $1 \le s \le 3$. In this case, the recursion formula (cf.\ Equation \ref{Tneq}) is different and, following the proof of\footnote{Note that here we fix an error of Nathanson's as he incorrectly claimed $T_n(D,z)<V(z)(f_n(3)+h_n(3))$, which would contradict the optimality of the linear sieve (see e.g.\ \cite[Section 12.3]{Opera}).} \cite[Theorem 9.5]{Nathanson}, one obtains
\begin{align}
    T_n(D,z)&< (K-1)V(D^{1/3})(f_{n-1}(2)+h_{n-1}(2))\notag\\
    &\qquad\qquad+\frac{KV(D^{1/3})}{3}\int_{3}^{\infty}(f_{n-1}(t-1)+h_{n-1}(t-1))\mathrm{d}t\notag\\
    &\le \frac{3K}{s}(K-1)V(z)(f_{n-1}(2)+h_{n-1}(2))\notag\\
    &\qquad\qquad+\frac{K^2V(z)}{s}\int_{3}^{\infty}(f_{n-1}(t-1)+h_{n-1}(t-1))\mathrm{d}t,
\end{align}
where we have used that 
\begin{equation*}
    V(D^{1/3})\le\frac{3K}{s}V(z)
\end{equation*}
by \eqref{eq:prod11}. One now argues similarly as before, using \eqref{eq:h12} and \eqref{eq: H(3)-upper} to deduce that
\begin{align*}
    \frac{3K}{s}(K-1)f_{n-1}(2)&<6(1+\varepsilon)\frac{(c_n)^{n-2}}{\tau_{n-1}}h_{n-1}(s),\\
    \frac{3K}{s}(K-1)h_{n-1}(2)&<3\varepsilon(1+\varepsilon)h_{n-1}(s),\\
    \frac{K^2}{s}\int_{3}^\infty f_{n-1}(t-1)\mathrm{d}t&<f_n(s)+2(2+\varepsilon)\frac{(c_n)^{n-1}}{\tau_{n-1}}h_{n-1}(s),\\
    \frac{K^2}{s}\int_{3}^\infty h_{n-1}(t-1)\mathrm{d}t&<\tilde{\kappa}h_{n-1}(s)+\varepsilon(2+\varepsilon)\tilde{\kappa}h_{n-1}(s)
\end{align*}
and thus
\begin{equation*}
    \frac{T_{n}(D,z)}{V(z)}<f_n(s)+\tilde{\xi}h_{n-1}(s).
\end{equation*}
This completes the proof.
\end{proof}

As an example, in Table \ref{tab:tau} we report upper bounds for $\tau_n$ for $\varepsilon=1/200$ and $n\le 500$, obtained by Lemma~\ref{lemma:fsmall}. In particular, the upper bound for $\tau_n$ is computed recursively using \eqref{eq:tau} with $c_1=1$ and $c_n$ as in Table \ref{tab:cn1} for $2\leq n\leq 500$.

\begin{table}[H]
\centering
    \begin{tabular}{ | l | l |}
    \hline
    $n$ & $\tau_n $ \\  
    \hline
    $1$& $3$ \\
    \hline
    $2$& $8$ \\
    \hline
    $3-10$& $10$ \\
    \hline
    $11-13$& $9$ \\
    \hline
    $14-16$& $8$ \\
    \hline
    \end{tabular}
    \begin{tabular}{ | l | l | }
    \hline
    $n$ & $\tau_n $\\  
    \hline
    $17-20$& $7$ \\
    \hline
    $21-24$ & $6$ \\
    \hline
    $25-30$ & $5$ \\
    \hline
    $31-37$ & $4$ \\
    \hline
    $38-46$ & $3$ \\
    \hline
    \end{tabular}
    \begin{tabular}{ | l | l |}
    \hline
    $n$ & $\tau_n $ \\
    \hline
    $47-63$ & $2$ \\
    \hline
    $64-118$ & $1$ \\
    \hline
    $119-173$ & $10^{-1}$ \\
    \hline
    $174-228$ & $10^{-2}$ \\
    \hline
    $229-283$ & $10^{-3}$ \\
    \hline
    \end{tabular}
    \begin{tabular}{ | l | l |}
    \hline
    $n$ & $\tau_n $ \\  
    \hline
    $284-338$ & $10^{-4}$ \\
    \hline
    $339-393$ & $10^{-5}$ \\
    \hline
    $394-448$& $10^{-6}$ \\
    \hline
    $449-500$ &$10^{-7}$\\
    \hline
     &\\
    \hline
    \end{tabular}
\caption{Upper bound for $\tau_n$ for $\varepsilon=1/200$.}  
\label{tab:tau}  
\end{table}

Before continuing, we also provide an upper bound for $\tau_n$ that will be easier to work with when $n$ is large.
\begin{lemma}\label{lem: taup}
    Keep the notation of Lemma \ref{lemma:expf_n}. Let $\tau_n'$ be such that $\tau'_1=3$ and for $n\ge 2$
    \begin{equation} 
    \label{eq:taup}
        \tau'_n:=\tau'_{n-1}\left(\kappa_2+(\gamma_2+\kappa_2)\varepsilon+\frac{8e}{3}\frac{(c_{n-1})^{n-2}}{\tau'_{n-1}}+2(2+\varepsilon)\frac{(c_n)^{n-1}}{\tau'_{n-1}}\right) 
    \end{equation}
    Then, we have $\tau_n\le\tau'_n$ whenever $0<\varepsilon\le 1/74$.
\end{lemma}
\begin{proof}
    We proceed by induction. First note that $\tau_1=\tau'_1=3$ so that the result holds for $n=1$. Now suppose that $n\ge 2$ and $\tau_{n-1}\le\tau'_{n-1}$. To begin with, we note that
    \begin{equation*}
        \kappa_{s_0}+(\gamma_{s_0}+\kappa_{s_0})\varepsilon
    \end{equation*}
    is a linear function of $\varepsilon$. Thus, through elementary analysis one finds
    \begin{equation*}
        \max_{s_0\in\{2,\:2.2,\:2.4,\:2.6,\:2.8,\:3\}}\left\{\kappa_{s_0}+(\gamma_{s_0}+\kappa_{s_0})\varepsilon\right\}\le \kappa_2+(\gamma_2+\kappa_2)\varepsilon
    \end{equation*}
    provided $\varepsilon\le 1/54$. Also, since $\gamma_{s_0}\le\gamma_3=4e/3$ and $2\le 2(2+\varepsilon)$, one then has
    \begin{align}
        \tau_{n-1}\xi_{s_0}&=\tau_{n-1}\left(\kappa_{s_0}+\left(\gamma_{s_0}+\kappa_{s_0}\right)\varepsilon+2\gamma_{s_0}\frac{(c_{n-1})^{n-2}}{\tau_{n-1}}+2\frac{(c_n)^{n-1}}{\tau_{n-1}}\right)\notag\\
        &\le \tau'_{n-1}\left(\kappa_{2}+\left(\gamma_{2}+\kappa_{2}\right)\varepsilon+\frac{8e}{3}\frac{(c_{n-1})^{n-2}}{\tau'_{n-1}}+2(2+\varepsilon)\frac{(c_n)^{n-1}}{\tau'_{n-1}}\right)=\tau'_n\label{xis0new}
    \end{align}
    for all $s_0\in\{2, 2.2, 2.4, 2.6, 2.8, 3\}$. Similarly,
    \begin{equation*}
        \tilde{\kappa}+(5+4\varepsilon)\varepsilon\le\kappa_2+(\gamma_2+\kappa_2)\varepsilon
    \end{equation*}
    provided $\varepsilon\le 1/74$ so that
    \begin{align}
        \tau_{n-1}\tilde{\xi}&=\tau_{n-1}\left(\tilde{\kappa}+(5+4\varepsilon)\varepsilon+6(1+\varepsilon)\frac{(c_{n-1})^{n-2}}{\tau_{n-1}}+2(2+\varepsilon)\frac{(c_n)^{n-1}}{\tau_{n-1}}\right)\notag\\
        &\le\tau'_{n-1}\left(\kappa_2+(\gamma_2+\kappa_2)\varepsilon+\frac{8e}{3}\frac{(c_{n-1})^{n-2}}{\tau'_{n-1}}+2(2+\varepsilon)\frac{(c_n)^{n-1}}{\tau'_{n-1}}\right)=\tau'_n\label{xitnew}.
    \end{align}
    From \eqref{xis0new} and \eqref{xitnew}, it follows that $\tau_n\le \tau'_n$ as required.
\end{proof}

We can now effectively bound the upper and lower bound sieves constructed in \cite[Theorem 9.3]{Nathanson} for $S(A,\mathbb{P}, z)$ and improve on \cite[Theorem 9.6]{Nathanson}.
\begin{prop}
\label{theo:G}
Let $z$, $D$, $s$, $\mathbb{P}$, $g(d)$ and $\varepsilon$ satisfy the hypotheses of Lemma~\ref{lemma:expf_n}. Let
\begin{equation*}
G(z, \lambda^{\pm}):=\sum_{d|P(z)}\lambda^{\pm}(d)g(d),
\end{equation*}
with $\lambda^{\pm}(d)$ the upper and lower bound sieves for $S(A,\mathbb{P}, z)$ constructed in \cite[Theorem 9.3]{Nathanson}.
Then
\begin{equation*}
G(z, \lambda^{+})< V(z)\left(F(s)+\varepsilon e^2 h(s) \sum_{n=1}^{\infty}\tau_{2n-1} \right)
\end{equation*}
and 
\begin{equation*}
G(z, \lambda^{-})> V(z)\left(f(s)-\varepsilon e^2 h(s) \sum_{n=1}^{\infty}\tau_{2n} \right),
\end{equation*}
where $F(s)$ and $f(s)$ are defined in \eqref{def-f(s)-F(s)}, $h(s)$ is defined in \eqref{eq-def-h(s)} and $\tau_n$ is defined in \eqref{eq:tau}.
\end{prop}
\begin{proof}
    By \cite[Lemma 9.3]{Nathanson}, we have
    \begin{equation*}
        G(z,\lambda^+)=V(z)+\sum_{\substack{n=1\\ n\equiv 1\ \text{(mod $2$)}}}^\infty T_n(D,z)
    \end{equation*}
    and 
    \begin{equation*}
        G(z,\lambda^-)=V(z)-\sum_{\substack{n=1\\ n\equiv 0\ \text{(mod $2$)}}}^\infty T_n(D,z).
    \end{equation*}
The proof then follows upon applying our upper bound for $T_n(D,z)$ in Lemma~\ref{lemma:expf_n}, and the definitions \eqref{bigFdeffn} and \eqref{lilFdeffn} of $F(s)$ and $f(s)$ in terms of $f_n(s)$.
\end{proof}

 We now obtain a bridging result which allows us to make Proposition \ref{theo:G} explicit.

\begin{lemma}\label{lemma:exprec}
    Let $\tau_n$ and $\tau'_n$ be as defined in \eqref{eq:tau} and \eqref{eq:taup} respectively. For some choice of $\varepsilon\in(0,1/74]$ and any $k_e,k_o\ge 1$, we have
    \begin{equation}\label{C1sum}
        C_1(\varepsilon):=\sum_{n=1}^{\infty}\tau_{2n-1} \le\sum_{n=1}^{k_o}\tau_{2n-1}+\tau'_{2k_o}\sum_{n=1}^\infty J(2k_o+1)^{2n-1}
    \end{equation}
    and
    \begin{equation}\label{C2sum}
        C_2(\varepsilon):=\sum_{n=1}^{\infty}\tau_{2n} \le\sum_{n=1}^{k_e}\tau_{2n}+\tau'_{2k_e+1}\sum_{n=1}^\infty J(2k_e+2)^{2n-1},
    \end{equation}
    where
    \begin{equation*}
        J(k):=\kappa_2+\left(\gamma_2+\kappa_2\right)\varepsilon+\frac{8e}{6}\left(\frac{\kappa_2}{\kappa_2+\left(\gamma_2+\kappa_2\right)\varepsilon}\right)^{k-2}+\frac{2(2+\varepsilon)}{3}\frac{\kappa_2^{k-1}}{(\kappa_2+\left(\gamma_2+\kappa_2\right)\varepsilon)^{k-2}}.
    \end{equation*}
\end{lemma}
\begin{proof}
    First we note that since $\kappa_2+(\gamma_2+\kappa_2)\varepsilon<1$ for $\varepsilon<1/74$, and
    \begin{equation*}
        \frac{\kappa_2}{\kappa_2+(\gamma_2+\kappa_2)\varepsilon}<1,
    \end{equation*}
    it follows that $|J(k)|<1$ for sufficiently large $k$ and the infinite sums in \eqref{C1sum} and \eqref{C2sum} converge.
    
    Now, we will only prove the inequality in \eqref{C1sum} since \eqref{C2sum} follows in an identical fashion. So, to begin with, we use Lemma \ref{lem: taup} to obtain
    \begin{equation}\label{tau2eq}
        \sum_{n=1}^\infty\tau_{2n-1}=\sum_{n=1}^{k_o}\tau_{2n-1}+\sum_{n=k_o+1}^\infty\tau_{2n-1}\le \sum_{n=1}^{k_o}\tau_{2n-1}+\sum_{n=k_o+1}^\infty\tau'_{2n-1}.
    \end{equation}
    Next we note that $c_n\le\kappa_2=0.9607$ (Lemma \ref{lemma:f_n}) and by the definition \eqref{eq:taup} of $\tau'_n$, we have
    \begin{equation*}
        \tau'_n\ge 3\cdot(\kappa_2+(\gamma_2+\kappa_2)\varepsilon)^{n-1}.
    \end{equation*}
    Hence, again by \eqref{eq:taup},
    \begin{equation*}
        \frac{\tau'_n}{\tau'_{n-1}}\le J(n)
    \end{equation*}
    for all $n\ge 2$. This means that
    \begin{equation*}
        \tau'_{2k_o+1}=\tau'_{2k_o}\left(\frac{\tau'_{2k_o+1}}{\tau'_{2k_o}}\right)\le\tau'_{2k_o}J(2k_o+1),    
    \end{equation*}
    and, since $J(k)$ is decreasing in $k$, we have by induction
    \begin{equation*}
        \tau'_{2k_o+2n-1}\le\tau'_{2k_o}J(2k_o+1)^{2n-1}.
    \end{equation*}
    Substituting this into \eqref{tau2eq} then gives \eqref{C1sum} as required.
\end{proof}
We can now conclude the proof of Theorem \ref{theo:JR} using the above machinery.
\begin{proof}[Proof of Theorem \ref{theo:JR}]
    The proof of Theorem \ref{theo:JR} is the same as \cite[Theorem 9.7]{Nathanson} but with our bounds for $C_1(\varepsilon)$ and $C_2(\varepsilon)$ from Lemma \ref{lemma:exprec}. To obtain the values in Table \ref{tab:fF} we first choose $k_e=250$ and $k_o=249$ and then use Lemmas \ref{lemma:fsmall}, \ref{lemma:f_n}, \ref{lemma:expf_n} and \ref{lem: taup} to iteratively compute $\tau_n$ for $n\le 500$ and $\tau'_n$ for $n\le 501$. Finally, to bound the infinite series in \eqref{C1sum} and \eqref{C2sum}, we first compute $J(2k_e+2)$ and $J(2k_o+1)$ and then evaluate the sum as a geometric series.
\end{proof}
\begin{remark}
    The restriction $\varepsilon\leq 1/74$ in Theorem \ref{theo:JR} is required to prove the intermediary result Lemma \ref{lem: taup}. It appears difficult to significantly weaken this restriction, especially since we also require 
    \begin{equation}\label{kappagameq}
        \kappa_2+(\gamma_2+\kappa_2)\varepsilon<1
    \end{equation}
    in the proof of Lemma \ref{lemma:exprec}. In particular, with $\gamma_2=e^{0.2}$ (Lemma \ref{lem:hsm1}) and $\kappa_2=0.9607$ (Lemma \ref{lemma:forg}), the inequality \eqref{kappagameq} necessitates that $\varepsilon<1/55$.
\end{remark}

\subsection{Explicit bounds for $\varepsilon$ in \eqref{eq:cond1}} \label{subsec: expl-epsilon}
For our applications of Theorem \ref{theo:JR} we will choose $g_n$ such that $g_n(p)=\frac{1}{p-1}$. With this choice, we now prove a series of lemmas that will be used to give a value of $\varepsilon$ in \eqref{eq:cond1}.
\begin{lemma}\label{betrandlem}
    For all $x\ge \exp(20)$, there exists a prime in the interval $[0.999x,x)$.
\end{lemma}
\begin{proof}
    For $\exp(20)\le x\le 4\cdot 10^{18}$ we use the results on gaps between primes in \cite[Table 8]{O_H_P_14}. For $x>4\cdot 10^{18}$, we use \cite[Table 2]{Kadiri2}.
\end{proof}

\begin{lemma}\label{mertencomplem}
    For all $2\le x\le 10^{12}$, we have
    \begin{equation*}
        \log\log x+M<\sum_{p\le x}\frac{1}{p}<\log\log x+M+\frac{2}{\sqrt{x}\log x},
    \end{equation*}
    with
    \begin{equation}
    \label{eq:MM}
        M:=\lim_{x \rightarrow \infty} \left(\sum_{p\le x}\frac{1}{p}-\log\log x\right)=0.26149 72128 47643\ldots.
    \end{equation}
\end{lemma}
\begin{proof}
    The proof of the lemma is by direct computation. The computation took just over 15 hours on an Intel Core i7 3.00GHz processor. To begin with, we used the primesieve package in Python to compute all the primes up to $10^8$, and used these primes to directly verify the lemma up to $x=10^8$. This process was then repeated for all primes $p$ satisfying $10^8<p\le 2\cdot 10^8$, and then similarly for intervals of length $10^8$ until we covered all primes up to $10^{12}$. Note that were unable to store all primes up to $10^{12}$ in one go due to the limited memory on our computer.
\end{proof}
\begin{remark}
    Lemma \ref{mertencomplem} extends a computation due to Rosser and Schoenfeld \cite[Theorem 20]{R-S} by a factor of $10^4$.
\end{remark}

\begin{lemma}\label{mertenlem}
    For all $x\ge 2$, with $M$ defined in \eqref{eq:MM}, we have
    \begin{equation}\label{small1p}
        \sum_{p\le x}\frac{1}{p}\ge\log\log x+M-\frac{2.964\cdot 10^{-6}}{\log x},
    \end{equation}
    and for all $x>\exp(4000)$, we have
    \begin{equation}\label{big1p}
        \sum_{p\le x}\frac{1}{p}\le\log\log x+M+\frac{1.436\cdot 10^{-16}}{\log x}.
    \end{equation}
\end{lemma}
\begin{proof}
    By \cite[(4.20)]{R-S} we have
    \begin{equation}\label{1pinteq}
        \sum_{p\le x}\frac{1}{p}=\log\log x+M+\frac{\theta(x)-x}{x\log x}+\int_x^\infty\frac{(y-\theta(y))(1+\log y)}{y^2\log^2y}\mathrm{d}y,
    \end{equation}
    where 
    \begin{equation}
    \label{eq:theta}
      \theta(x)=\sum_{p\le x}\log p,  
    \end{equation}
    is Chebyshev's theta function. To obtain \eqref{big1p}, we simply substitute into \eqref{1pinteq} the bound for $M_1$ in \cite[Table 15]{Broad}, such that
    $$ \theta(x)-x \le \frac{M_1 x}{\log x},$$
    corresponding to $x>\exp(4000)$, namely $M_1=5.7410\cdot 10^{-13}$. To prove \eqref{small1p} for $2\le x\le 10^{12}$ we use Lemma \ref{mertencomplem}. To prove \eqref{small1p} for $x>10^{12}$ we take much more care. 
    
    Firstly, by \cite[Table 15]{Broad}, the first error term in \eqref{1pinteq} can be bounded by
    \begin{equation*}
        \left|\frac{\theta(x)-x}{x\log x}\right|\le\frac{6.9322\cdot 10^{-5}}{\log^2x}.
    \end{equation*}
We now split into the cases $10^{12}<x\le 10^{19}$ and $x>10^{19}$. In the first case, we have
\begin{align*}
    \int_x^\infty\frac{(y-\theta(y))(1+\log y)}{y^2\log^2y}\mathrm{d}y&=\int_x^{10^{19}}\frac{(y-\theta(y))(1+\log y)}{y^2\log^2y}\mathrm{d}y\\
    &\qquad\qquad+\int_{10^{19}}^\infty\frac{(y-\theta(y))(1+\log y)}{y^2\log^2y}\mathrm{d}y.
\end{align*}
For the first integral, we use \cite[Theorem 2]{Buthe3} to obtain
\begin{align*}
    \int_x^{10^{19}}\frac{(y-\theta(y))(1+\log y)}{y^2\log^2y}\mathrm{d}y&\ge\int_x^{10^{19}}\frac{0.05(1+\log y)}{y^{3/2}\log^2y}\mathrm{d}y\\
    &=0.05\left[\frac{1}{2}~\text{li}\left(\frac{1}{\sqrt{y}}\right)-\frac{1}{\sqrt{y}\log y}\right]_x^{10^{19}}\\
    &\ge\frac{0.05}{\sqrt{x}\log x}-0.025~\text{li}\left(\frac{1}{\sqrt{x}}\right)-7.077\cdot 10^{-13},
\end{align*}
where
\begin{equation} \label{eq: def-li(x)}
    \text{li}(x) = \int_0^x \frac{\mathrm{d}t}{\log t},
\end{equation}
is the logarithmic integral function.
For the second integral we again use \cite[Table 15]{Broad} to obtain
\begin{equation*}
    \left|\int_{10^{19}}^\infty\frac{(y-\theta(y))(1+\log y)}{y^2\log^2y}\mathrm{d}y\right|\le 8.6315\cdot 10^{-7}\left(\frac{1}{2\log^2(10^{19})}+\frac{1}{\log(10^{19})}\right).
\end{equation*}
Thus, for $10^{12}<x\le 10^{19}$, we have an error term bounded below by
\begin{align*}
    &-\frac{6.9322\cdot 10^{-5}}{\log^2x}+\frac{0.05}{\sqrt{x}\log x}-0.025~\text{li}\left(\frac{1}{\sqrt{x}}\right)-7.077\cdot 10^{-13}\\
    &\qquad\qquad\qquad-8.6315\cdot 10^{-7}\left(\frac{1}{2\log^2(10^{19})}+\frac{1}{\log(10^{19})}\right)\\
    &\quad\ge-\frac{2.964\cdot 10^{-6}}{\log x}.
\end{align*}
For $x>10^{19}$ an even sharper bound is obtained by simply substituting into \eqref{1pinteq} the entry for $M_1$ in \cite[Table 15]{Broad} corresponding to $10^{19}$.
\end{proof}
We are now able to obtain explicit bounds for $\varepsilon$ in \eqref{eq:cond1}.
\begin{lemma} \label{lem: prod-u>u_0} Let $z>\exp(4000)$ and $u_0=10^9$. Then for all $u_0<u<z$, we have
    \begin{equation}\label{epsmainbound}
        \prod_{u\le p<z}\left(1-\frac{1}{p-1}\right)^{-1}<\left(1+1.452\cdot 10^{-7}\right)\frac{\log z}{\log u}.
    \end{equation}
\end{lemma}
\begin{proof}
    We first note that
    \begin{equation*}
        \prod_{u\le p<z}\left(1-\frac{1}{p-1}\right)^{-1}=\prod_{u\le p<z}\left(\frac{(p-1)^2}{p(p-2)}\right)\prod_{u\le p<z}\left(1-\frac{1}{p}\right)^{-1}.
    \end{equation*}
    By Lemma \ref{betrandlem}, noting that $\exp(20)<10^9$, we then have
    \begin{align*}
        \prod_{u\le p<z}\left(\frac{(p-1)^2}{p(p-2)}\right)&=\prod_{u\le p<z}\left(1+\frac{1}{p(p-2)}\right)\\
        &\le\prod_{0.999u\le p<z}\left(1+\frac{1}{p^2}\right)\\
        &\le 1+\sum_{n\ge 0.999u}\frac{1}{n^2}\le 1+\frac{1}{0.999u-1}.
    \end{align*}
    Thus,
    \begin{equation}\label{prodeq1}
        \prod_{u\le p<z}\left(1-\frac{1}{p-1}\right)^{-1}<\left(1+\frac{1}{0.999u-1}\right)\prod_{u\le p<z}\left(1-\frac{1}{p}\right)^{-1}.
    \end{equation}
    Next, we note that
    \begin{equation}\label{prodtoexp}
        \prod_{u\le p<z}\left(1-\frac{1}{p}\right)^{-1}=\exp\left(-\sum_{u\le p<z}\log\left(1-\frac{1}{p}\right)\right).
    \end{equation}
    Now, by Lemma \ref{mertenlem}
    \begin{align}\label{merteneq}
        \sum_{u\le p< z}\frac{1}{p}&=\sum_{p<z}\frac{1}{p}-\sum_{p<u}\frac{1}{p}\le\log\log z-\log\log u+\frac{1.436\cdot 10^{-16}}{\log z}+\frac{2.964\cdot 10^{-6}}{\log u}\\
        &=\sum_{p<z}\frac{1}{p}-\sum_{p<u}\frac{1}{p}\le\log\log z-\log\log u+1.431\cdot 10^{-7}
    \end{align}
    since $z>\exp(4000)$ and $u>u_0=10^9$. Hence, using \eqref{prodtoexp}, \eqref{merteneq} and that for $x\in(0,1/2]$,
    \begin{equation*}
        \log(1-x)\ge-x-x^2,\quad e^x\le 1+x+x^2,
    \end{equation*}
    we have,
    \begin{align}
        \prod_{u\le p<z}\left(1-\frac{1}{p}\right)^{-1}&\le\frac{\log z}{\log u}\exp\left(1.431\cdot 10^{-7}\right)\exp\left(\sum_{p\ge u}\frac{1}{p^2}\right)\notag\\
        &\le\frac{\log z}{\log u}\exp\left(1.431\cdot 10^{-7}\right)\exp\left(\sum_{n\ge u}\frac{1}{n^2}\right)\notag\\
        &\le\frac{\log z}{\log u}\exp\left(1.431\cdot 10^{-7}\right)\left(1+\frac{1}{u-1}+\frac{1}{(u-1)^2}\right).\label{prodeq2}
    \end{align}
    Using \eqref{prodeq1}, \eqref{prodeq2}, $u>u_0=10^{9}$, and merging to the term $\frac{\log z}{\log u}$, gives the desired result.
\end{proof}
Note that the range $u$ was chosen to be near optimal for the final computations in \S \ref{section:B}. It is also worth noting that the above result is one of the key numerical ingredients in the proof of Theorem~\ref{Theo:B.}. In particular, sharpening the bound \eqref{epsmainbound} is necessary if one wishes to obtain a substantial improvement to the range of $N$ in Theorem~\ref{Theo:B.}.

\section{Some useful lemmas}
\label{section:sur}
In this section, we introduce some general explicit results from the literature that will be highly useful. Here and throughout the rest of the paper, $\pi(x)$, $\theta(x)$ and $\psi(x)$ (and their generalisations) denote the standard prime counting functions, $\mu(n)$ denotes the Möbius function and $\varphi(n)$ denotes the Euler totient function. 

We begin by giving some explicit expressions for the sieving functions $f(s)$ and $F(s)$ defined in \eqref{def-f(s)-F(s)}. For a particular range of $s$, one can do this inductively using the definition of these functions. For our purposes, we will only need $0<s\le 4$ so the following lemma is restricted to this case only.
\begin{lemma}[{\cite[Lemma 2]{Cai1}}] \label{lem: estimetes-f-F}
    Let $f$ and $F$ be as defined in \eqref{def-f(s)-F(s)}. Then,
    \begin{align}
        F(s)&=\frac{2e^{\gamma}}{s}, ~&0 < s \le 3,\label{eq: approx-F03}\\
        F(s)&=\frac{2e^{\gamma}}{s} \left( 1 + \int_2^{s-1} \frac{\log(t-1)}{t} \mathrm{d}t \right), ~&3 \le s \le 4,\label{eq: approx-F34}\\
        f(s)&= 0, ~&0 < s \le 2,\label{eq: approx-f02}\\
        f(s)&=\frac{2e^{\gamma} \log(s-1)}{s}, ~&2 \le s \le 4.\label{eq: approx-f24}
    \end{align}
\end{lemma}
Next we give a general result which one can use to bound sums over arithmetic functions.
\begin{lemma}[{\cite[Lemma 1 (ii)]{Greaves}}]
\label{lemma:sump}
Let $f(t)$ be a positive, monotone function defined for $w \le t \le z$ with $f'(t)$ piecewise continuous on $[w,z]$, and $c(n)$ be an arithmetic function satisfying 
\begin{equation*}
\sum_{x\le n <y}c(n) \le g(y)-g(x)+E,
\end{equation*}
for some constant $E$ whenever $w \le x < y \leq z$. Then,
\begin{equation*}
\sum_{w\le n <z}c(n)f(n)\le \int_w^z f(t)g'(t)\mathrm{d}t+E\max \left(f(w),f(z)\right).
\end{equation*}
\end{lemma}
Finally, we give some explicit bounds on functions relating to primes and prime factors.
\begin{lemma}[See {\cite[Theorem 5]{R-S}}]\label{1plem}
    For any $a>1$ and $b\ge 286$,
    \begin{equation}
    \label{eq:s1/p}
        \sum_{a \le p \le b}\frac{1}{p}<\log \log b -\log \log a+\frac{1}{\log^2 a}.
    \end{equation}
\end{lemma}
\begin{lemma}\label{omegalem}
    Let $\omega(n)$ count the number of unique prime divisors of $n\ge 3$. We have
    \begin{equation}\label{wntriv}
        \omega(n)\le\frac{\log n}{\log 2}
    \end{equation}
    and
    \begin{equation}
    \label{eq:w(n)}
        \omega(n)< \frac{1.3841\log n}{\log \log n}.
    \end{equation}
\end{lemma}
\begin{proof}
    The first bound \eqref{wntriv} follows by noting that each prime factor of $n$ is greater than or equal to 2. The second bound \eqref{eq:w(n)} is \cite[Theorem 11]{Guy}.
\end{proof}
Certainly, the bound \eqref{eq:w(n)} is stronger asymptotically than \eqref{wntriv} and in fact stronger explicitly when $n\ge 14$. However, in some cases where the error terms we are working with are insignificant, we will use the simpler bound \eqref{wntriv} to improve readability. 
\begin{lemma}\label{squarefreelem}
    For all $x\ge 45$, we have
    \begin{equation}\label{squarefreeeq}
        \sum_{n\le x}\mu^2(n)\le 0.65x.
    \end{equation}
\end{lemma}
\begin{proof}
    For $x\ge 10^5$, the result follows by \cite[(4.6)]{Buthe1}. For smaller values of $x$, the result can be verified via a simple computation.
\end{proof}
\begin{lemma}[{\cite[Lemma 4.5]{Buthe1}}]\label{muphilem}
    For all $x\ge 10^9$, we have
    \begin{equation*}
        \sum_{n\le x}\frac{\mu^2(n)}{\varphi(n)}\le\log x + B_{\mu, \varphi} +\frac{58}{\sqrt{x}}\le 1.1\log x,
    \end{equation*}
    where $B_{\mu, \varphi}=1.332\ldots$ is a constant.
\end{lemma}
\begin{lemma}\label{thetauplem}
    For all $x\ge 2$ and $\theta$ defined in \eqref{eq:theta}, we have
    \begin{equation*}
        \theta(x)<x\left(1+\frac{9\cdot 10^{-7}}{\log x}\right).
    \end{equation*}
\end{lemma}
\begin{proof}
    For $2\le x< 10^{19}$, the result follows by \cite[Theorem 2]{Buthe3}. For $x>10^{19}$ the result follows by \cite[Table 15]{Broad}.
\end{proof}
While most of the results above are not optimal, they are sufficient for our purposes.

\section{Results on primes in arithmetic progressions}
\label{sec:PNT}
As is commonplace in applications of the linear sieve, the remainder terms $r(d)$ (defined in \eqref{eq:Adrd}) essentially correspond to the error term for the prime number theorem in arithmetic progressions. Thus in this section, we will obtain estimates for
\begin{equation} \label{eq: E_f(x;k,l)}
    E_f(x;k,l):=f(x;k,l)-\frac{f(x)}{\varphi(k)},
\end{equation}
for $f=\pi, \theta, \psi$, the standard prime counting functions, where $f(x;k,l)$ means that the counting function $f(x)$ is restricted to $n \equiv l\: (\text{mod }k)$ for $n \le x$. In particular, for our application we are interested in averaged estimates for $E_\pi(N,d,N)$ with $N$ a large integer and $d$ square-free. Such estimates will allow us to bound the total remainder $R$ (see \eqref{Rdef}) in the linear sieve. 

We start with a result on the zeroes of Dirichlet $L$-functions, namely \cite[Theorem 1.1 \& 1.3]{Kadiri}.
\begin{theorem}[Kadiri]
\label{theo:kadiri}
Define $\prod (s,q)=\prod_{\chi\: (\text{mod }q)} L(s,\chi)$, where the product is over Dirichlet characters $\chi\: (\text{mod }q)$, $R_0=6.3970$ and $R_1=2.0452$. Then the function $\prod(s,q)$ has at most one zero $\rho=\beta+i \gamma$, in the region $\beta \ge 1-1/\left(R_0\log \max \left( q,q\left|\gamma\right|\right)\right)$. Such a zero is called a Siegel zero and if it exists, then it must be real, simple and correspond to a non-principal real character $\chi\: (\text{mod }q)$. Moreover, for any given $Q_1$, among all the zeroes of primitive characters with modulus $q \le Q_1$ there is at most one zero with $\beta \ge 1-1/2R_1 \log Q_1$, we will call this zero and the related modulus exceptional.
\end{theorem}

We now introduce a bound on $E_{\psi}(x;k,l)$ that is a specific case of \cite[Theorem 1.2, Table 6]{Bordignon4}. Namely, using the notation of \cite{Bordignon4}, we set $Y_0=10.4$, $\alpha_1=10$ and $\alpha_2=8$. We will also use the following notation:
\begin{equation} \label{eq: f(x,chi)}
    \quad \pi(x, \chi) := \sum_{p \le x} \chi(p),\quad \psi(x, \chi) := \sum_{n \le x} \Lambda(n) \chi(n), \quad \theta(x, \chi) := \sum_{p \le x} \chi(n) \log p, 
\end{equation}
where $\chi$ denotes a Dirichlet character.

\begin{lemma}
\label{lemma:PNTPAP1}
Let $E_\psi(x;k,l)$ and $\psi(x,\chi)$ be as in \eqref{eq: E_f(x;k,l)} and \eqref{eq: f(x,chi)} respectively. Let ${x\ge\exp (\exp (10.4))}$ and $k<\log^{10} x$ be an integer. Let $\textnormal{Ind}_k=1$ if $\beta_k$, the Siegel zero modulo $k$, exists and $\textnormal{Ind}_k=0$ otherwise. Then,
\begin{equation*}
\label{eq:psi1}
\frac{\varphi(k)}{x}\left|E_{\psi}(x;k,l)\right|< \frac{3.2\cdot 10^{-8}}{\log^{8} x}+\textnormal{Ind}_k \frac{x^{\beta_k-1}}{\beta_k}
\end{equation*}
and 
\begin{equation}
\label{eq:psi1.1}
-1+x^{-1}\sum_{\chi\: (\text{mod }k)} \left|\psi(x,\chi)\right|< \frac{3.2\cdot 10^{-8}}{\log^{8} x}+\textnormal{Ind}_k \frac{x^{\beta_k-1}}{\beta_k}.
\end{equation}
\end{lemma}
Importantly, the -1 appearing in \eqref{eq:psi1.1} appears when bounding the contribution from the principal character in the proof of \cite[Theorem 1.2]{Bordignon4}. Thus, we also have the following variant of Lemma \ref{lemma:PNTPAP1}.
\begin{lemma}\label{notrivialPNTAP}
    Keep the notation and conditions of Lemma \ref{lemma:PNTPAP1} and let $\chi_0$ denote the trivial character modulo $k$. We have
    \begin{equation*}
        x^{-1}\sum_{\substack{\chi\: (\text{mod }k) \\ \chi \ne \chi_0}} \left|\psi(x,\chi)\right|< \frac{3.2\cdot 10^{-8}}{\log^{8} x}+\textnormal{Ind}_k \frac{x^{\beta_k-1}}{\beta_k}.
    \end{equation*}
\end{lemma}
We now introduce a function $x_2(x)=x/\log^{15}x<x$. In doing so, we can apply a partial summation argument and obtain a sufficiently strong analogue of Lemma \ref{notrivialPNTAP} for the sum over $|\pi(x,\chi)|$. 
\begin{lemma}
\label{cor:chisum}
Let $\pi(x,\chi)$ be as in \eqref{eq: f(x,chi)}.
Let $x_2(x)=x/\log^{15} x$. Also assume $x> X_1$, with $X_1$ such that $\log\log x_2(X_1)\ge 10.4$. We then have, for $k<\log^{10}(x_2(x))$ 
\begin{equation}\label{pichieq}
    \sum_{\substack{\chi\: (\text{mod }k) \\ \chi \ne \chi_0}} \left|\pi(x,\chi)\right|< \frac{v_k(X_1)x}{\log^5 x},
\end{equation}
where
\begin{equation}\label{veq}
v_k(X_1):=v'_k(X_1)\left(1+\frac{1}{\log^{10}(X_1)\log^5 x_2(X_1)}+\frac{1}{\left(1-\frac{6}{\log x_2(X_1)}\right)\log X_1}\right)+\frac{3}{\log(X_1)},
\end{equation}
and 
\begin{align}
\label{eq:v1}
\begin{split}
v'_k(X_1):=&\max_{y\ge x_2(X_1)}\left[\frac{3.2\cdot 10^{-8}}{\log^{4} y}+  \log^{4} y\left(\textnormal{Ind}_k\frac{y^{\beta_k-1}}{\beta_k}+\frac{1.02\log^{10}y}{\sqrt{y}}+\frac{3\log^{10}y}{y^{2/3}}\right)\right].
\end{split}
\end{align}
with $\textnormal{Ind}_k$ and $\beta_k$ as in Lemma \ref{lemma:PNTPAP1}.
\end{lemma}
\begin{proof}
Let $y\in [x_2(x),x]$. By Lemma \ref{notrivialPNTAP}, we have
\begin{equation*}
    \sum_{\substack{\chi\: (\text{mod }k) \\ \chi \ne \chi_0}} \left|\psi(y,\chi)\right|<\frac{3.2\cdot 10^{-8}y}{\log^{8} y}+\textnormal{Ind}_k \frac{y^{\beta_k}}{\beta_k}.
\end{equation*}
We then use the estimate $\left|\psi(y, \chi)-\theta(y, \chi)\right|\le \psi(y)-\theta(y)<1.02y^{1/2}+3y^{1/3}$ (\cite[(3.39)]{R-S}) to obtain 
\begin{equation*}
    \sum_{\substack{\chi\: (\text{mod }k) \\ \chi \ne \chi_0}}\left|\theta(y,\chi)\right|<  \frac{v'_k(X_1)y}{\log^{4} y}.
\end{equation*}
Finally, by partial summation,
\begin{equation*}
\pi(x, \chi)=\pi(x_2(x), \chi)+\frac{\theta(x, \chi)}{\log x}-\frac{\theta(x_2(x), \chi)}{\log x_2(x)}+\int_{x_2(x)}^{x} \frac{\theta(y, \chi)}{y\log^2 y}\mathrm{d}y,
\end{equation*}
where,  using that $x_2(x)$ is increasing within the concerned range of $x$,
\begin{align*}
\int_{x_2(x)}^{x} \frac{\sum_{\substack{\chi\: (\text{mod }k) \\ \chi \ne \chi_0}}\left|\theta(y,\chi)\right|}{y\log^2 y}\mathrm{d}y & \le \frac{v_1(X_1)}{1-\frac{6}{\log x_2(X_1)}}\int_{x_2(x)}^{x}\frac{1}{\log^6 y}\left(1-\frac{6}{\log y}\right)\mathrm{d}y\\ &\le\frac{v_1(X_1)}{1-\frac{6}{\log x_2(X_1)}}\cdot\frac{x}{\log^6 x}
\end{align*}
and
\begin{align*}
    \sum_{\substack{\chi\: (\text{mod }k) \\ \chi \ne \chi_0}}\left|\pi(x_2(x),\chi)\right|&=\sum_{\substack{\chi\: (\text{mod }k) \\ \chi \ne \chi_0}}\left|\sum_{p\le x_2(x)}\chi(p)\right|=\sum_{\substack{\chi\: (\text{mod }k) \\ \chi \ne \chi_0}}\left|\sum_{\substack{a\: (\text{mod }k)\\(a,k)=1}}\chi(a)\sum_{\substack{p\le x_2(x)\\p\equiv a\:(\text{mod }k)}}1\right|\\ &\le (\varphi(k))^2\max_{\substack{a\\(a,k)=1}}\pi(x_2(x);k,a)
\end{align*}
which, by an explicit form of the Brun--Titchmarsh theorem \cite[Theorem 2]{M-V}, is bounded above by 
\begin{equation*}
    \varphi(k)\frac{2x_2(x)}{\log (x_2(x)/k)}\le \varphi(k)\frac{3x_2(x)}{\log x_2(x)}\le k\cdot \frac{3x_2(x)}{\log x_2(x)}\le\frac{3x}{\log^6 x}.    
\end{equation*}
This proves \eqref{pichieq} as required.
\end{proof}

\subsection{Notation and conditions}\label{notationsect}
We now introduce some further notation and conditions that will be used throughout. First, we let $0<\delta<2$ be a parameter, $X_2$ and $X_3$ be fixed positive real numbers, and $N$ be a positive even integer which we will set to be greater than either $X_2$ or $X_3$. Then, we set $X,Y,Z>0$ to be real numbers such that 
\begin{equation}\label{XYZdef}
\frac{N}{y}<X\le \frac{N}{z}, \quad XY<2N, \quad Y>Z, \quad Y>z,
\end{equation}
where
\begin{equation} \label{eq: def-z-y}
    z=N^{1/8},\quad\text{and}\quad y=N^{1/3}.
\end{equation}
In the final part of our sieving process (\S \ref{section:3}), precise expressions will be given for $X$, $Y$ and $Z$. We also define
\begin{align}
    &x_1=x_1(N):=\frac{N}{\log^{5} N},\quad x_2=x_2(Y):=\frac{Y}{\log^{15} Y},\label{notcon1}\\
    &K_{\delta}(x)
    :=\log^\delta x,\quad Q_1(x):=\log^{10}x,\quad P(z):=\prod_{\substack{p<z\\p\nmid N}}p.\label{notcon2}
\end{align}

As in Lemma \ref{cor:chisum}, the functions $x_1$ and $x_2$ will be used to control the error term resulting from partial summation arguments. Now, let $i\in\{1,2\}$. With regard to Theorem \ref{theo:kadiri}, we let $k_0(x_i)$ be the exceptional modulus up to $Q_1(x_i)$ (if it exists) and
\begin{equation}
\label{eq:k11}
    k_i:=
    \begin{cases}
        k_0(x_i),&\text{if $k_0(x_i)$ exists and $(k_0(x_i),N)=1$},\\
        0,&\text{otherwise.}
    \end{cases}
\end{equation}
By \cite[pp. 296--297]{M-V2}, $k_0(x_i)$ is square-free or 4 times a square-free number. Thus, since $N$ is even, $k_i$ is a square-free odd number whenever $k_i \ne 0$. 

In what follows, we will separately consider the cases $k_i<K_{\delta}(x_i)$ and $K_{\delta}(x_i)\le k_i\le Q_1(x_i)$. This is because if $k_i<K_{\delta}(x_i)$ then we can directly bound the contribution of the Siegel zero in Lemma \ref{lemma:PNTPAP1} using the results of \cite{Bordignon1} and \cite{Bordignon2}. On the other hand, if $k_i$ is too large this is not possible so a more complicated argument is required as to avoid the contribution from the exceptional zero.

\subsection{The case when the exceptional modulus is small}\label{sectsmall}

We first consider the case  
\begin{equation}
    \label{eq:smallmodulus}
    k_i<K_{\delta}(x_i)=\log^{\delta}x_i,
\end{equation}
where the value of $i\in\{1,2\}$ will be specified for each result. We begin with the following lemma, which is very similar\footnote{In the proof of Lemma \ref{small313} we also fix a couple of errors in the proof of \cite[Theorem 1.4]{Bordignon4}. For example, the bound (12) in \cite{Bordignon4} is missing a factor of $\log x$. This error also appears in the proof of \cite[Theorem 1.2]{Bordignon4} but contributes so little that none of the final computational results are affected.} to \cite[Theorem 1.4]{Bordignon4}.

\begin{lemma}\label{small313}
    Let $E_{\psi}(x;k,l)$ be as in \eqref{eq: E_f(x;k,l)} and $N$ be a positive even integer.
    Suppose $x_1 = x_1(N)$, $K_{\delta}(x_1)$, and $Q_1(x_1)$ are defined by \eqref{notcon1} and \eqref{notcon2}, $k_1<K_{\delta}(x_1)$, and $\log\log x_1\ge 10.4$. Let $H=H(N):=\frac{\sqrt{x_1}}{\log^{10}x_1}$ and $y\in[x_1,N]$. We have
    \begin{align}\label{psisum}
        \sum_{\substack{d\le H\\(d,N)=1}}\mu^2(d)\left|E_\psi(y;d,N)\right|&<1.1\log(Q_1(x_1))\left(\frac{3.2\cdot 10^{-8}y}{\log^{8}y}+\frac{y^{\beta_0(x_1)}}{\beta_0(x_1)}\right)\nonumber\\ 
        &+27\cdot\mathcal{E}(y)+\frac{\sqrt{y}}{2(\log 2)\log^{8}y} + 0.4\log^3y,
    \end{align}
    where
    \begin{equation} \label{eq: def-mathcal-E}
        \mathcal{E}(y):=\frac{4y\log^{\frac{9}{2}}y}{\log^{10}x_1(y)}+\frac{4y}{\log^{5.5}y}+\frac{18y^{\frac{11}{12}}}{\log^{\frac{1}{2}}y}+\frac{5}{2}y^{\frac{5}{6}}\log^{\frac{11}{2}}y
    \end{equation}
    and
    \begin{equation}\label{siegelmin}
        \beta_0(x_1):=1-\nu(x_1),\quad \nu(x_1):=\min\left\{\frac{100}{\sqrt{K_{\delta}(x_1)}\log^2 K_{\delta}(x_1)},\frac{1}{2R_1\log(Q_1(x_1))}\right\},
    \end{equation}
    with $R_1=2.0452$ as in Theorem \ref{theo:kadiri}.
\end{lemma}
\begin{proof}
    By \cite[(11.22)]{M-V2} we have for $(a,q)=1$
    \begin{equation*}
        \psi(x;q,a) = \frac{1}{\varphi(q)} \sum_{\chi} \overline{\chi}(a) \psi(N, \chi),
    \end{equation*}
    with $\psi(x;q,a)$ defined as in \eqref{sec:PNT}, and thus,
    \begin{equation*}
        \psi(y;d,N)-\frac{\psi(y)}{\varphi(d)} =
        \frac{1}{\varphi(d)} \sum_{\substack{\chi \: (\text{mod }d) \\ \chi \ne \chi_0}} \overline{\chi}(N) \psi(y, \chi) - \frac{1}{\varphi(d)}(\psi(y) - \psi(y, \chi_0)),
    \end{equation*}
    where $\chi_0$ denotes the trivial character modulo $d$. We note that if $\chi^*$ induces $\chi$ modulo $d$, recalling the definition for $\psi(y,\chi)$ in \eqref{eq: f(x,chi)}, then by Lemma \ref{omegalem}
    \begin{equation}\label{induceeq}
        \left|\psi(y,\chi)-\psi(y,\chi^*)\right|\le
        \sum_{\substack{p^m\le y\\m\mid d}}\log p\le\log y\sum_{p\mid d}1\le\log y\frac{\log d}{\log 2} \le \frac{\log^2 y}{2 \log 2}.
    \end{equation}
    Thus,
    \begin{equation*}
        |\psi(y) - \psi(y, \chi_0)| \le\frac{\log^2 y}{2 \log 2},
    \end{equation*}
    so that by Lemma \ref{muphilem}
    \begin{equation} \label{eq: psi(y)-psi(y,chi_0)}
        \sum_{\substack{d\le H\\(d,N)=1}} \frac{\mu^2(d)}{\varphi(d)}|\psi(y) - \psi(y, \chi_0)| \le \frac{\log^2 y}{2 \log 2}
        \sum_{\substack{d\le H\\(d,N)=1}} \frac{\mu^2(d)}{\varphi(d)} \le 0.4\log^3 y.
    \end{equation}
    It remains to bound
    \begin{equation*}
        \sum_{\substack{d\le H\\(d,N)=1}} \frac{\mu^2(d)}{\varphi(d)} \sum_{\substack{\chi \: (\text{mod }d) \\ \chi \ne \chi_0}} |\psi(y, \chi)|.
    \end{equation*}
    For this we consider two cases $1\le d\le Q_1(x_1)$ and $Q_1(x_1)<d\le H$. In the first case, we have by Lemmas \ref{muphilem} and \ref{notrivialPNTAP}
    \begin{equation} \label{eq: sum_psi(y,chi)}
        \sum_{\substack{d\le Q_1(x_1)\\(d,N)=1}} \frac{\mu^2(d)}{\varphi(d)} \sum_{\substack{\chi \: (\text{mod }d) \\ \chi \ne \chi_0}} |\psi(y, \chi)| \le 1.1\log Q_1(x_1) \left( \frac{3.2\cdot 10^{-8}y}{\log^{8}y}+\frac{y^{\beta_0(x_1)}}{\beta_0(x_1)} \right),
    \end{equation}
    where $\beta_0(x_1)$ is as defined in \eqref{siegelmin}. To see why each potential Siegel zero $\beta_d$ modulo $d$ satisfies $\beta_d\le \beta_0(x_1)$ we consider two cases:
    \begin{itemize}
        \item[(a)] Suppose $k_1 \mid d$. Then $k_1 = k_0(x_1)$ is non-zero and bounded above by $K_{\delta}(x_1)$ by the assumption in \eqref{eq:smallmodulus}. The Siegel zero $\beta_d$ is the exceptional zero modulo $k_1$, and thus bounded by \cite[Theorem 1.3]{Bordignon1} and \cite[Theorem 1.3]{Bordignon2} as follows
        \begin{equation*}
            \beta_d\le 1-\frac{100}{\sqrt{K_{\delta}(x_1)}\log^2K_{\delta}(x_1)}.
        \end{equation*}
        \item [(b)] Suppose $k_1\nmid d$ and recall the definition of $k_1$ in \eqref{eq:k11}. If $k_1 = 0$, then either $k_0(x_1)$ does not exist or $(k_0(x_1), N) \ne 1$ and $d$ is not divisible by the exceptional modulus $k_0$ because $(d,N)=1$. If $k_1 \ne 0$, then $k_1 = k_0(x_1) \nmid d$ and again $\beta_d$ is not exceptional. In both cases $\beta_d$ can be bounded by Theorem \ref{theo:kadiri}, that is
        \begin{equation*}
            \beta_d\le 1-\frac{1}{2R_1\log(Q_1(x_1))}.
        \end{equation*}
    \end{itemize}
    Combining cases (a) and (b) gives \eqref{siegelmin} as desired. Now we consider $Q_1(x_1)<d\le H$. For this range of $d$, we roughly follow the proof of \cite[Theorem 1.3]{Akbary}. First, by \eqref{induceeq},
    \begin{equation} \label{eq: case Q_1 < d < H, part 1}
        \frac{1}{\varphi(d)} \sum_{\substack{\chi \: (\text{mod }d) \\ \chi \ne \chi_0}} |\psi(y, \chi)|
        \le\frac{1}{\varphi(d)}\sum_{\chi\:(\text{mod }d)}\left|\psi(y,\chi^*)\right|+\frac{\log^2y}{2\log 2}.
    \end{equation}
    Letting $\sum^*$ denote the sum over all primitive characters, we then have
    \begin{align}
        &\sum_{\substack{Q_1(x_1)< d\le H\\(d,N)=1}}\frac{\mu^2(d)}{\varphi(d)} \sum_{\substack{\chi \: (\text{mod }d) \\ \chi \ne \chi_0}} |\psi(y, \chi)|\nonumber\\
        &\qquad\qquad\le \sum_{\substack{Q_1(x_1)< d\le H\\(d,N)=1}}\frac{\mu^2(d)}{\varphi(d)}\sum_{\chi\: (\text{mod }d)}|\psi(y,\chi^*)|+H\frac{\log^2y}{2\log 2}\nonumber\\
        &\qquad\qquad\le\left(\sum_{1\le m\le H}\frac{\mu^2(m)}{\varphi(m)}\right) \sum_{\substack{Q_1(x_1)<d\le H\\(d,N)=1}}\frac{\mu^2(d)}{\varphi(d)}\sideset{}{^*}\sum_{\chi\: (\text{mod }d)}|\psi(y,\chi)|+\frac{\sqrt{y}}{2(\log 2)\log^{8}y} \nonumber\\
        &\qquad\qquad\le 1.1\log y\sum_{\substack{Q_1(x_1)<d\le H\\(d,N)=1}}\frac{\mu^2(d)}{\varphi(d)}\sideset{}{^*}\sum_{\chi\: (\text{mod }d)}|\psi(y,\chi)|+\frac{\sqrt{y}}{2(\log 2)\log^{8}y}, \label{eq: case Q_1 < d < H, part 2}
    \end{align}
    where in the second inequality we used that $\varphi(ab)\ge\varphi(a)\varphi(b)$ and in the third inequality we used Lemma \ref{muphilem} and that $H$ is an increasing function in the range of interest. To finish off, one repeats the argument from \cite[pp. 1929--1930]{Akbary}, whereby \cite[Theorem 1.2]{Akbary} and partial summation give
    \begin{align*}
        &\log y\sum_{\substack{Q_1(x_1)<d\le H\\(d,N)=1}}\frac{\mu^2(d)}{\varphi(d)}\sideset{}{^*}\sum_{\chi\: (\text{mod }d)}|\psi(y,\chi)|\\
        &\qquad\le\frac{1.1}{2}\log y\cdot 48.84\left(4\frac{y}{Q_1(x_1)}+4y^{\frac{1}{2}}H+18y^{\frac{2}{3}}H^{\frac{1}{2}}+5y^{\frac{5}{6}}\log\left(\frac{eH}{Q_1(x_1)}\right)\right)(\log y)^{\frac{7}{2}}\\
        &\qquad\le 27\cdot \mathcal{E}(y).
    \end{align*}
    This proves the lemma.
\end{proof}

We now convert the above result into a statement involving $\pi(x)$.

\begin{lemma}\label{small533}
    Keep the notation and conditions of Lemma \ref{small313}, and assume $N \ge X_2$ with $\log \log x_1(X_2) \ge 10.4$. Then
    \begin{equation*}
        \sum_{\substack{d\le H\\(d,N)=1}}\mu^2(d)|E_{\pi}(N;d,N)|<\frac{p(X_2)N}{\log^3N}
    \end{equation*}
    with
    \begin{align}\label{peq}
        p(X_2)&:=p_1(X_2)\left(1+\frac{1}{\log^2 X_2\log^3 x_1(X_2)}+\frac{1}{\left(1-\frac{4}{\log x_1(X_2)}\right)\log X_2}\right)+\frac{2.2}{\log^2 X_2}\notag,\\
        p_1(X_2)&:=p_2(X_2)+\frac{1}{\log^{8}x_1(X_2)}\left(0.67+\frac{2}{x_1(X_2)^{\frac{1}{6}}}\right)\notag,\\
        p_2(X_2)&:=\max_{y\ge x_1(X_2)}\left[\frac{\log^2 y}{y}\left(1.1\log(Q_1(y))\left(\frac{3.2\cdot 10^{-8}y}{\log^{8}y}+\frac{y^{\beta_0(x_1)}}{\beta_0(x_1)}\right)\right.\right.\\
        &\qquad\qquad\qquad\qquad\qquad\left.\left.+27\cdot\mathcal{E}(y)+\frac{\sqrt{y}}{2(\log 2)\log^{8}y}+0.4\log^3y\right)\right].\notag
    \end{align}
\end{lemma}
\begin{proof}
Let $y \in [x_1(N), N]$. By Lemma \ref{small313}
\begin{equation*}
    \sum_{\substack{d\le H\\(d,N)=1}}\mu^2(d)|E_{\psi}(y;d,N)| \le \frac{p_2(X_2)y}{\log^2 y}.
\end{equation*}

Next, since $|\psi(y; d,N) - \theta(y; d,N)| \le \psi(y)-\theta(y)\le 1.02y^{1/2}+3y^{1/3}$ (\cite[(3.39)]{R-S}) we have
\begin{align*}
    |E_{\psi}(y;d,N) - E_{\theta}(y;d,N)| &=\left|\psi(y;d,N)-\theta(y;d,N) - \frac{\psi(y) - \theta(y)}{\varphi(d)}\right|\\
    &<\max \left\{1.02y^{1/2}+3y^{1/3},\frac{1.02y^{1/2}+3y^{1/3}}{\varphi(d)}\right\}\\
    &=1.02y^{1/2}+3y^{1/3},
\end{align*}
noting that $\psi(y;d,N)-\theta(y;d,N)$ and $\frac{\psi(y) -\theta(y)}{\varphi(d)}$ are both positive.

Thus, using Lemma \ref{squarefreelem}
\begin{align}\label{p1psitheta}
    \sum_{\substack{d\le H\\(d,N)=1}}\mu^2(d)|E_{\theta}(y;d,N)|&\le\sum_{\substack{d\le H\\(d,N)=1}}\mu^2(d)|E_{\psi}(y;d,N)| + 0.65H(1.02y^{1/2}+3y^{1/3})\notag\\
    &\le\frac{p_1(X_2)y}{\log^2 y}.
\end{align}

Next, by partial summation
\begin{align}\label{partialsum2}
    E_{\pi}(N;d,N) = E_{\pi}(x_1; d, N) + \frac{E_{\theta}(N; d, N)}{\log N} - \frac{E_{\theta}(x_1; d, N)}{\log x_1} + \int_{x_1}^{N} \frac{E_{\theta}(y; d, N)}{y \log^2 y} \mathrm{d}y.
\end{align}
Now, by the Brun-Titchmarsh theorem \cite[Theorem 2]{M-V}, $$\pi(x_1, d; N) < \frac{2x_1}{\log(x_1/d) \varphi(d)}$$ and $$\frac{\pi(x_1)}{\varphi(d)} < \frac{2x_1}{\log(x_1) \varphi(d)} \le \frac{2x_1}{\log(x_1/d)\varphi(d)},$$ 
for any integer $d \ge 1$. By combining this with Lemma \ref{muphilem}, we get
\begin{equation*}
    \sum_{\substack{d\le H\\(d,N)=1}}\mu^2(d)|E_{\pi}(x_1;d,N)|\le 2\frac{x_1}{\log(x_1/H)}\sum_{\substack{d\le H\\(d,N)=1}}\frac{\mu^2(d)}{\varphi(d)}\le 2.2x_1.
\end{equation*}
Then,
\begin{align*}
    \sum_{\substack{d\le H\\(d,N)=1}}\mu^2(d)\left|\int_{x_1}^{N} \frac{E_{\theta}(y; d, N)}{y \log^2 y} \mathrm{d}y\right|&\le\frac{p_1(X_2)}{1-\frac{4}{\log x_1(X_2)}}\int_{x_1}^N\frac{1}{\log^4y}\left(1-\frac{4}{\log y}\right)\mathrm{d}y\\
    &<\frac{p_1(X_2)}{1-\frac{4}{\log x_1(X_2)}}\cdot\frac{N}{\log^4 N}. 
\end{align*}
The remaining terms of \eqref{partialsum2} can be bounded using \eqref{p1psitheta} to give
\begin{equation*}
    \sum_{\substack{\substack{d\le H\\(d,N)=1}}}\mu^2(d)|E_{\pi}(N;d,N)|<\frac{p(X_2)N}{\log^3N}
\end{equation*}
as required.
\end{proof}

Finally, we prove an upper bound related to a bilinear form, to be used in Section \ref{section:3}.
\begin{lemma}
\label{lemma:bf}
Suppose $N$ is a positive even integer, $y=N^{\frac{1}{3}}$, $z=N^{\frac{1}{8}}$ and $X,Y,Z>0$ be real numbers such that
\begin{equation*}
    \frac{N}{y}<X\le \frac{N}{z}, \quad XY<2N, \quad Y>Z, \quad Y>z.
\end{equation*}
Let $a(n)$ be an arithmetic function with $\left|a(n)\right|\le 1$ for all $n$.
Suppose $x_2=x_2(Y)$ and $K_{\delta}(x_2)$ are defined by \eqref{notcon1} and \eqref{notcon2} respectively,
$k_2<K_{\delta}(x_2)$, and $N>(X_3)^8$, with $X_3$ such that $\log\log x_2(X_3)\ge 10.4$. With $D^*=\frac{\sqrt{XY}}{\log^{10} Y}$, we have
\begin{equation}\label{bilineq}
\sum_{\substack{d<D^* \\ d\mid P(y)}} \max_{(a,d)=1}\left| \sum_{n<X}\sum_{\substack{Z\le p <Y\\ np\equiv N\: (\text{mod }d)}}a(n)-\frac{1}{\varphi(d)}\sum_{n<X} \sum_{\substack{Z\le p <Y\\ (np,d)=1}}a(n)\right|\le \frac{m(X_3)XY}{\log^3 Y},
\end{equation}
Here,
\begin{align*}
m(X_3):=39 v_0(X_3)+\frac{108\log^{16} X_3}{ X_3\log \log X_3}+\frac{ 26\log^5X_3}{\log^{10}(x_2(X_3))} +88\log^5 X_3\left( \frac{1}{ X_3^{\frac{8}{3}}}+\frac{1}{ X_3^{\frac{1}{2}}}\right)+\frac{106}{\log^6 X_3}
\end{align*}
and $v_0(X_3)$ is equal to $v_k(X_3)$ from Lemma \ref{cor:chisum} but with $\beta_k$ (appearing in \eqref{eq:v1}) replaced with
\begin{equation}\label{siegelmin2}
    \beta_0(x_2):=1-\nu(x_2),\quad \nu(x_2):=\min\left\{\frac{100}{\sqrt{K_{\delta}(x_2)}\log^2 K_{\delta}(x_2)},\frac{1}{2R_1\log(Q_1(x_2))}\right\}.
\end{equation}
Note that $R_1=2.0452$ as in Theorem \ref{theo:kadiri}.
\end{lemma}
\begin{proof}
Following \cite[\S 10.7]{Nathanson}, we write $\chi=\chi_{0,s} \chi_1$ with $d=sr$ and $\chi_1$ primitive, and rewrite the left-hand side of \eqref{bilineq} as
\begin{align}
\label{eq:rs}
\sum_{\substack{rs<D^*\\ rs\mid P(y)}} \frac{1}{\varphi(sr)}\sideset{}{^{*}}\sum_{\substack{\chi\: (\text{mod }r)\\ \chi\neq \chi_{0,r}}}\left|\sum_{\substack{n<X\\ (n,s)=1}}a(n) \chi(n) \right|\left|\sum_{\substack{Z\le p <Y \\ p\nmid s}}\chi(p) \right|,
\end{align}
where $*$ means that the sum is restricted to primitive characters. We begin estimating the sum restricted to $r<D_0$, with 
\begin{equation}\label{def-D_0}
    D_0:=\log^{10} (x_2(Y)).
\end{equation} 
Since $Y>Z$, $Y>z=N^{\frac{1}{8}}$ and $N> (X_3)^8$, we obtain, by Lemma \ref{cor:chisum}, 
\begin{align}
\label{eq:fa}
\begin{split}
\sideset{}{^{*}}\sum_{\substack{\chi\: (\text{mod }r)\\\chi\neq\chi_{0,r}}}\left|\sum_{\substack{Z\le p <Y \\ p\nmid s}}\chi(p) \right|&\le \sum_{\substack{\chi\: (\text{mod }r)\\\chi\neq\chi_{0,r}}}\left(\left| \pi(\lceil Y-1\rceil,\chi)-\pi(\lfloor Z+1\rfloor,\chi)\right|+\omega(s) \right)
\\ & \le \frac{2v_r(X_3)Y}{\log^5 Y}+\frac{1.3841\varphi(r) \log D^*}{\log \log D^*},
\end{split}
\end{align}
where we used Lemma~\ref{omegalem} to bound $\omega(n)$. The Siegel zero $\beta_r$ (appearing in the function $v_r(X_3)$) satisfies $\beta_r\le \beta_0(x_2)$ by the same argument as in the proof of Lemma \ref{small313}. Hence $v_r(X_3)$ can be bounded by $v_0(X_3)$ in \eqref{eq:fa}. Then, using Lemma \ref{muphilem} and $|a(n)| \le 1$ we obtain
\begin{align}
\label{eq:rs1}
\sum_{\substack{rs<D^*\\ r <D_0\\rs\mid P(y)}}\frac{1}{\varphi(rs)}&\sideset{}{^{*}}\sum_{\substack{\chi\: (\text{mod }r)\\\chi\neq\chi_{0,r}}}\left|\sum_{\substack{n<X\\ (n,s)=1}}a(n) \chi(n) \right|\left|\sum_{\substack{Z\le p <Y \\ p\nmid s}}\chi(p) \right| \notag\\ & \le \left(\frac{2v_0(X_3)XY}{\log^5 Y} +\frac{ 1.3841D_0X\log D^*}{\log \log D^*}\right) \left(\sum_{l\le D^*}\frac{\mu^2(l)}{\varphi(l)}\right)^2\notag\\ & \le \frac{2.42v_0(X_3)XY \log^2 D^* }{\log^5 Y} +\frac{1.21\cdot1.3841D_0X\log^3 D^*}{\log \log D^*}\notag\\
&\le\frac{39v_0(X_3)XY}{\log^3 Y}+\frac{108X\log^{13}Y}{\log\log Y},
\end{align}
where in the last line we used that $\frac{\log^3 D^*}{\log \log D^*}$ increases for $5 \le D^*$ and $D^*\le Y^4$, which follow readily from the restrictions on $X$ and $Y$ and the definition of $D^*$.
We are now left with estimating the sum in \eqref{eq:rs} restricted to $r \ge D_0$, which upon using $\varphi(rs)\ge\varphi(r)\varphi(s)$ is bounded by
\begin{equation*}
\sum_{\substack{s<D^*\\s\mid P(y)}}\frac{1}{\varphi(s)} \sum_{\substack{ D_0\le r \le D^*\\r\mid P(y) }}\frac{1}{\varphi(r)}\sideset{}{^{*}}\sum_{\substack{\chi\: (\text{mod }r)\\\chi\neq\chi_{0,r}}}\left|\sum_{\substack{n<X\\ (n,s)=1}}a(n) \chi(n) \right|\left|\sum_{\substack{Z\le p <Y \\ p\nmid s}}\chi(p) \right|.
\end{equation*} 
To do so, we divide the interval $D_0\le r \le D^*$ into subintervals of the form 
\begin{equation*}
D_k \le r \le 2D_k, \quad \text{where} \quad D_k:=2^kD_0, \quad 0 \le k \le \frac{\log (D^*/D_0)}{\log 2}.
\end{equation*}
Using the Cauchy--Schwarz inequality and the large sieve inequality  \cite[p. 160]{Davenport} as in the proof of \cite[Theorem 10.7]{Nathanson}, we obtain, for each $D_k$, 
\begin{align*}
 \sum_{\substack{ D_k\le r < 2D_k \\r\mid P(y)}}\frac{1}{\varphi(r)} &\ \ \sideset{}{^{*}}\sum_{\substack{\chi\: (\text{mod }r)\\\chi\neq\chi_{0,r}}}\left|\sum_{\substack{n<X\\ (n,s)=1}}a(n) \chi(n) \right|\left|\sum_{\substack{Z\le p <Y \\ p\nmid s}}\chi(p) \right| 
 \\
  \le & \frac{1}{D_k}\left( \sum_{D_k\le r < 2D_k}\ \ \sideset{}{^{*}}\sum_{\substack{\chi\: (\text{mod }r)\\\chi\neq\chi_{0,r}}}\frac{r}{\varphi(r)} \left|\sum_{\substack{n<X\\ (n,s)=1}}a(n) \chi(n) \right|^2\right)^{\frac{1}{2}} 
  \\&
 \cdot \left( \sum_{D_k\le r < 2D_k}\ \ \sideset{}{^{*}}\sum_{\substack{\chi\: (\text{mod }r)\\\chi\neq\chi_{0,r}}}\frac{r}{\varphi(r)} \left|\sum_{\substack{Z\le p <Y \\ p\nmid s}}\chi(p) \right|^2\right)^{\frac{1}{2}} 
\\ 
 \le &\frac{1}{D_k}\left((X+12D_k^2)(Y+12D_k^2)XY \right)^{\frac{1}{2}}\\
 = & \left( \frac{(XY)^2}{D_k^2}+12XY^2+12YX^2+144D_k^2XY \right)^{\frac{1}{2}}\\
 \le &\frac{XY}{D_0}+\sqrt{12X}Y+\sqrt{12Y}X+12D_k\sqrt{XY}.
\end{align*}
Thus, summing over $0 \le k \le \frac{\log (D^*/D_0)}{\log 2}$ and using Lemma \ref{muphilem},
\begin{align}
\label{eq:rs2}
 &\sum_{\substack{s<D^*\\s\mid P(y) }}\frac{1}{\varphi(s)} \sum_{\substack{D_0<r<D^* \\ r\mid P(y)}}\frac{1}{\varphi(r)} \sideset{}{^{*}}\sum_{\chi\: (\text{mod }r)}\left|\sum_{\substack{n<X\\ (n,s)=1}}a(n) \chi(n) \right|\left|\sum_{\substack{Z\le p <Y \\ p\nmid s}}\chi(p) \right| \notag\\& \le 1.1\log D^*\left[\frac{\log D^*}{\log 2}\left(\frac{XY}{\log^{10}(x_2(Y))}+\sqrt{12}XY\left(\frac{1}{\sqrt{X}}+\frac{1}{\sqrt{Y}}\right)\right)+24D^*\sqrt{XY}\right]\notag\\
 &\le\frac{26XY\log^2(Y)}{\log^{10}(x_2(Y))}+88XY\log^2Y\left(\frac{1}{\sqrt{X}}+\frac{1}{\sqrt{Y}}\right)+\frac{106XY}{\log^9Y},
\end{align}
where we have again used $D^*\le Y^4$ and also $\sum_kD_k\le 2D^*$. We now obtain the desired result from \eqref{eq:rs}, \eqref{eq:rs1} and \eqref{eq:rs2}.
\end{proof}

\subsection{The case when the exceptional modulus is large}
We recall that for $i\in\{1,2\}$, $k_0(x_i)$ denotes the modulus of the exceptional zero up to $Q_1(x_i)$ (Equations \eqref{notcon1} and \eqref{notcon2}) if it exists and $k_i$ is defined by \eqref{eq:k11}. In this section, we suppose that $k_i\ge K_{\delta}(x_i)=\log^{\delta}x_i$. This means that $(k_0(x_i), N)=1$ and thus $k_i = k_0(x_i)$. Several of the following lemmas will be slight variations on those in Section \ref{sectsmall}.

\begin{lemma}\label{lem: general-5.3.2}
Suppose $N$ is a positive even integer.
Let $E_f(x;k,l)$ be as in \eqref{eq: E_f(x;k,l)}.
Suppose $k_1\ge K_{\delta}(x_1)$, $N \ge X_2$, and $\log\log x_1 (X_2)\ge 10.4$. Then, for any $k$ strictly dividing $k_1$,
$$\frac{\varphi(k)}{N}|E_{\pi}(N; k, N)| < \frac{c(X_2)}{\log^3 N},$$
with
\begin{align}
    c(X_2)&:=c_1(X_2)\left(1+\frac{1}{\log^2(X_2)\log^3x_1(X_2)}+\frac{1}{\left(1-\frac{4}{\log x_1(X_2)}\right)\log X_2}\right)+\frac{1}{\log^2 X_2}, \notag\\
    c_1(X_2) &:= \max_{y\ge x_1(X_2)}\Bigg[\frac{3.2 \cdot 10^{-8}}{\log^{6} y} + \log^2y\Bigg(\left(1 - \frac{1}{2 R_1 \log Q_1(y)} \right)^{-1} y^{ - \frac{1}{2 R_1 \log Q_1(y)}} \notag\\
    &\qquad\qquad\qquad+ Q_1(y)\left(\frac{1.02}{\sqrt{y}}+\frac{3}{y^{2/3}}\right)
    + 9.4(\log y)^{1.515}\exp(-0.8274\sqrt{\log y})\Bigg)\Bigg].\label{c1eq}
\end{align}
\end{lemma}

\begin{proof}
Let $y\in[x_1(N),N]$. Let $\beta_k$ be the Siegel zero modulo $k$ if it exists. Since $k$ strictly divides $k_1$, it is not the exceptional modulus $k_0$ up to $Q_1(x_1)$ and thus by Theorem \ref{theo:kadiri},
\begin{equation} \label{eq: beta-bound}
    \beta_k \le 1 - \frac{1}{2 R_1 \log Q_1(x_1)}\le 1-\frac{1}{2R_1\log Q_1(y)}.
\end{equation}
Then, since $\frac{y^{\beta_k - 1}}{\beta_k}$ increases as a function of $\beta_k$,
$$\frac{y^{\beta_k - 1}}{\beta_k} \le (1 - \nu(y))^{-1} y^{-\nu(y)},$$
with $\nu(y)=\frac{1}{2 R_1 \log Q_1(y)}$. Thus by Lemma \ref{lemma:PNTPAP1},
$$\frac{\varphi(k)}{y} \left| \psi(y; k, N) - \frac{y}{\varphi(k)} \right| < \frac{3.2 \cdot 10^{-8}}{\log^{8} y} + (1-\nu(y))^{-1} y^{- \nu(y)}.$$
By definition, $k_1$ is the exceptional modulus up to $Q_1(x_1(N))$ hence $\varphi(k) < k_1 \le Q_1(x_1(N)) \le Q_1(y)$. Thus, using $|\psi(y; k,N) - \theta(y; k,N)| \le \psi(y)-\theta(y)\le 1.02y^{1/2}+3y^{1/3}$ (\cite[(3.39)]{R-S}), we obtain
\begin{align*}
    \frac{\varphi(k)}{y} \left| \theta(y; k, N) - \frac{y}{\varphi(k)} \right| < \frac{3.2\cdot 10^{-8}}{\log^{8} y} + (1 - \nu(y))^{-1} y^{ - \nu(y)}
    + Q_1(y)\left(\frac{1.02}{\sqrt{y}}+\frac{3}{y^{2/3}}\right).
\end{align*}
Then, by \cite[Corollary 1.2 \& Table 1, l.1]{JohnstonYang} and the triangle inequality
\begin{align*}
    \frac{\varphi(k)}{y} \left| E_{\theta}(y; k,N) \right| &< \frac{3.2\cdot 10^{-8}}{\log^{8} y} + (1 - \nu(y))^{-1} y^{ - \nu(y)}
    + Q_1(y)\left(\frac{1.02}{\sqrt{y}}+\frac{3}{y^{2/3}}\right) \\
    &\qquad\qquad + 9.4(\log y)^{1.515}\exp(-0.8274\sqrt{\log y}).
\end{align*}
Therefore, $$\frac{\varphi(k)}{y} \left| E_{\theta}(y; k,l) \right| < \frac{c_1(X_2)}{\log^2 y},$$
with $c_1(X_2)$ defined in \eqref{c1eq}.

It remains to express $E_{\pi}(x; k,N)$ by partial summation
\begin{align*}
    E_{\pi}(N; k, N) = E_{\pi}(x_1; k, N) + \frac{E_{\theta}(N; k, N)}{\log N} - \frac{E_{\theta}(x_1; k, N)}{\log x_1} + \int_{x_1}^{N} \frac{E_{\theta}(y; k, N)}{y \log^2 y} \mathrm{d}y,
\end{align*}
where
\begin{align*}
    \varphi(k)\left|\int_{x_1}^{N} \frac{E_{\theta}(y; k, N)}{y \log^2 y} \mathrm{d}y\right|&\le\frac{c_1(X_2)}{1-\frac{4}{\log x_1(X_2)}}\int_{x_1}^N\frac{1}{\log^4y}\left(1-\frac{4}{\log y}\right)\mathrm{d}y\\
    &<\frac{c_1(X_2)}{1-\frac{4}{\log x_1(X_2)}}\frac{N}{\log^4 N}
\end{align*}
and, by \cite[Theorem 2]{M-V}
\begin{align*}
    \frac{\varphi(k)}{N}|E_{\pi}(x_1; k,N)| &\le\max\left\{\frac{\varphi(k)}{N}\pi(x_1;k,N),\frac{\pi(x_1)}{N}\right\}\\
    &\le\max\left\{\frac{2x_1}{N\log(x_1/k)},\frac{x_1}{N}\right\}\\
    &=\frac{x_1}{N}\\
    &=\frac{1}{\log^5 N}.
\end{align*}
Therefore,
\begin{align*}
    \frac{\varphi(k)}{N}|E_{\pi}(N; k, N)| < \frac{c(X_2)}{\log^3 N}
\end{align*}
as required.
\end{proof}

Next we introduce some variants of Lemma \ref{small533}.

\begin{lemma}\label{large533}
Suppose $N$ is a positive even integer.
Let $E_f(x;k,l)$ be as in \eqref{eq: E_f(x;k,l)}. Suppose $k_1\ge K_{\delta}(x_1)$, $N \ge X_2$, and $\log\log x_1(X_2)\ge 10.4$. Then
    \begin{equation*}
        \sum_{\substack{d\le H\\(d,N)=1\\ k_1\nmid d}}\mu^2(d)|E_{\pi}(N;d,N)|<\frac{p^*(X_2)N}{\log^3 N},
    \end{equation*}
    where $p^*(X_2)=p(X_2)$ as in Lemma \ref{small533} with the $*$ indicating that $\beta_0(x_1)$ (appearing in \eqref{peq}) is replaced by
    \begin{equation}\label{betabound2}
        \beta_0^*(x_1):= 1-\frac{1}{2R_1\log (Q_1(x_1))}
    \end{equation}
    which is sharper than \eqref{siegelmin}.
\end{lemma}
\begin{proof}
    Identical to the proof of Lemmas \ref{small313} and \ref{small533} however the condition $k_1\nmid d$ means that $d$ is never exceptional so that any Siegel zero $\beta_d$ modulo $d$ can always be bounded as in \eqref{betabound2} using Theorem \ref{theo:kadiri}.
\end{proof}

\begin{lemma}\label{large534}
    Suppose $N$ is a positive even integer, $k_1\ge K_{\delta}(x_1)$, $N \ge X_2$, and $\log\log x_1(X_2)\ge 10.4$. Then, for each $k\mid k_1$ with $k\neq 1$, we have
    \begin{equation*}
        \sum_{\substack{d\le H/k \\ (d,N)=(d,k)=1}} \mu^2(d)\left|\pi(N;kd,N)-\frac{\pi(N;k,N)}{\varphi(d)}\right|<\frac{p^*(X_2)N}{\log^3 N},
    \end{equation*}
    where $p^*(X_2)$ is as in Lemma \ref{large533}, and $\pi(N; k, l)$ denotes the number of primes up to $N$ congruent to $l$ modulo $k$.
\end{lemma}
\begin{proof}
    Much of this proof is identical to those of Lemmas \ref{small313} and \ref{small533}, so we will be terse in some algebraic manipulations, only highlighting the differences to the previous proofs. For convenience, for an arithmetic function $f \in \{ \pi, \theta, \psi\}$ we will denote
    \begin{equation} \label{eq: def-D-f-x-q1-q2-l}
        D_{f}(x; q_1, q_2, l) := f(x; q_1 q_2, l) - \frac{f(x; q_1, l)}{\varphi(q_2)}
    \end{equation}
    so that we are trying to prove
    \begin{equation*}
        \sum_{\substack{d\le H/k \\(d,N)=(d,k)=1}} \mu^2(d)\left|D_{\pi}(N; k, d, N)\right|<\frac{p^*(X_2)N}{\log^3 N}.
    \end{equation*}
    Let $y\in[x_1(N),N]$. We have,
    \begin{align*}
        &D_{\psi}(y; k, d, N)\\
        &=\frac{1}{\varphi(kd)}\sum_{\substack{\chi\:(\text{mod } kd)}}\overline{\chi}(N)\psi(y;\chi) - \frac{1}{\varphi(d)\varphi(k)}\sum_{\substack{\chi_1\:(\text{mod } k)}}\overline{\chi_1}(N)\psi(y;\chi_1) \\
        &=\frac{1}{\varphi(kd)}\left[\sum_{\substack{\chi_1\:(\text{mod } k) \\ \chi_2\: (\text{mod }d)}}\overline{\chi_1}(N)\overline{\chi_2}(N)\psi(y;\chi_1\chi_2) - \sum_{\substack{\chi_1\:(\text{mod } k)}}\overline{\chi_1}(N)\overline{\chi_{0,d}}(N)\psi(y;\chi_1)\right] \\
        &=\frac{1}{\varphi(kd)}
        \left[\sum_{\substack{\chi_1\:(\text{mod } k) \\ \chi_2\: (\text{mod }d)}}\overline{\chi_1}(N)\overline{\chi_2}(N)\psi(y;\chi_1\chi_2) - \sum_{\substack{\chi_1\:(\text{mod } k)}}\overline{\chi_1}(N)\overline{\chi_{0,d}}(N)\psi(y;\chi_1 \chi_{0,d})\right] \\
        &\qquad\qquad\qquad\qquad-\frac{1}{\varphi(kd)}\sum_{\substack{\chi_1\:(\text{mod } k)}}\overline{\chi_1}(N)\overline{\chi_{0,d}}(N)\left(\psi(y;\chi_1)-\psi(y;\chi_1 \chi_{0,d})\right) \\
        &=\frac{1}{\varphi(kd)}
        \sum_{\substack{\chi_1\:(\text{mod } k) \\ \chi_2\ne \chi_{0,d}\: (\text{mod }d)}}\overline{\chi_1}(N)\overline{\chi_2}(N)\psi(y;\chi_1\chi_2) \\
        &\qquad\qquad\qquad\qquad-\frac{1}{\varphi(kd)}\sum_{\substack{\chi_1\:(\text{mod } k)}}\overline{\chi_1}(N)\overline{\chi_{0,d}}(N)\left(\psi(y;\chi_1)-\psi(y;\chi_1 \chi_{0,d})\right),
    \end{align*}
    where $\chi_{0,d}$ is the principal character modulo $d$.

We use above that $\chi_{0,d}(l)=1$ for $(l,d)=1$, and that for $(k,d)=1$ the character modulo $kd$ is represented in a unique way as the product of two characters modulo $k$ and modulo $d$.

Summing over $d$ and noting that $\mu^2(kd)=\mu^2(d)$, the above last term is bounded by 
\begin{align}
    \sum_{\substack{d\le H/k\\(d,N)=(d,k)=1}}\frac{\mu^2(d)}{\varphi(kd)}\left|\sum_{\substack{\chi_1\:(\text{mod } k)}}\overline{\chi_1}(N)\overline{\chi_{0,d}}(N)\left(\psi(y;\chi_1)-\psi(y;\chi_1 \chi_{0,d})\right)\right| \le 0.4\log^3 y \label{eq: 534error}
\end{align}
similar to inequality \eqref{eq: psi(y)-psi(y,chi_0)} from the proof of Lemma \ref{small313}.
Then, for the case $kd\le Q_1(x_1)$, we get analogously to \eqref{eq: sum_psi(y,chi)}
\begin{align}
    \sum_{\substack{dk \le Q_1(x_1) \\ (d,N)=(d, k)=1}} \frac{\mu^2(d)}{\varphi(kd)}
    \Bigg| \sum_{\substack{\chi_1\:(\text{mod } k) \\ \chi_2\ne \chi_{0,d}\: (\text{mod }d)}} & \overline{\chi_1}(N)\overline{\chi_2}(N)\psi(y;\chi_1\chi_2) \Bigg|\label{doublesum534}\\ 
    &\le 1.1\log Q_1(y) \Bigg(\frac{3.2\cdot 10^{-8}y}{\log^{8}y} + \frac{y^{\beta_0^*(x_1)}}{\beta_0^*(x_1)} \Bigg). \nonumber
\end{align}

Here we note that since $\chi_2\neq\chi_{0,d}$ and $k\mid k_1$, the exceptional character never appears in the inner sum of \eqref{doublesum534}. Thus, using Theorem \ref{theo:kadiri} we can bound each Siegel zero modulo $d$ by $\beta_0^*(x_1)$ (Equation \eqref{betabound2}).

The case $Q_1(x_1)/k < d \le H/k$ is also dealt with analogously to the inequalities \eqref{eq: case Q_1 < d < H, part 1} and \eqref{eq: case Q_1 < d < H, part 2} from the proof of Lemma \ref{small313}. Then finally, the conversion from $D_\psi(y;k,d,N)$ to $D_\pi(N;k,d,N)$ is done using the same reasoning as in Lemma \ref{small533}. 
\end{proof}

\begin{lemma}\label{bflem2}
    Keep the notation and conditions of Lemma \ref{lemma:bf} except assuming $k_2\ge K_{\delta}(x_2)$. We then have
    \begin{equation*}
        \sum_{\substack{d<D^* \\ d\mid P(y)\\k_2\nmid d}} \max_{(a,d)=1}\left| \sum_{n<X}\sum_{\substack{Z\le p <Y\\ np\equiv N\: (\text{mod }d)}}a(n)-\frac{1}{\varphi(d)}\sum_{n<X} \sum_{\substack{Z\le p <Y\\ (np,d)=1}}a(n)\right|\le \frac{m^*(X_3)XY}{\log^3 Y}
    \end{equation*}
    where $m^*(X_2)=m(X_2)$ as in Lemma \ref{lemma:bf} with the $*$ indicates that $\beta_0(x_2)$ is replaced with
    \begin{equation}\label{betabound3}
        \beta_0^*(x_2)= 1-\frac{1}{2R_1\log (Q_1(x_2))}
    \end{equation}
    which is sharper than \eqref{siegelmin2}.
\end{lemma}
\begin{proof}
    As with Lemma \ref{large533}, the proof is the same as the case $k_2<K_{\delta}(x_2)$, with the added condition $k_2\nmid d$ meaning that $d$ is not exceptional, giving \eqref{betabound3}. 
\end{proof}
 
\section{Preliminaries to sieving}
\label{sec:ps}
We now set up the main sieving argument that will be used to prove an explicit version of Chen's theorem. Some of the definitions that will be presented in this section were previously introduced. We decided to include them here to ease readability and make this section somewhat self-contained.\par
Fix 
\begin{equation}\label{Nyzeq}
N\ge X_2,\quad  z=N^{\frac{1}{8}},\quad  y=N^{\frac{1}{3}}.
\end{equation}
We shall consider the sets
\begin{equation}\label{Aapeq}
A:=\left\{ N-p:p\le N, p\nmid N \right\},\quad A_p:=\left\{ a \in A:p|a \right\} \quad 
A_d:=\bigcap_{p|d} A_p.
\end{equation}
with $p$ prime and $d$ square-free. Note that
\begin{equation}\label{AAqeq}
    \left|A\right|=\pi(N)-\omega(N)\quad \text{and}\quad\left|A_d\right|=\pi(N;d,N)-\omega (N;d,N)
\end{equation} 
where $\omega (n;q,a)$ denotes the number of prime factors of $n$ which are congruent to $a$ modulo $q$.
We also set
\begin{equation} \label{eq: def-S-A-n}
S(A,n):=\left| A-\bigcup_{p|n} A_p\right|
\end{equation}
and
\begin{equation}\label{Bset}
B:=\left\{ N-p_1p_2p_3:z \le p_1<y \le p_2 \le p_3, p_1p_2p_3 <N, (p_1p_2p_3,N)=1   \right\}
\end{equation}
where $p_1$, $p_2$ and $p_3$ are primes. Then, defining 
\begin{equation} \label{eq: def-P(x)}
    P(x):=\prod_{\substack{p<x\\p\nmid N}} p,
\end{equation}
one obtains the following bound.
\begin{lemma}[{\cite[Theorem 10.2]{Nathanson}}]
\label{lem:>}
Let $\pi_2(N)$ denote the number of representations of a given even integer $N$ as the sum of a prime and a semi-prime. We have\footnote{For ease of argument, we have added the condition $q\nmid N$ to the second term which was not present in \cite{Nathanson}. This condition is vacuous since if $q\mid N-p\in A$ and $q\mid N$ then $q=p$. However, this is not possible since $p\nmid N$.}
\begin{equation*}
\pi_2(N)>S(A,P(z))-\frac{1}{2}\sum_{\substack{z \le q <y\\q\nmid N}}S(A_q,P(z))-\frac{1}{2}S(B,P(y))-2N^{\frac{7}{8}}-N^{\frac{1}{3}}.
\end{equation*}
\end{lemma}
To prove Theorem \ref{Theo:B.} it thus suffices to give a good lower bound for $S(A,P(z))$ and good upper bounds for $S(B,P(y))$ and $S(A_q,P(z))$ for each prime $q$ with $z \le q <y$.\par
We now let $x_i$, $K_{\delta}$, $k_i$ and $Q_1$ be as in Section \ref{notationsect}.
For fixed $i\in\{1,2\}$, if $k_i\ge K_{\delta}(x_i)$ we let $q_1> \ldots> q_\ell$ be all prime factors of $k_i$ and set $m_0 = 1$, $P^0(x) := P(x)$ and for $1 \le j \le \ell$,
\begin{equation} \label{eq: def-mj-Aj-Pj(x)}
    m_j:=q_1\cdots q_j,\quad A^{(j)}:=A_{m_j},\quad 
    P^{(j)}(x):=\prod_{\substack{p<x,~ p\nmid N,\\  p\neq q_1, \ldots, q_j}}p.
\end{equation}
In relation to the function $V(z)$ from the explicit linear sieve (see \eqref{vzdef}), we then define 
\begin{equation} \label{eq: def-V(x)-Vj(x)}
V(x):=\prod_{p|P(x)} \left(1-\frac{1}{p-1}\right),\quad
V^{(j)}(x):=\prod_{p|P^{(j)}(x)} \left(1-\frac{1}{p-1}\right),
\end{equation}
for $1 \le j \le \ell$ and set $V^0(x) = V(x)$. To bound $V^{(j)}(x)$ we use the following result.
\begin{lemma}\label{vjlem}
    For $x\ge 285$ and $j=0,\ldots,\ell$, we have
    \begin{align} \label{eq: Vj-UNj}
        V^{(j)}(x) &= \frac{U_N^{(j)}}{\log x} \left[ 1 + 1.45~\theta_1(x) \frac{\log N}{x-1} \left( 1 + \frac{10 \log \log N}{\log N} \right) \right]\notag\\
        &\qquad\qquad\qquad\qquad\qquad\qquad\cdot\left( 1 + \frac{1.002~\theta_2(x)}{x-3} \right) \left( 1 + \frac{\theta_3(x)}{2 \log^2 x} \right),
    \end{align}
    where $|\theta_i|\le 1$, $i = 1, 2, 3$, and\footnote{In the analogous (non-explicit) theorem by Nathanson \cite[Theorem 10.3]{Nathanson}, there is a small typo. Namely, a factor of $2$ is missing from $U_N$ ($\mathfrak{G}(N)$ in Nathanson's notation).  }
    \begin{equation} \label{eq: def-UNj}
        U^{(j)}_N:=2e^{-\gamma}\prod_{p>2} \left(1-\frac{1}{(p-1)^2} \right) \prod_{\substack{p>2\\ p\mid Nm_j}}\frac{p-1}{p-2},
    \end{equation}
    so that $U^{(0)}_N=U_N$ as defined by \eqref{UNeq}. In particular, when $x=z=N^{1/8}\ge\exp(20)$,
    \begin{equation}\label{vjzbound}
        \frac{U_N^{(j)}}{\log z}\left(1-\frac{32.02}{\log^2 N}\right)<V^{(j)}(z)<\frac{U_N^{(j)}}{\log z}\left(1+\frac{32.02}{\log^2 N}\right)
    \end{equation}
    and when $x=y=N^{1/3}\ge\exp(20)$, we have
    \begin{equation}\label{vjybound}
        \frac{U_N^{(j)}}{\log y}\left(1-\frac{4.51}{\log^2 N}\right)<V^{(j)}(y)<\frac{U_N^{(j)}}{\log y}\left(1+\frac{4.51}{\log^2 N}\right).
    \end{equation}
\end{lemma}
\begin{proof}
We follow the argument in the proof of \cite[Theorem 10.3]{Nathanson}. Let
\begin{equation*}
    W(x) = \prod_{2 < p < x} \left( 1 - \frac{1}{p-1} \right), 
\end{equation*}
then
\begin{align}
    \frac{V^{(j)}(x)}{W(x)} &= \prod_{\substack{2 < p < x \\ p | N m_j}} \left( 1 - \frac{1}{p-1} \right)^{-1} \nonumber \\
    &= \prod_{\substack{p > 2 \\ p | N m_j}} \left( 1 - \frac{1}{p-1} \right)^{-1} \prod_{\substack{p \ge x \\ p | N m_j}} \left( 1 - \frac{1}{p-1} \right) \nonumber\\
    &= \prod_{\substack{p > 2 \\ p | N m_j}} \frac{p-1}{p-2} \prod_{\substack{p \ge x \\ p | N m_j}} \left( 1 - \frac{1}{p-1} \right). \label{eq: Nathanson-Vjz-1}
\end{align}
Let us estimate the second product in \eqref{eq: Nathanson-Vjz-1}. To do so, we note that 
\begin{align}
    p - 1 &\ge x - 1 \ge 284,\notag\\
    1 - t &> \exp(-1.002 t),\qquad\text{for $0<t\leq 1/284$},\label{exp1002eq}\\
    1 - t &\le \exp(-t),\qquad\text{for all $t\in\mathbb{R}$.}\notag
\end{align}
Hence,
\begin{align*}
    \prod_{\substack{p \ge x \\ p | N m_j}} \left( 1 - \frac{1}{p-1} \right) &> \exp \left( -1.002 \sum_{\substack{p \ge x \\ p | N m_j}} \frac{1}{p-1} \right) \\
    &\ge \exp \left( -1.002~\frac{\omega(N m_j)}{x-1} \right).
\end{align*}
By the definition \eqref{eq: def-mj-Aj-Pj(x)} of $m_j$ and conditions \eqref{notcon1} and \eqref{notcon2} on $k_i$, we have ${m_j \le k_i \le \log^{10} N}$. Thus by Lemma \ref{omegalem}, we have $$\omega(N m_j) \le \frac{\log N + 10 \log \log N}{\log 2},$$ and we can continue the chain of inequalities as follows:
\begin{align*}
    \exp \left( -1.002~\frac{\omega(N m_j)}{x-1} \right) &\ge \exp \left( -1.002~\frac{\log N + 10 \log \log N}{(\log 2)(x - 1)} \right) \\
    &\ge 1 - 1.45 \frac{\log N + 10 \log \log N}{x - 1}.
\end{align*}
To summarise,
\begin{equation} \label{eq: Vj-over-W}
\frac{V^{(j)}(x)}{W(x)} = \left[ 1 + 1.45~\theta_1(x) \frac{\log N}{x-1} \left( 1 + \frac{10 \log \log N}{\log N} \right) \right] \prod_{\substack{p > 2 \\ p | N m_j}} \frac{p-1}{p-2},
\end{equation}
where $|\theta_1(x)| \le 1$.

Now,
\begin{align*}
    W(x) \prod_{p < x} \left( 1 - \frac{1}{p}\right)^{-1} &= 2 \prod_{2 < p < x} \left( 1 - \frac{1}{(p-1)^2} \right)\\
    &= 2 \prod_{p > 2} \left( 1 - \frac{1}{(p-1)^2} \right) \prod_{p \ge x} \left( 1 + \frac{1}{p(p-2)} \right) \\
    &\le 2 \prod_{p > 2} \left( 1 - \frac{1}{(p-1)^2} \right) \exp \left( \sum_{p \ge x} \frac{1}{p(p-2)} \right),
\end{align*}
where we have used that $1+x\leq \exp(x)$ for all $x\in\mathbb{R}$. Next we note that
\begin{align*}
    0 \le \sum_{p \ge x} \frac{1}{p(p-2)} \le \sum_{n \ge x-2} \frac{1}{n^2} \le \frac{1}{(x-3)},
\end{align*}
whence, by \eqref{exp1002eq},
\begin{align*}
    W(x) \prod_{p < x} \left( 1 - \frac{1}{p}\right)^{-1} = 2 \prod_{p > 2} \left( 1 - \frac{1}{(p-1)^2} \right) \left( 1 + \frac{1.002~\theta_2(x)}{x-3} \right),
\end{align*}
with $|\theta_2(x)| \le 1$. Using an explicit form of Mertens' third theorem \cite[Theorem 7]{R-S} then yields
\begin{equation}\label{finalWeq}
    W(x) = \frac{2e^{-\gamma}}{\log x} \prod_{p > 2} \left( 1 - \frac{1}{(p-1)^2} \right) \left( 1 + \frac{\theta_3(x)}{2 \log^2 x} \right) \left( 1 + \frac{1.002~\theta_2(x)}{x-3} \right),
\end{equation}
for all $x\geq 285$, with $|\theta_3(x)| \leq 1$. The expressions \eqref{finalWeq} and \eqref{eq: Vj-over-W} imply \eqref{eq: Vj-UNj}.
\end{proof}
We also have the following bounds relating to $\ell$, $q_j$ and $U_N^{(j)}$.
\begin{lemma}\label{elllem}
If $\ell$ primes divide $k_i$ then
\begin{equation}
\label{eq:l}
\ell \le \frac{1.3841 \log (\log^{10} x_i)}{\log \log (\log^{10} x_i)}.
\end{equation}
\end{lemma}
\begin{proof}
The result follows from \eqref{eq:w(n)} using that $k_i \le Q_1(x_i)=\log^{10} x_i$.
\end{proof}
\begin{lemma}\label{leq1lem}
    If $\ell\ge 2$,
    \begin{equation*} 
        \frac{1}{q_2-2}+\frac{1}{(q_2-2)(q_3-2)}+\cdots+\frac{1}{(q_2-2)(q_3-2)\cdots(q_\ell-2)}\le 1.
    \end{equation*}
\end{lemma}
\begin{proof}
    Since $N$ is even, and $(k_i, N) = 1$, $k_i$ and all its prime factors are odd. Thus $q_2>q_3>\cdots>q_{\ell}$ is a decreasing set of odd numbers, so we have $q_2\ge 2\ell-1$. Thus,
    \begin{align*}
        \frac{1}{q_2-2}+\frac{1}{(q_2-2)(q_3-2)}+\cdots+\frac{1}{(q_2-2)(q_3-2)\cdots(q_\ell-2)}&\le(\ell-1)\frac{1}{q_2-2}\\
        &\le\frac{\ell-1}{2\ell-3}\\
        &=\frac{1}{2}+\frac{1}{4\ell-6}\\
        &\le 1
    \end{align*}
\end{proof}
\begin{lemma}\label{e0lem}
We recall that $0 < \delta < 2$ and $K_{\delta}(x) = \log^\delta x$. Suppose $N$ is a positive even integer with $N \ge X_2$, and $k_i\ge K_{\delta}(x_i(X_2)) \ge 3$ and 
\begin{equation} \label{eq: def-epsilon0-X2-delta}
    \varepsilon_i(X_2,\delta):=\frac{1}{\overline{p}-2},    
\end{equation}
with $\overline{p}$ the largest prime such that
    \begin{equation*}
        \log^{\delta}x_i(X_2)\ge \prod_{2<p\le\overline{p}}p.
    \end{equation*}
    We then have 
    \begin{equation*}
        U_N^{(1)}\le U_N(1+\varepsilon_i(X_2,\delta))
    \end{equation*}
    where $U_N^{(1)}$ is defined in \eqref{eq: def-UNj}.
\end{lemma}
\begin{proof}
    Since $k_i\ge\log^\delta x_i(X_2)$,
    \begin{equation*}
        k_i\ge\prod_{2<p\le\overline p}p.
    \end{equation*}
    Then, since $k_i$ is odd and square-free, this means that $q_1\ge\overline{p}$. Thus
    \begin{equation*}
        U_N^{(1)} =  U_N\frac{q_1-1}{q_1-2}\le U_N(1+\varepsilon_i(X_2,\delta)),
    \end{equation*}
    as required.
\end{proof}
In relation to the remainder term \eqref{Rdef} appearing in the linear sieve, we now define
\begin{equation} \label{eq: def-r(d)-rk(d)}
r(d):=\left|A_d\right|-\frac{\left|A\right|}{\varphi(d)}\quad \text{and}\quad
r_k(d):=\left|A_{kd}\right|-\frac{\left|A_k\right|}{\varphi(d)},
\end{equation}
with $A$ and $A_d$ defined in \eqref{Aapeq}. By \eqref{AAqeq}, $r(d)$ and $r_k(d)$ can be expressed as
\begin{align}
    r(d)&=\pi(N;d,N)-\omega(N;d,N)-\frac{\pi(N)-\omega(N)}{\varphi(d)}\label{rddef}\\
    r_k(d)&=\pi(N;kd,N)-\omega(N;kd,N)-\frac{\pi(N;k,N)-\omega(N;k,N)}{\varphi(d)}\label{rkddef}.
\end{align}
This leads us to the following estimates.

\begin{lemma}
\label{lemma:r}
Let $N$ be a positive even integer with $N\ge X_2$ and $\log\log(x_1(X_2))\ge 10.4$. We have
\begin{equation}
\label{eq:A>}
    \frac{N}{\log N}<\left|A\right|<1.00005\frac{N}{\log N}.
\end{equation}
If $k_1\ge K_{\delta}(x_1)$ then for $j=0,1,\ldots,l-1$ we have
\begin{equation}
\label{eq:rj}
\left|r(m_j)\right|<\frac{c_2(X_2)N}{\log^3N},
\end{equation}
with 
\begin{equation} \label{eq: def-c-2}
c_2(X_2):=c(X_2)+\frac{1.3841\log^4 X_2}{X_2\log\log X_2},
\end{equation}
where $c(X_2)$ defined in Lemma~\ref{lem: general-5.3.2}.
We also have
\begin{equation}
\label{eq:rl}
\left|r(m_l)\right|\le\frac{c_3(X_2) N \log \log \log N}{\log^{1+\delta}N},
\end{equation}
with
\begin{align}
    c_3(X_2)&:=\max_{N\ge X_2}\Bigg[\frac{1}{\log\log\log N}\cdot\left(\frac{3}{2\log N}+\frac{\log(N\log^{10}x_1)}{\log(N/\log^{10}x_1)}\right)\frac{\log^\delta N}{\log^\delta x_1}\label{eq: def-c-3}\\
    &\qquad\qquad\cdot\left(e^\gamma\log\log\log^\delta x_1+\frac{5}{2\log\log\log^{\delta}x_1}\right)\notag\\
    &\qquad\qquad+\frac{1.3841\log^{2+\delta}N}{N\log\log N\log\log\log N}\Bigg].\notag
\end{align}
\end{lemma}
\begin{proof}
By \eqref{AAqeq} and Lemma \ref{omegalem}, we have
$$\pi(N) - \frac{1.3841 \log N}{\log \log N} \le |A| \le \pi(N),$$
where we used that $\frac{\log N}{\log \log N}$ increases for $\log \log N \ge 1$. Hence, by \cite[Theorem 2]{R-S},
\begin{equation}
    \frac{N}{\log N - \frac{1}{2}} - \frac{1.3841 \log N}{\log \log N} < |A| < \frac{N}{\log N - \frac{3}{2}},
\end{equation}
which implies \eqref{eq:A>} for $\log \log x_1(N) \ge 10.4$.\par
Now let us prove \eqref{eq:rj}. By \eqref{rddef}, and Lemmas \ref{omegalem} and \ref{lem: general-5.3.2}, we get
    \begin{align}
        r(m_j)&=\left||A_{m_j}|-\frac{|A|}{\varphi(m_j)}\right| \nonumber\\
        &\le\left|\pi(N;m_j,N)-\frac{\pi(N)}{\varphi(m_j)}\right|+\left|\omega(N;m_j,N)-\frac{\omega(N)}{\varphi(m_j)}\right| \nonumber\\
        &\le|E_{\pi}(N; m_j, N)|+\max\left\{\omega(N;m_j,N), \frac{\omega(N)}{\varphi(m_j)}\right\} \nonumber\\
        &\le |E_{\pi}(N; m_j, N)| + \omega(N) \nonumber\\
        &\le |E_{\pi}(N; m_j, N)| + \frac{1.3841 \log N}{\log \log N} \label{eq: bound-rmj}\\
        &\le \frac{N}{\log^3 N} \left( c(X_2) + \frac{1.3841 \log^4 N}{N \log \log N} \right), \nonumber
    \end{align}
    and we conclude \eqref{eq:rj} since the function $\frac{\log^4 N}{N \log \log N}$ decreases for $N \ge \exp(\exp(10.4))$.
    To prove \eqref{eq:rl}, we use \eqref{eq: bound-rmj} with $j = \ell$ and bound $|E_{\pi}(N; m_{\ell}, N)|$ with a bit more care. First, by an explicit form of the Brun--Titchmarsh theorem \cite[Theorem 2]{M-V} we have
    \begin{equation*}
        \pi(N;m_\ell,N)<\frac{2N}{\varphi(m_\ell)\log(N/m_\ell)}=\frac{N}{\varphi(m_\ell)\log N}+\frac{\log(Nm_\ell)}{\log(N/m_\ell)}\cdot\frac{N}{\varphi(m_\ell)\log N}.
    \end{equation*}
    Thus, noting $\pi(N;m_\ell,N)\ge 0$,
    \begin{equation*}
        \pi(N;m_\ell,N)=\frac{N}{\varphi(m_\ell)\log N}+\varepsilon\frac{\log(Nm_\ell)}{\log(N/m_\ell)}\cdot\frac{N}{\varphi(m_\ell)\log N}
    \end{equation*}
    for some $|\varepsilon|\le 1$. Similarly, by \cite[Theorem 1]{R-S}
    \begin{equation*}
        \frac{\pi(N)}{\varphi(m_{\ell})}=\frac{N}{\varphi(m_\ell)\log N}+\varepsilon' \frac{3N}{2 \varphi(m_{\ell})\log^2 N}
    \end{equation*}
    for some $|\varepsilon'| \le 1$. Therefore,
    \begin{equation*}
        |E_{\pi}(N; m_{\ell}, N)| = \left|\pi(N;m_\ell,N)-\frac{\pi(N)}{\varphi(m_\ell)}\right|\le\frac{N}{\varphi(m_\ell)\log N}\left(\frac{3}{2\log N}+\frac{\log(Nm_\ell)}{\log(N/m_\ell)}\right).
    \end{equation*}
    Finally, since $k_1$ is square-free and by \eqref{eq: def-mj-Aj-Pj(x)}, we get that $m_\ell=k_i$ is odd, hence by \cite[Theorem 15]{R-S}
    \begin{equation*}
        \frac{1}{\varphi(m_\ell)}<\frac{e^\gamma\log\log m_\ell}{m_\ell}+\frac{5}{2m_\ell\log\log m_\ell}.
    \end{equation*}
    Combining all the above estimates with \eqref{eq: bound-rmj} for $j = \ell$ and using $m_\ell=k_i\in[\log^\delta x_1,\log^{10}x_1]$ gives the desired result.
\end{proof}

\begin{lemma}\label{544small}
    Let $N$ be a positive even integer with $N \ge X_2$ and $\log\log x_1(X_2)\ge 10.4$, $H=H(N)=\frac{\sqrt{x_1}}{\log^{10}x_1}$, and suppose $p(X_2)$, $p^*(X_2)$ are as in Lemmas \ref{small533} and \ref{large533}. We have
    \begin{equation}\label{unsharpeq}
        \sum_{\substack{d<H\\d\mid P(z)}}|r(d)|<\frac{c_4(X_2)N}{\log^3 N}
    \end{equation}
    and if $k_1\ge K_{\delta}(x_1)$
    \begin{equation}\label{sharpeq}
         \sideset{}{^\sharp}\sum_{\substack{d<H/m_j}}|r_{m_j}(d)|<\frac{c_4^*(X_2)N}{\log^3N},
    \end{equation}
    for all $1 \le j \le \ell$, where
    \begin{align}
        c_4(X_2)&:=p(X_2)+\frac{0.9\sqrt{x_1(X_2)}\log^{4}X_2}{X_2\log^{10}(x_1(X_2))\log\log X_2},\label{eq: def-c-4}\\
        c_4^*(X_2)&:=p^*(X_2)+\frac{0.9\sqrt{x_1(X_2)}\log^{4}X_2}{X_2\log^{10}(x_1(X_2))\log\log X_2} \label{eq: def-c-4-*}
    \end{align}
    and the $\sharp$ means that the sum is over $d\mid P^{(j+1)}(z)$ if $j<\ell$ and $d\mid P^{(\ell)}(z)$ if $j=\ell$, with $P^{(j)}$ defined in \eqref{eq: def-mj-Aj-Pj(x)}.
\end{lemma}
\begin{proof}
    First, we prove \eqref{unsharpeq}. As in \eqref{eq: bound-rmj},
    \begin{equation*}
        |r(d)|\le|E_{\pi}(N; d, N)| + \frac{1.3841 \log N}{\log \log N}.\label{rdineq}
    \end{equation*}
    Thus by Lemmas \ref{small533} and \ref{squarefreelem},
    \begin{equation*}
        \sum_{\substack{d<H\\d\mid P(z)}}|r(d)|<\frac{p(X_2)N}{\log^3 N}+0.65H\cdot\frac{1.3841\log N}{\log\log N},
    \end{equation*}
    which gives the required result. The proof of \eqref{sharpeq} is essentially the same, using Lemma \ref{large534} in place of Lemma \ref{small533}.
\end{proof}

\section{A lower bound for $S(A,P(z))$}
\label{section:1}
In this section, we obtain a lower bound for $S(A,P(z))$. This is the first term appearing in the bound for $\pi_2(N)$ in Lemma \ref{lem:>}. We recall that functions $f(x)$ and $F(x)$ are defined in \eqref{def-f(s)-F(s)}, $x_1(N)$ defined in \eqref{notcon1}, and $c_2(X_2)$, $c_3(X_2)$ defined in \eqref{eq: def-c-2} and \eqref{eq: def-c-3}. We introduce the following notations: for $\alpha_1 > 0$, set
\begin{align}
&c_{\alpha_1,X_2}:=4-8\alpha_1-\frac{160\log\log X_2}{\log X_2}, \label{eq: def-c-alpha1-X2}\\
&\overline{m}_{\alpha_1,X_2}:=\max\{(1-f(c_{\alpha_1,X_2}),F(c_{\alpha_1,X_2})-1)\},\label{eq: def-m-bar-alpha1-X2}\\
&a(X_2):=a_1(X_2)\max_{N\ge X_2}\left[\frac{\log\log\log N}{\log^{\delta}N}\cdot\prod_{p>2}\frac{(p-1)^2}{p(p-2)}\right.\label{eq: def-a(X_2)}\\
&\qquad\qquad\qquad\qquad\qquad\left.\cdot\left(e^{\gamma}\log\log(\log^{10}x_1(N))+\frac{2.5}{\log\log(\log^{10}x_1(N))}\right)\right], \nonumber\\
&a_1(X_2):=\max_{N\ge X_2}\left[\frac{c_2(X_2)}{\log^{2-\delta}N\log\log\log N}\cdot\frac{1.3841\log(\log^{10}x_1(N))}{\log\log(\log^{10}x_1(N))}\right]+c_3(X_2).\label{eq: def-a_1(X_2)}
\end{align}

\begin{theorem}
\label{theo:S>1}
Let $u_0=10^9$ and $\varepsilon=1.452\cdot 10^{-7}$ be the corresponding values in Lemma~\ref{lem: prod-u>u_0}. Recall that $K_{\delta}(x) = \log^\delta(x)$ for $0 < \delta < 2$.
Let $X_2$ be such that $\log\log x_1(X_2)\ge 10.4$, assume that $\alpha_1>0$, $N\ge X_2$ is an even integer, and $z = N^{1/8}$ such that
\begin{align}
&\frac{N^{\alpha_1}}{\log^{10} x_1(N) \log^{2.5} N}\ge \exp\left(u_0\left(1+\frac{9\cdot 10^{-7}}{\log u_0} \right)\right), \quad \quad \frac{N^{\frac{1}{2}-\alpha_1}}{\log^{20} N}\ge z^2 \nonumber \\
&8\alpha_1+\frac{160\log \log N}{\log N}<1, \quad \quad K_{\delta}(x_1(X_2)) \ge 3022.\label{cond: S-A-P-z}
\end{align} 
Let $A$, $S(A, n)$, and $P(z)$ be defined by \eqref{Aapeq}, \eqref{eq: def-S-A-n}, and \eqref{eq: def-P(x)} respectively. Assume $U_N = U_N^{(0)}$ is defined in \eqref{eq: def-UNj}, $h(x)$ in \eqref{eq-def-h(s)}, $\varepsilon_1(X_2, \delta)$ in \eqref{eq: def-epsilon0-X2-delta}, $c_4(X_2)$ and $c_4^*(X_2)$ in Lemma \ref{544small}, and let $C_1(\varepsilon)=106$ and $C_2(\varepsilon)=107$ be the values from Table \ref{tab:fF}. Let \label{def-C-upper-bar} $\overline{C}(\varepsilon)=\max\{C_1(\varepsilon),C_2(\varepsilon)\}$.
For $k_1$ as defined in \eqref{eq:k11}, we consider two cases.\newline
(a) If $k_1<K_{\delta}(x_1(N))$, we have
\begin{align*}
    S(A,P(z))&>8\frac{|A| U_N}{\log N}\left(1-\frac{32.02}{\log^2 N}\right)\Bigg\{\frac{2e^{-\gamma}\log(3-8\alpha_1)}{4-8\alpha_1}-C_2(\varepsilon)\varepsilon e^2h(4-8\alpha_1)\nonumber\\
    &\qquad\qquad-\frac{1}{8}\left(1-\frac{32.02}{\log^2 N}\right)^{-1}\left(2e^{-\gamma}\prod_{p>2}\left(1-\frac{1}{(p-1)^2}\right)\right)^{-1}\frac{c_4(X_2)}{\log N}\Bigg\}.
\end{align*}
(b) If $k_1\ge K_{\delta}(x_1(N))$, we have
\begin{align*}
    &S(A,P(z))>\\
    &\ 8\frac{|A| U_N}{\log N}\left(1+\frac{32.02}{\log^2N}\right)\Bigg\{\frac{2e^{\gamma}\log(3-8\alpha_1-\frac{160\log\log X_2}{\log X_2})}{4-8\alpha_1-\frac{160\log\log X_2}{\log X_2}}-\varepsilon_1(X_2,\delta)(1-f(c_{\alpha_1,X_2})) \\
    &\ -(1+\varepsilon_1(X_2,\delta))\varepsilon C_2(\varepsilon)e^2h(c_{\alpha_1,X_2}) \\
    &\ -\left(3\varepsilon_1(X_2,\delta)+ a(X_2) \right)\cdot(\overline{m}_{\alpha_1,X_2}+\varepsilon \overline{C}(\varepsilon)e^2h(c_{\alpha_1,X_2}))-a(X_2)-\frac{64.04}{\log^2N}\\
    &\ -\frac{1}{8}\left(1+\frac{32.02}{\log^2 N}\right)^{-1}\left(2e^{-\gamma}\prod_{p>2}\left(1-\frac{1}{(p-1)^2}\right)\right)^{-1}\frac{c_4^*(X_2)}{\log N}\frac{1.3841\log(\log^{10}x_1(N))}{\log\log(\log^{10}x_1(N))}\Bigg\},
\end{align*}
with $c_2(X_2)$ and $c_3(X_2)$ as in Lemma~\ref{lemma:r}.
\end{theorem}
\begin{remark}
    The constant $$\prod_{p>2}\left(1-\frac{1}{(p-1)^2}\right)=\prod_{p>2}\frac{p(p-2)}{(p-1)^2}=0.66016\ldots,$$ is called the {\slshape twin prime constant}.
\end{remark}

\begin{proof}[Proof of Theorem \ref{theo:S>1} in case (a).]
Assume $k_1<K_{\delta}(x_1)$. In this case, we set
\begin{equation} \label{eq: TheoA-def-D-s-Q}
    D^{(1)}:=N^{\frac{1}{2}-\alpha_1},\ s^{(1)}:=\frac{\log D^{(1)}}{\log z}=4-8\alpha_1\quad\text{and}\quad Q(u):=\prod_{\substack{p< u\\p\nmid N}}p.
\end{equation}
We note that $3 \le s^{(1)} \le 4$ since $\alpha_1 > 0$ by definition and $\alpha_1 < \frac{1}{8}$ by \eqref{cond: S-A-P-z}.

We will apply Theorem \ref{theo:JR} to the set $A$ with $\mathbb{P}$ the set of primes coprime to $N$, $g_n(p)=1/(p-1)$, $Q=Q(u_0)$, $D=D^{(1)}$, and $s = s^{(1)}$. Setting parameters in this way implies that $S(A, P(z)) = S(A, \mathbb{P}, z)$ from Theorem \ref{theo:JR}. Then $D\ge z^2$ follows from the condition $\frac{N^{\frac{1}{2}-\alpha_1}}{\log^{20} N}\ge z^2$ assumed in the first line of \eqref{cond: S-A-P-z}, and thus
\begin{align}\label{smallsieveeq-prep}
    S(A,P(z))&>(f(s^{(1)})-\varepsilon C_{2}(\varepsilon) e^2h(s^{(1)}))|A|V(z)- \sum_{\substack{d\mid P(z)\\ d<QD^{(1)}}}|r(d)| \nonumber \\
    &>\frac{8|A|U_N}{\log N}\left(1-\frac{32.02}{\log^2N}\right)(f(s^{(1)})-C_2(\varepsilon)\varepsilon e^2h(s^{(1)}))-\sum_{\substack{d\mid P(z)\\ d<QD^{(1)}}}|r(d)|,
\end{align}
where we used Theorem \ref{theo:JR} in the first line and \eqref{vjzbound} in the second line. By the definition \eqref{eq: approx-f24} of $f(s^{(1)})$ for $2 \le s^{(1)} \le 4$, the second line of \eqref{smallsieveeq-prep} coincides with
\begin{equation} \label{smallsieveeq}
    8\frac{|A| U_N}{\log N}\left(1-\frac{32.02}{\log^2 N}\right)\Bigg(\frac{2e^{\gamma}\log(3-8\alpha_1)}{4-8\alpha_1}-C_2(\varepsilon)\varepsilon e^2h(4-8\alpha_1)\Bigg) -\sum_{\substack{d\mid P(z)\\ d<QD^{(1)}}}|r(d)|.
\end{equation}

We remark that the condition 
\begin{equation*}
    \frac{N^{\alpha_1}}{\log^{10}x_1(N)\log^{2.5}N}\ge\exp\left(u_0\left(1+\frac{9\cdot 10^{-7}}{\log u_0}\right)\right)
\end{equation*}
implies that
\begin{equation}\label{qeq}
    Q\le \frac{N^{\alpha_1}}{\log^{10}x_1(N)\log^{2.5}N}
\end{equation}
by Lemma \ref{thetauplem}. As a result, $QD^{(1)}\le H=\frac{\sqrt{x_1}}{\log^{10}x_1}$ so that we may apply Lemma \ref{544small} (in particular the bound \eqref{unsharpeq}) to the error term in \eqref{smallsieveeq}:
\begin{align}
    \sum_{\substack{d\mid P(z)\\ d<QD^{(1)}}}|r(d)| &< \frac{c_4(X_2) N}{\log^3 N} = \frac{U_N N}{\log^3 N} \cdot U_N^{-1} \cdot c_4(X_2) \nonumber\\
    &< \frac{|A| U_N}{\log^2 N} \left(2e^{\gamma}\prod_{p>2}\left(1-\frac{1}{(p-1)^2}\right)\right)^{-1} c_4(X_2), \label{eq: ThmA-case-a-r(d)}
\end{align}
where in the last line we used the definition \eqref{eq: def-UNj} of $U_N$ and the condition $|A|>N/\log N$, see \eqref{eq:A>}. 

Combining \eqref{smallsieveeq-prep}, \eqref{smallsieveeq}, and \eqref{eq: ThmA-case-a-r(d)} proves the theorem in case (a).
\end{proof}

In case (b), when $k_1\ge K_{\delta}(x_1)$, we need to apply an inclusion-exclusion argument which in essence allows us to avoid the large exceptional zero.

\begin{lemma}\label{exinclem}
Keep the notations from Theorem \ref{theo:S>1}. We have
    \begin{equation*}
        S(A,P(z))=\sum_{j=0}^{\ell-1}(-1)^j S(A^{(j)},P^{(j+1)}(z))+(-1)^{\ell} S(A^{(\ell)},P^{(\ell)}(z)),
    \end{equation*}
    with $A^{(j)}$ and $P^{(j)}(z)$ defined in \eqref{eq: def-mj-Aj-Pj(x)}.
\end{lemma}
\begin{proof}
    First note that $S(A,P^{(1)}(z))-S(A,P(z))$ counts the number of integers in $A$ that are divisible by $q_1$ but not by any other primes below $z$. Then, $S(A^{(1)},P^{(2)}(z))-S(A,P^{(1)}(z))+S(A,P(z))$ counts the number of integers in $A$ that are divisible by $q_1$ and $q_2$ but not any other primes less than $z$. By generalising this argument, we have that $\sum_{j=0}^{\ell-1}(-1)^{\ell-1-j}S(A^{(j)},P^{(j+1)}(z))+(-1)^{\ell}S(A,P(z))$ counts the number of integers in $A$ divisible by $q_1,\ldots,q_{\ell}$ but no other primes less than $z$. That is,
    \begin{equation*}
        S(A^{(\ell)},P^{(\ell)}(z))=\sum_{j=0}^{\ell-1}(-1)^{\ell-1-j}S(A^{(j)},P^{(j+1)}(z))+(-1)^{\ell}S(A,P(z)),
    \end{equation*}
    which rearranges to give the desired result.
\end{proof}

We now bound $S(A^{(j)},P^{(j+1)}(z))$ and $S(A^{(\ell)},P^{(\ell)}(z))$.

\begin{lemma}\label{linearlem}
Keep the notation from the beginning of Section \ref{section:1} and Theorem \ref{theo:S>1}. Let
    \begin{equation*}
        D^{(1)}_j:=\frac{N^{\frac{1}{2}-\alpha_1}}{k_1m_j}\quad\text{and}\quad s^{(1)}_j:=\frac{\log D^{(1)}_j}{\log z}.
    \end{equation*}
    for $j=0,\ldots,\ell$. Let $Q(u)$ be as in \eqref{eq: TheoA-def-D-s-Q}, $r_k(d)$ be as in \eqref{eq: def-r(d)-rk(d)}, and
    \begin{equation*}
        E_j:=
        \begin{cases}
            \sum_{d\mid P^{(j+1)}(z), \:d<D^{(1)}_jQ(u_0)}|r_{m_j}(d)|,&\text{if $j=0,\ldots,\ell-1$},\\
            \sum_{d\mid P^{(\ell)}(z), \:d<D^{(1)}_\ell Q(u_0)}|r_{m_{\ell}}(d)|,&\text{if $j=\ell$}.
        \end{cases}
    \end{equation*}
    Provided that each $D^{(1)}_j\ge z^2$,
    \begin{align*}
        &\left|A^{(j)}\right|\left[V^{(j+1)}(z)-8U_N^{(j+1)}\left(1+\frac{32.02}{\log^2N}\right)\frac{(1-f(s^{(1)}_j))+\varepsilon C_2(\varepsilon)e^2h(s_j^{(1)})}{\log N}\right]-E_j\\
        &\qquad\qquad\qquad\qquad\qquad <S(A^{(j)},P^{(j+1)}(z))\\
        &<\left|A^{(j)}\right|\left[V^{(j+1)}(z)+8U_N^{(j+1)}\left(1+\frac{32.02}{\log^2N}\right)\frac{(F(s^{(1)}_j)-1)+\varepsilon C_1(\varepsilon)e^2h(s_j^{(1)})}{\log N}\right]+E_j
    \end{align*}
    for $j=0,\ldots,\ell-1$, and
    \begin{align*}
        &\left|A^{(\ell)}\right|\left[V^{(\ell)}(z)-8U_N^{(\ell)}\left(1+\frac{32.02}{\log^2N}\right)\frac{(1-f(s_{\ell}^{(1)}))+\varepsilon C_2(\varepsilon)e^2h(s_\ell^{(1)})}{\log N}\right]-E_{\ell}\\
        &\qquad\qquad\qquad\qquad\qquad <S(A^{(\ell)},P^{(\ell)}(z))\\
        &<\left|A^{(\ell)}\right|\left[V^{(\ell)}(z)+8U_N^{(\ell)}\left(1+\frac{32.02}{\log^2N}\right)\frac{(F(s_{\ell}^{(1)})-1)+\varepsilon C_1(\varepsilon)e^2h(s_{\ell}^{(1)})}{\log N}\right]+E_{\ell},
    \end{align*}
    where $A^{(j)}$, $V^{(j)}(z)$, and $U^{(j)}_N(z)$ are defined in \eqref{eq: def-mj-Aj-Pj(x)}, \eqref{eq: def-V(x)-Vj(x)}, and \eqref{eq: def-UNj} respectively.
\end{lemma}
\begin{proof}
    We only prove the lower bound for $j=0,\ldots,\ell-1$ as the proof for the other cases will follow by almost identical reasoning. We apply Theorem \ref{theo:JR} to the set $A^{(j)}$, with $\mathbb{P}$ the set of primes coprime to $N$, $g_n(p)=1/(p-1)$, $Q=Q(u_0)$, and $D=D^{(1)}_j$. With $P^{(j)}(z)$ as in \eqref{eq: def-mj-Aj-Pj(x)}, we thus have
    \begin{align*}
        S(A^{(j)}, P^{(j+1)}(z)) &> |A^{(j)}| V^{(j+1)}(z) \left( f(s^{(1)}_j) - \varepsilon C_2(\varepsilon) e^2 h(s) \right) - E_j \\
        &= \left| A^{(j)} \right| \left( V^{(j+1)}(z) - V^{(j+1)}(z) \cdot \left( 1 - f(s^{(1)}_j) + \varepsilon C_2(\varepsilon) e^2 h(s) \right) \right)-E_j.
    \end{align*}
    Next, we note that
    \begin{equation*}
        s^{(1)}_j=\frac{\log(D^{(1)}_j)}{\log z}=4-8\alpha_1-\frac{8\log(k_1) + 8\log(m_j)}{\log N}
    \end{equation*}
    so by \eqref{eq: def-c-alpha1-X2},
    \begin{equation*}
        c_{\alpha_1,N}=4-8\alpha_1-\frac{160\log\log N}{\log N}\le s^{(1)}_j<4-8\alpha_1
    \end{equation*}
    since $m_j\le k_1\le\log^{10}N$ by the definition \eqref{eq:k11} of $k_1$ and the definition \eqref{eq: def-mj-Aj-Pj(x)} of $m_j$. From the condition $0 < \alpha_1 < \frac{1}{8}$, which follows from \eqref{cond: S-A-P-z}, we have $3\le s^{(1)}_j\le 4$ and hence by \eqref{eq: approx-f24}
    \begin{equation*}
        1-f(s^{(1)}_j)=1-\frac{2e^{\gamma}\log(s^{(1)}_j-1)}{s^{(1)}_j}>0.
    \end{equation*}
    The result then follows by applying the upper bound in \eqref{vjzbound} in Lemma \ref{vjlem}.
\end{proof}
By the definition of $D_j^{(1)}$ from Lemma \ref{linearlem} and \eqref{qeq}, we have $Q D_j^{(1)} \le H/m_j$, and thus we can apply Lemma \ref{544small} to get the following bound for $E_j$ (with $E_j$ defined in Lemma \ref{linearlem}).
\begin{lemma}\label{errorlem}
With the notations from Lemma \ref{linearlem}, we have for $j=0,\ldots,\ell$,
    \begin{equation*}
        E_j<\frac{c_4^*(X_2)N}{\log^3N}.
    \end{equation*}
\end{lemma}
~\newline
\begin{proof}[Proof of Theorem \ref{theo:S>1} in case (b).]
We assume that $k_1\ge K_{\delta}(x_1)$.
Combining Lemmas \ref{elllem}, \ref{exinclem}, \ref{linearlem} and \ref{errorlem} gives
\begin{align}
    S(A,&P(z))>\sum_{j=0}^{\ell-1}(-1)^j|A^{(j)}|V^{(j+1)}(z)+(-1)^{\ell}|A^{(\ell)}|V^{(\ell)}(z)\notag\\
    &-8U_N^{(1)}|A|\left(1+\frac{32.02}{\log^2 N}\right)\left(\frac{(1-f(c_{\alpha_1,X_2}))+\varepsilon C_2(\varepsilon)e^2h(c_{\alpha_1,X_2})}{\log N}\right)\notag\\
    &-8\left(\sum_{j=1}^{\ell-1}|A^{(j)}|U_N^{(j+1)}+|A^{(\ell)}|U_N^{(\ell)}\right)\left(1+\frac{32.02}{\log^2N}\right)\left(\frac{\overline{m}_{\alpha_1,X_2}+\varepsilon \overline{C}(\varepsilon)e^2h(c_{\alpha_1,X_2})}{\log N}\right)\notag\\
    &-\frac{c_4^*(X_2)N}{\log^3N}\frac{1.3841\log(\log^{10}x_1(N))}{\log\log(\log^{10}x_1(N))},\label{bigeq}
\end{align}
where we recall that $\overline{m}_{\alpha_1,X_2}$ is defined in \eqref{eq: def-m-bar-alpha1-X2}. Above, we used that $c_{\alpha_1, X_2} \le s^{(1)}_j$ for all $j$ and that $h(s)$ and $F(s)$ are decreasing whereas $f(s)$ is increasing for $3 \le s \le 4$ --- this follows from the definitions of these functions \eqref{eq-def-h(s)}, \eqref{eq: approx-F34} and \eqref{eq: approx-f24}.

We note that the term corresponding to $j = 0$ is written separately in line 2 of \eqref{bigeq} since it will be estimated differently from the cases $0 < j < \ell$ from line 3. 

We now bound each line in \eqref{bigeq}.

\begin{lemma} \label{lem: new-case-b-main}
Keep the notations from the beginning of Section \ref{section:1} and Theorem \ref{theo:S>1}, and assume $k_1 \ge K_{\delta}(x_1)$.
Let $A^{(j)}$ (with $A = A^{(0)}$) and $V^{(j)}(z)$ be as in \eqref{eq: def-mj-Aj-Pj(x)} and \eqref{eq: def-V(x)-Vj(x)} respectively. Then
\begin{equation}\label{firstlineeq}
    \sum_{j=0}^{\ell-1}(-1)^j|A^{(j)}|V^{(j+1)}(z)+(-1)^{\ell}|A^{(\ell)}|V^{(\ell)}(z)=|A|V(z)\left(1+\theta a(X_2)\right)
\end{equation}
where $|\theta|\le 1$ and $a(X_2)$ is defined in \eqref{eq: def-a(X_2)}.
\end{lemma}

\begin{proof}
By the definition of $r(m_j)$, given in \eqref{eq: def-r(d)-rk(d)}, we have
\begin{align*}
    \sum_{j=0}^{\ell-1}(-1)^j|A^{(j)}|V^{(j+1)}(z)&+(-1)^{\ell}|A^{(\ell)}|V^{(\ell)}(z)=\\
    &\sum_{j=0}^{\ell-1}(-1)^j\frac{|A|}{\varphi(m_j)}V^{(j+1)}(z)+(-1)^{\ell}\frac{|A|}{\varphi(m_\ell)}V^{(\ell)}(z)\\
    &\qquad+\sum_{j=0}^{\ell-1}(-1)^jr(m_j)V^{(j+1)}(z)+(-1)^{\ell}r(m_\ell)V^{(\ell)}(z),
\end{align*}
where $V^{(j)}(z)$ is defined in \eqref{eq: def-V(x)-Vj(x)}.
Writing \label{def-varphi-*-n} $\varphi^*(n)=n\prod_{p\mid n}\frac{p-2}{p}$ (so that $\varphi^*(m_j)=\prod_{i=1}^j(q_i-2$) for $j\ge 1$), we have
\begin{equation*}
    \frac{|A|}{\varphi(m_j)}V^{(j+1)}(z)=\frac{|A|V(z)}{\varphi^*(m_j)}\left(1+\frac{1}{q_{j+1}-2}\right)
\end{equation*}
so that
\begin{equation}\label{firstlinephieq}
    \sum_{j=0}^{\ell-1}(-1)^j\frac{|A|}{\varphi(m_j)}V^{(j+1)}(z)=|A|V(z)\left(1+\frac{(-1)^{\ell-1}}{\varphi^*(m_\ell)}\right)
\end{equation}
and thus
\begin{equation*}
    \sum_{j=0}^{\ell-1}(-1)^j\frac{|A|}{\varphi(m_j)}V^{(j+1)}(z)+(-1)^{\ell}\frac{|A|}{\varphi^*(m_{\ell})}V^{(\ell)}(z)=|A|V(z).
\end{equation*}
Next, by Lemmas \ref{elllem} and \ref{lemma:r}
\begin{align}
    \left|\sum_{j=0}^{\ell-1}(-1)^jr(m_j)V^{(j+1)}(z)+(-1)^{\ell}r(m_\ell)V^{(\ell)}(z)\right|&\le a_1(X_2)\frac{NV^{(\ell)}(z)\log\log\log N}{\log^{1+\delta}N}\notag\\
    &\le a_1(X_2)\frac{|A|V^{(\ell)}(z)\log\log\log N}{\log^{\delta}N}\label{a1eq},
\end{align}
with $a_1(X_2)$ defined in \eqref{eq: def-a_1(X_2)}. We have
\begin{align}
    1 \le \frac{V^{(\ell)}(z)}{V(z)}&=\prod_{j=1}^\ell\frac{q_j-1}{q_j-2}\notag\\
    &\le\prod_{p>2}\frac{(p-1)^2}{p(p-2)}\frac{k_1}{\varphi(k_1)}\notag\\
    &\le\prod_{p>2}\frac{(p-1)^2}{p(p-2)}\left(e^{\gamma}\log\log k_1+\frac{2.5}{\log\log k_1}\right)\label{phiineq}\\
    &\le\prod_{p>2}\frac{(p-1)^2}{p(p-2)}\left(e^{\gamma}\log\log\log^{10}(x_1(N))+\frac{2.5}{\log\log\log^{10}(x_1(N))}\right),\label{vleq} 
\end{align}
where in \eqref{phiineq} we used \cite[Theorem 15]{R-S} noting that $k_1$ is odd, and in \eqref{vleq} we used that $k_1\ge K_{\delta}(x_1)\ge 3022$ so that the expression is increasing. 

Substituting \eqref{vleq} into \eqref{a1eq} and using the definition \eqref{eq: def-a(X_2)} of $a(X_2)$ then completes the proof of \eqref{firstlineeq}.
\end{proof}

We now move onto the second line of \eqref{bigeq}. By \eqref{vjzbound} in Lemma \ref{vjlem} we have
\begin{equation*}
    \frac{8 U_N}{\log N} < V(z) + \frac{8 \cdot 32.02 U_N}{\log^3 N}.
\end{equation*}
Thus, by the formula \eqref{eq: approx-f24} of $f$, for any $3 \le s' \le 4$, 
\begin{equation*}
    \frac{8U_N}{\log N}-\frac{8U_Nf(s')}{\log N}<V(z)+8\left(\frac{32.02 U_N}{\log^3 N}-\frac{2U_N e^{\gamma}\log(s'-1)}{s'\log N}\right).
\end{equation*}

Using this result, along with Lemma \ref{e0lem} and the definition \eqref{eq: def-c-alpha1-X2} of $c_{\alpha_1,X_2} \in [3,4]$, gives
\begin{align}
    &8|A|U_N^{(1)}\left(\frac{1-f(c_{\alpha_1,X_2})+\varepsilon C_2(\varepsilon)e^2h(c_{\alpha_1,X_2})}{\log N}\right) \nonumber\\
    &<|A|V(z)+8|A|U_N\Bigg\{-\frac{2e^{\gamma}\log\left(3-8\alpha_1-\frac{160\log\log X_2}{\log X_2}\right)}{(4-8\alpha_1-\frac{160\log\log X_2}{\log X_2})\log N}+\frac{32.02}{\log^3 N} \nonumber\\
    &\quad+\frac{\varepsilon_1(X_2,\delta)(1-f(c_{\alpha_1,X_2}))}{\log N}+\frac{(1+\varepsilon_1(X_2,\delta))\varepsilon C_2(\varepsilon)e^2h(c_{\alpha_1,X_2})}{\log N}\Bigg\}. \label{eq: bound-U-N-1}
\end{align}

Finally, we deal with the third line of \eqref{bigeq}.
\begin{lemma} \label{lem: new-case-b-line-3}
Keep the notations from the beginning of Section \ref{section:1} and Theorem \ref{theo:S>1}, and assume $k_1 \ge K_{\delta}(x_1)$.
Let $A^{(j)}$ (with $A = A^{(0)}$) and $V^{(j)}(z)$ be as in \eqref{eq: def-mj-Aj-Pj(x)} and \eqref{eq: def-V(x)-Vj(x)} respectively. Then
\begin{equation}
    \sum_{j=1}^{\ell-1}|A^{(j)}|U_N^{(j+1)}+|A^{(\ell)}|U_N^{(\ell)}\notag \le |A|U_N \left(3\varepsilon_1(X_2,\delta)+a(X_2) \right),
\end{equation}
where $a(X_2)$ is defined in \eqref{eq: def-a(X_2)}.
\end{lemma}
\begin{proof}
The argument is analogous to the proof of Lemma \ref{lem: new-case-b-main}. Namely, by the definition of $r(m_j)$, given in \eqref{eq: def-r(d)-rk(d)}, we have
\begin{align}
    \sum_{j=1}^{\ell-1}|A^{(j)}|U_N^{(j+1)}&+|A^{(\ell)}|U_N^{(\ell)}\notag\\
    =\sum_{j=1}^{\ell-1}\frac{|A|}{\varphi(m_j)}U_N^{(j+1)}+\frac{|A|}{\varphi(m_\ell)}U_N^{(\ell)}&+\sum_{j=1}^{\ell-1}r(m_j)U_N^{(j+1)}+r(m_\ell)U_N^{(\ell)}\label{thirdlineeq}.
\end{align}
Similarly to \eqref{firstlinephieq},
\begin{equation*}
    \frac{|A|}{\varphi(m_j)}U_N^{(j+1)}=\frac{|A|U_N}{\varphi^*(m_j)}\left(1+\frac{1}{q_{j+1}-2}\right)
\end{equation*}
and 
\begin{equation*}
    \frac{|A|}{\varphi(m_\ell)}U_N^{(\ell)}=\frac{|A|U_N}{\varphi^*(m_\ell)}
\end{equation*}
so that
\begin{align*}
    &\sum_{j=1}^{\ell-1}\frac{|A|}{\varphi(m_j)}U_N^{(j+1)}+\frac{|A|}{\varphi(m_\ell)}U_N^{(\ell)}
    =\\
    &=\frac{|A|U_N}{q_1 - 2}\left(1+\frac{2}{q_2-2}+\frac{2}{(q_2-2)(q_3-2)}+\cdots+\frac{2}{(q_2-2)\cdots(q_\ell-2)}\right)\\
    &\le 3|A|U_N\varepsilon_1(X_2,\delta).
\end{align*}
In the last line, we used Lemma \ref{leq1lem} and the inequality $\frac{1}{q_1 - 2} \le \varepsilon_1(X_2, \delta)$, which follows from definition \eqref{eq: def-epsilon0-X2-delta} of $\varepsilon_1(X_2, \delta)$.

Therefore, we have bounded the first two terms of \eqref{thirdlineeq}. We use Lemma \ref{lemma:r} and that $U_N^{(j)} \le U_N^{(\ell)}$ for $0 \le j \le \ell$ following from the definition \eqref{eq: def-UNj}, to bound the last two terms:
\begin{align*}
    \left| \sum_{j=1}^{\ell-1}r(m_j)U_N^{(j+1)}+r(m_\ell)U_N^{(\ell)}\right| &< (\ell - 1) c_2(X_2) \frac{N}{\log^3 N} U_N^{(\ell)} + \frac{c_3(X_2) N \log \log N}{\log^{1 + \delta} N} U_N^{(\ell)} \\
    &\le \left(\frac{a_1(X_2)\log\log\log N}{\log^{1+\delta}N}\right)U_N^{(\ell)},
\end{align*}
by the definition \eqref{eq: def-a_1(X_2)} of $a_1(X_2)$ and the bound for $\ell$ from Lemma \ref{elllem}.

Thus we get an upper bound for \eqref{thirdlineeq}:
\begin{equation*}
    \sum_{j=1}^{\ell-1}|A^{(j)}|U_N^{(j+1)}+|A^{(\ell)}|U_N^{(\ell)}\le 3|A|U_N\varepsilon_1(X_2,\delta)+\left(\frac{a_1(X_2)|A|\log\log\log N}{\log^{\delta}N}\right)U_N^{(\ell)}.
\end{equation*}
Similarly to \eqref{vleq}, we derive
\begin{align*}
    \frac{U_N^{(\ell)}}{U_N}\le\prod_{p>2}\frac{(p-1)^2}{p(p-2)}\left(e^{\gamma}\log\log\log^{10}(x_1(N))+\frac{2.5}{\log\log\log^{10}(x_1(N))}\right),
\end{align*}
whence
\begin{align*}
    &\sum_{j=1}^{\ell-1}|A^{(j)}|U_N^{(j+1)}+|A^{(\ell)}|U_N^{(\ell)}\le 3|A|U_N\varepsilon_1(X_2,\delta)+a(X_2)U_N \\
    &\le |A|U_N \left(3\varepsilon_1(X_2,\delta) + a(X_2)\right) \le |A|U_N \left(3\varepsilon_1(X_2,\delta) + a(X_2) \right),
\end{align*}
where in the last inequality we used that $|A| > \frac{N}{\log N} > \frac{X_2}{\log X_2}$ by \eqref{eq:A>}.
\end{proof}
We combine \eqref{bigeq} with Lemmas \ref{lem: new-case-b-main} and \ref{lem: new-case-b-line-3} and the bound \eqref{eq: bound-U-N-1} to obtain:
\begin{align}\label{eq: second-last-bound-TheoA}
    S(A,&P(z))> |A|V(z)-a(X_2)|A|V(z) \\
    &-\left(1 + \frac{32.02}{\log^2 N} \right)\Bigg\{|A|V(z)+8|A|U_N\Bigg[-\frac{2e^{\gamma}\log\left(3-8\alpha_1-\frac{160\log\log X_2}{\log X_2}\right)}{(4-8\alpha_1-\frac{160\log\log X_2}{\log X_2})\log N}+\frac{32.02}{\log^3 N} \nonumber\\
    &+\frac{\varepsilon_1(X_2,\delta)(1-f(c_{\alpha_1,X_2}))}{\log N}+\frac{(1+\varepsilon_1(X_2,\delta))\varepsilon C_2(\varepsilon)e^2h(c_{\alpha_1,X_2})}{\log N}\Bigg]\Bigg\}\notag\\
    &-8|A|U_N \left(3\varepsilon_1(X_2,\delta)+a(X_2)\right)\left(1+\frac{32.02}{\log^2N}\right)\left(\frac{\overline{m}_{\alpha_1,X_2}+\varepsilon \overline{C}(\varepsilon)e^2h(c_{\alpha_1,X_2})}{\log N}\right)\notag\\
    &-\frac{c_4^*(X_2)N}{\log^3N}\frac{1.3841\log(\log^{10}x_1(N))}{\log\log(\log^{10}x_1(N))}.\notag 
\end{align}
The terms involving $V(z)$ from the first and the second lines cancel out as follows:
\begin{align*}
    |A|V(z)-a(X_2)|A|V(z) - \left(1 + \frac{32.02}{\log^2 N} \right)|A|V(z) &= -|A|V(z) \left( a(X_2) + \frac{32.02}{\log^2 N}\right) \\
    &> -\frac{8|A|U_N}{\log N} \left( a(X_2) + \frac{32.02}{\log^2 N}\right),
\end{align*}
where we used the bound \eqref{vjzbound} in the last inequality. The lower bound \eqref{eq: second-last-bound-TheoA} therefore simplifies to
\begin{align*}
    &\frac{8|A|U_N}{\log N}\left(1 + \frac{32.02}{\log^2 N} \right)\Bigg\{\frac{2e^{\gamma}\log\left(3-8\alpha_1-\frac{160\log\log X_2}{\log X_2}\right)}{(4-8\alpha_1-\frac{160\log\log X_2}{\log X_2})}-\frac{32.02}{\log^2 N}-a(X_2) - \frac{32.02}{\log^2 N} \nonumber\\
    &-\varepsilon_1(X_2,\delta)(1-f(c_{\alpha_1,X_2}))-(1+\varepsilon_1(X_2,\delta))\varepsilon C_2(\varepsilon)e^2h(c_{\alpha_1,X_2})\nonumber\\
    &-\left(3\varepsilon_1(X_2,\delta)+a(X_2) \right)\left(\overline{m}_{\alpha_1,X_2}+\varepsilon \overline{C}(\varepsilon)e^2h(c_{\alpha_1,X_2})\right)\Bigg\},\notag\\
    &-\frac{c_4^*(X_2)N}{\log^3N}\frac{1.3841\log(\log^{10}x_1(N))}{\log\log(\log^{10}x_1(N))}.
\end{align*}
which completes the proof of case (b) in Theorem \ref{theo:S>1} upon noting that $|A| > \frac{N}{\log N}$ by \eqref{eq:A>} and $U_N^{-1} \le \left(2e^{\gamma}\prod_{p>2}\left(1-\frac{1}{(p-1)^2}\right)\right)^{-1}$ which follows from the definition of $U_N = U^{(0)}_N$ in \eqref{eq: def-UNj}.
\end{proof}

\section{An upper bound for $\sum_{z \le q<y} S(A_q,P(z))$}
\label{section:2}
In this section, we shall obtain an upper bound for the sum over primes $\sum_{z \le q < y} S(A_q,P(z))$ with $z=N^{1/8}$, $y=N^{1/3}$ and each $q$ not dividing $N$. This is the second term appearing in the bound for $\pi_2(N)$ in Lemma \ref{lem:>}. Compared to the lower bound for $S(A,P(z))$ in \S \ref{section:1}, this is obtained in a quite straightforward way, using Theorem~\ref{theo:JR} and Lemmas~\ref{small533} and \ref{large533}.
We start by defining 
\begin{equation} \label{eq: def-k-x}
    k_{x}:=8 \left(\frac{1}{2}-\frac{1}{3}-x \right)=8\left(\frac{1}{6}-x\right).
\end{equation}
\begin{theorem}
\label{theo:2.}
Let $u_0=10^9$ and $\varepsilon=1.452\cdot 10^{-7}$ be the corresponding values in Lemma~\ref{lem: prod-u>u_0}. Let $x_1$ be as in \eqref{notcon1} and $X_2$ be such that $\log\log x_1(X_2)\ge 10.4$. Assume that $0<\alpha_2 <1/24$ and $N\ge X_2$ is an even integer such that
\begin{equation*}
\frac{N^{\alpha_2}}{\log^{10} x_1(N) \log^{2.5} N}\ge \exp \left(u_0\left(1+\frac{9\cdot 10^{-7}}{\log u_0} \right)\right).
\end{equation*}
Let $U_N = U_N^{(0)}$ and $K_{\delta}(x_1)$ be as in \eqref{eq: def-UNj} and \eqref{notcon2} respectively. Consider the cases (a) $k_1<K_{\delta}(x_1)$ and (b) $k_1\ge K_{\delta}(x_1)$. In case (a) we have
\begin{align*}
\sum_{\substack{z \le q < y\\ q\nmid N}} S(A_q,P(z)) <& \frac{U_N N}{\log^2N}\Bigg(8.0004\left( 1+\frac{32.02}{\log^2 N}\right)(l_1(X_2)+l_2(X_2))\\&\qquad\qquad+\frac{l(X_2)}{2 e^{\gamma}\prod_{p>2}\left( 1-\frac{1}{(p-1)^2}\right) \log N} \Bigg).
\end{align*}
On the other hand, in case (b)
\begin{align*}
\sum_{\substack{z \le q < y\\q\nmid N}} S(A_q,P(z)) <& \frac{U_N N}{\log^2N}\Bigg(8.0004\left( 1+\frac{32.02}{\log^2 N}\right)\left(1+\varepsilon_1(X_2,\delta)\right)(l_1^*(X_2)+l_2(X_2))\\&\qquad\qquad+\frac{l^*(X_2)}{2 e^{\gamma}\prod_{p>2}\left( 1-\frac{1}{(p-1)^2}\right) \log N} \Bigg),
\end{align*}
where
\begin{align}
l(X_2):=&p(X_2)\left(\frac{1}{\log X_2}+0.55 \right)\notag\\&+ 1.3841\left(\log \frac{8}{3}+\frac{64}{\log^2 X_2} \right)\frac{ \sqrt{x_1(X_2)}\log^3 X_2}{X_2 \left(\log \log X_2\right)\left( \log^{10} x_1(X_2)\right)},\label{elleq}\\
l^*(X_2):=&p^*(X_2)\left(\frac{1}{\log X_2}+0.55 \right)\notag\\&+ 1.3841\left(\log \frac{8}{3}+\frac{64}{\log^2 X_2} \right)\frac{ \sqrt{x_1(X_2)}\log^3 X_2}{X_2 \left(\log \log X_2\right)\left( \log^{10} x_1(X_2)\right)},\label{ellstareq}\\
l_1(X_2):=&\frac{\log^2 X_2}{X_2}  \left(\frac{ 2e^{\gamma}}{k_{\alpha_2}}
+\varepsilon C_1(\varepsilon) e^2 h(k_{\alpha_2})\right) \notag\\ & \cdot \left(\frac{X_2^{1/8}}{X_2^{1/8}-1}\left(\log \frac{8}{3}+\frac{64}{\log^2X_2}\right)\frac{1.3841}{\log \log X_2} +\frac{p(X_2)X_2}{\log^4X_2}\right),\label{ell1eq}\\
l_1^*(X_2):=&\frac{\log^2 X_2}{X_2}  \left(\frac{ 2e^{\gamma}}{k_{\alpha_2}}
+\varepsilon C_1(\varepsilon) e^2 h(k_{\alpha_2})\right) \notag\\ & \cdot \left(\frac{X_2^{1/8}}{X_2^{1/8}-1}\left(\log \frac{8}{3}+\frac{64}{\log^2X_2}\right)\frac{1.3841}{\log \log X_2} +\frac{p^*(X_2)X_2}{\log^4X_2}\right),\label{ell1stareq}\\
l_2(X_2)&:=\frac{X_2^{1/8}}{X_2^{1/8}-1} 
\Bigg[\frac{e^{\gamma}}{4}\left(\frac{\log(6)+\log\left(\frac{3-8\alpha_2}{3-18\alpha_2}\right)}{\left(\frac{1}{2}-\alpha_2\right) }+\frac{512}{k_{\alpha_2}\log^2X_2}\right)\notag\\& +\left(\log \frac{8}{3}+\frac{64}{\log^2X_2}\right)\varepsilon C_1(\varepsilon) e^2 h(k_{\alpha_2}) \Bigg],\label{ell2eq}
\end{align}
with $p(X_2)$ and $p^*(X_2)$ defined in Lemmas~\ref{small533} and \ref{large533}, $\varepsilon_1(X_2,\delta)$ in \eqref{eq: def-epsilon0-X2-delta}, and $C_1(\varepsilon)=106$ is as in Table \ref{tab:fF}.
\end{theorem}
\begin{proof}
We again begin with case (a) $k_1<K_{\delta}(x_1)$. Let $N^{\frac{1}{8}}=z\le q<y=N^{\frac{1}{3}}$. Similar to the proof of Theorem \ref{theo:S>1} we set
\begin{equation} \label{eq: def-Q-D2-D2q-sq}
    Q(u):=\prod_{\substack{p< u\\p\nmid N}} p,\quad D^{(2)}:=N^{\frac{1}{2}-\alpha_2},\quad D^{(2)}_q:=\frac{D^{(2)}}{q}, \quad \text{and} \quad s^{(2)}_q:=\frac{\log D^{(2)}_q}{\log z}.
\end{equation}
The condition
\begin{equation*}
    \frac{N^{\alpha_2}}{\log^{10} x_1(N) \log^{2.5} N}\ge \exp\left(u_0\left(1+\frac{9\cdot 10^{-7}}{\log u_0} \right)\right)
\end{equation*}
then guarantees that
\begin{equation}
\label{cond:N*}
\frac{N^{\alpha_2}}{\log^{10} x_1(N) \log^{2.5} N}\ge Q(u_0)
\end{equation} 
by Lemma \ref{thetauplem} and we also have $qD^{(2)}_qQ(u_0)\le H=\frac{\sqrt{x_1}}{\log^{10}x_1}$. Moreover, the condition $\alpha_2<1/24$ gives
\begin{equation}
\label{cond:N.1*}
\frac{N^{\frac{1}{2}-\alpha_2}}{q}\ge z,
\end{equation} 
for all $z\le q<y$. Namely, this means that we can apply the upper bound $\eqref{suppereq}$ in Theorem \ref{theo:JR}. In particular, setting $g_n(p)=1/(p-1)$, $A=A_q$, $Q=Q(u_0)$ and $D=D^{(2)}_q$ in \eqref{suppereq} gives
\begin{equation}\label{eq:up1}
    \sum_{\substack{z \le q<y\\ q\nmid N}} S(A_q,P(z))\le 8U_N\left( 1+\frac{32.02}{\log^2 N}\right)\sum_{\substack{z \le q < y\\q\nmid N}}\left|A_q\right|\left(\frac{F(s^{(2)}_q)+\varepsilon C_{1}(\varepsilon) e^2 h(s^{(2)}_q)}{\log N}\right)
\end{equation}
\begin{equation*}
    \qquad\qquad\qquad\qquad\qquad\qquad\qquad+\sum_{\substack{z \le q < y\\ q\nmid N}}\sum_{\substack{d|P(z)\\d<D^{(2)}_qQ(u_0)}}\left|r_{q} (d)\right|.
\end{equation*}
where we have used \eqref{eq: def-X_A}, namely,
\begin{equation*}
    X_A:=|A|\prod_{p\mid P(z)}(1-g_n(p))\le \frac{8U_N}{\log N}\left(1+\frac{32.02}{\log^2 N}\right).
\end{equation*}
by \eqref{vjzbound} in Lemma \ref{vjlem}.

We start by bounding the sum over $|r_q(d)|$ in \eqref{eq:up1}. By the definition of $r_q(d)$ given in \eqref{rkddef},
\begin{align}\label{sumrq}
\sum_{\substack{z \le q < y\\q\nmid N}}\sum_{\substack{d|P(z)\\d<D^{(2)}_qQ(u_0)}}\left|r_{q} (d)\right|
\le \sum_{\substack{z \le q < y\\q\nmid N}}\sum_{\substack{d|P(z)\\d<D^{(2)}_qQ(u_0)}}\left(\left|E_{\pi}(N;qd,N)\right|+\frac{\left|E_{\pi}(N;q,N)\right|}{\varphi (d)}+\omega(N)\right),
\end{align}
noting that $(q,d)=1$ since $q$ is a prime greater than or equal to $z=N^{\frac{1}{8}}$. We thus have to bound three sums. Since $qD^{(2)}_qQ(u_0)\le H$, we have by Lemma \ref{small533} that the first sum in \eqref{sumrq} can be bounded as
\begin{equation}\label{ellfirstbound}
    \sum_{\substack{z \le q < y\\ q\nmid N}}\sum_{\substack{\substack{d|P(z)\\d<D^{(2)}_qQ(u_0)}}}\left|E_{\pi}(N;qd,N)\right| < \frac{p(X_2)N}{\log^3 N}.
\end{equation}
We can now bound the second sum in \eqref{sumrq} using Lemmas \ref{muphilem} and \ref{small533}. That is,
\begin{align}
\label{eq:reasonx2}
 \sum_{\substack{z \le q < y\\q\nmid N}}\sum_{\substack{d|P(z)\\d<D^{(2)}_qQ(u_0)}}\frac{\left|E_{\pi}(N;q,N)\right|}{\varphi (d)} \le \left(\sum_{z \le q < y} \left|E_{\pi}(N;q,N)\right| \right)\left( \sum_{\substack{d\le H\\ d\mid P(z)}}\frac{1}{\varphi (d)}\right)\le
0.55\frac{p(X_2)N}{\log^2N}.
\end{align}
Finally, for the third sum in \eqref{sumrq}, Lemmas \ref{1plem} and \ref{omegalem} give
\begin{equation}\label{ellfinalbound}
\sum_{\substack{z \le q < y\\q\nmid N}}\sum_{\substack{d|P(z)\\d<D^{(2)}_qQ(u_0)}}\omega(N)\le  1.3841\frac{\log N}{\log \log N}\left(\log \frac{8}{3}+\frac{64}{\log^2 N} \right)\frac{\sqrt{x_1(N)}}{\log^{10} x_1(N)},
\end{equation}
noting that $\log\log(y/z)=\log\log(N^{1/3}/N^{1/8})=\log\frac{8}{3}$. Hence, with $l(X_2)$ defined as in \eqref{elleq}, the three bounds \eqref{ellfirstbound}, \eqref{eq:reasonx2} and \eqref{ellfinalbound} give
\begin{equation}
\label{eq:rsup}
\sum_{\substack{z \le q < y\\q\nmid N}}\sum_{\substack{d|P(z)\\d<D^{(2)}_qQ(u_0)}}\left|r_{q} (d)\right|\le \frac{l(X_2)N}{\log^2 N}.
\end{equation}
We now bound the sum on the right-hand side of \eqref{eq:up1}. By \eqref{cond:N.1*} and the definition of $s^{(2)}_q$, we have $1< s^{(2)}_q<3$ and thus $F(s^{(2)}_q)=\frac{2e^{\gamma}}{s^{(2)}_q}$ by \eqref{eq: approx-F03}. Moreover, we note that $s^{(2)}_q\ge k_{\alpha_2}$ and by \eqref{AAqeq} 
\begin{equation*}
    \left|A_q\right|\le \frac{\left|A\right|+\omega(N)}{q-1}+E_{\pi}(N;q,N).
\end{equation*}
Therefore, 
\begin{align}
\nonumber
\sum_{\substack{z \le q < y\\q\nmid N}}\left|A_q\right|&\left(\frac{F(s^{(2)}_q)+\varepsilon C_{1}(\varepsilon) e^2 h(s^{(2)}_q)}{\log N}\right)
\\\le &
\label{eq:sf1}
 \sum_{\substack{z \le q < y\\q\nmid N}}\frac{\left|A\right|}{q-1} \left(\frac{ e^{\gamma}}{4\log D^{(2)}/q}+\frac{\varepsilon C_{1}(\varepsilon) e^2 h(k_{\alpha_2})}{\log N}\right)
 \\&
\label{eq:sf2}
+ \left(\frac{ e^{\gamma}}{4\log D^{(2)}/y}
+\frac{\varepsilon C_{1}(\varepsilon) e^2 h(k_{\alpha_2})}{\log N}\right)\sum_{\substack{z \le q < y\\q\nmid N}} \left(\frac{\omega(N)}{q-1} +E_{\pi}(N;q,N)\right).
\end{align}
where we have used that $h(s)$ is decreasing. We start bounding \eqref{eq:sf2}. Using Lemma~\ref{small533} we obtain $\sum_{z \le q < y} E_{\pi}(N;q,N)<\frac{p(X_2)N}{\log^3N}$. Then, by Lemmas \ref{1plem} and \ref{omegalem},
\begin{equation*}
\sum_{\substack{z \le q < y\\q\nmid N}}\frac{\omega(N)}{q-1}\le \frac{N^{1/8}}{N^{1/8}-1} \left(\log \frac{8}{3}+\frac{64}{\log^2N}\right)\frac{1.3841\log N}{\log \log N}.
\end{equation*}
This allows us to bound \eqref{eq:sf2} with
\begin{equation}
\label{eq:sf22}
    \left(\frac{ 2e^{\gamma}}{k_{\alpha_2}}+\varepsilon C_{1}(\varepsilon) e^2 h(k_{\alpha_2})\right)\left(\frac{N^{1/8}}{N^{1/8}-1}\left(\log \frac{8}{3}+\frac{64}{\log^2N}\right)\frac{1.3841}{\log \log N} +\frac{p(X_2)N}{\log^4N}\right).
\end{equation}
Dividing \eqref{eq:sf22} by $N/\log^2N$ gives rise to a monotonically decreasing function for $N\ge X_2$ which is thus bounded by $l_1(X_2)$ (Equation \eqref{ell1eq}). We now bound \eqref{eq:sf1}. Applying Lemma \ref{lemma:sump} with $f(t)=1/\log(D/t)$, $g(t)=\log\log t$, 
\begin{equation*}
    c(n)=
    \begin{cases}
        1/n,&\text{if $n$ is prime},\\
        0,&\text{otherwise}
    \end{cases}
\end{equation*}
and $E=64/\log^2 N$ (as a consequence of Lemma \ref{1plem}), we have
\begin{align}
    \sum_{\substack{z \le q < y\\q\nmid N}}\frac{1}{q\log D^{(2)}/q}&\le \sum_{z \le q < y}\frac{1}{q\log D^{(2)}/q}\le\int_z^y\frac{1}{t\log t \log D^{(2)}/t}\mathrm{d}t+\frac{64}{\log^2N}\frac{1}{\log D^{(2)}/y}\notag\\
    & = \frac{\log(6)+\log\left(\frac{3-8\alpha_2}{3-18\alpha_2}\right)}{\left(\frac{1}{2}-\alpha_2\right) \log N}+\frac{512}{k_{\alpha_2}\log^3N},\label{logdtint}
\end{align}
where we substituted $z=N^{1/8}$, $y=N^{1/3}$ and $D^{(2)}=N^{1/2-\alpha_2}$ to obtain the final equality. Using \eqref{logdtint} and Lemma \ref{1plem}, we have that \eqref{eq:sf1} is at most
\begin{align}
\label{eq:sf11}
&\frac{\left|A\right|}{\log N}\frac{N^{1/8}}{N^{1/8}-1}\Bigg[\frac{e^{\gamma}}{4} \left( \frac{\log(6)+\log\left(\frac{3-8\alpha_2}{3-18\alpha_2}\right)}{\frac{1}{2}-\alpha_2}+\frac{512}{k_{\alpha_2}\log^2N}\right) \notag\\
&\qquad\qquad+\left(\log \frac{8}{3}+\frac{64}{\log^2N}\right)\varepsilon C_{1}(\varepsilon) e^2 h(k_{\alpha_2}) \Bigg]\\
&\le \frac{|A|}{\log N}l_2(X_2)\label{lx2argument},
\end{align}
noting that each term in \eqref{eq:sf11} (upon taking out the factor of $|A|/\log N$) is either constant or decreasing in $N$. Combining \eqref{eq:A>}, \eqref{eq:rsup}, \eqref{eq:sf22} and \eqref{eq:sf11} we obtain the desired result and thereby finish the proof of case (a).

Now we consider the case (b) $k_1\ge K_{\delta}(x_1)$. This case requires a slightly different argument to case (a). In particular, we can no longer apply Lemma \ref{small533} for such large values of $k_1$. To circumvent this, we note that $\sum_{z \le q < y} S(A_q,P(z))\le \sum_{z \le q < y} S(A_q,P^{(1)}(z))$ so that it suffices to bound the latter. Working with $P^{(1)}(z)$ as opposed to just $P(z)$ then allows us to guarantee $d\neq k_1$ in the sieve remainder term. In particular, defining
\begin{equation}\label{Q1eq}
    Q^{(1)}(u):=\prod_{\substack{p< u,\ p\nmid N \\ p\neq q_1}}p
\end{equation}
we then, similar to case (a), use Lemma \ref{vjlem} and Theorem \ref{theo:JR} with $g_n(p)=1/(p-1)$, $A=A_q$, $Q=Q^{(1)}(u_0)$ and $D=D^{(2)}_q$ in \eqref{suppereq} to obtain
\begin{align*}
    S(A_q,P^{(1)}(z))<&\left|A_q\right|\left(8U_N^{(1)}\left( 1+\frac{32.02}{\log^2 N}\right)\frac{ F(s^{(2)}_q)+\varepsilon C_{1}(\varepsilon) e^2 h(s^{(2)}_q)}{\log N}\right) \\ &+\sum_{\substack{d|P^{(1)}(z)\\d<D^{(2)}_qQ^{(1)}(u_0)}}\left|r_{q} (d)\right|.
\end{align*}
Here,
\begin{equation*}
    U_N^{(1)}\le U_N(1+\varepsilon_1(X_2,\delta))
\end{equation*}
by Lemma \ref{e0lem}. The proof then follows as in case (a). The only difference is that we have
\begin{align}
    \sum_{\substack{z \le q < y\\q\nmid N}}\sum_{\substack{\substack{d|P^{(1)}(z)\\d<D^{(2)}_qQ^{(1)}(u_0)}}}\left|E_{\pi}(N;qd,N)\right| &< \frac{p^*(X_2)N}{\log^3 N},\ \text{and}\label{pstar1}\\
    \sum_{\substack{z \le q < y\\q\nmid N}} |E_{\pi}(N;q,N)|&<\frac{p^*(X_2)N}{\log^3N}\label{pstar2}
\end{align}
by Lemma \ref{large533}. In particular, to apply Lemma \ref{large533} we require in \eqref{pstar1} that $k_1\nmid qd$ and in \eqref{pstar2} that $k_1\nmid q$. However, this is true since $q\ge z>k_1$ is prime and $d\neq k_1$ as $d\mid P^{(1)}(z)$.
\end{proof}

\section{An upper bound for $S(B,P(y))$}
\label{section:3}
We will now prove an upper bound for  $S(B,P(y))$. This is the third term appearing in the bound for $\pi_2(N)$ in Lemma \ref{lem:>}. The bound will be obtained using Theorem~\ref{theo:JR} together with Lemmas \ref{lemma:bf} and \ref{bflem2}. Unlike the proofs of Theorems \ref{theo:S>1} and \ref{theo:2.} we will only provide a single bound for $S(B,P(y))$ rather than giving two bounds depending on the value of $k_i$. This is in part because we define a sequence of different values for $Y=Y_j$ appearing in Lemma \ref{lemma:bf} and then take a overall bound which is independent of the value of the exceptional modulus.

\begin{theorem}
\label{theo:B}
Let $u_0=10^9$ and $\varepsilon=1.452\cdot 10^{-7}$ be the corresponding values in Lemma~\ref{lem: prod-u>u_0}. Let $0<\delta<2$ and $0<\varepsilon_0<1$. Set $X_3$ to be such that $\log\log x_2(X_3)\ge 10.4$ with $x_2$ defined in \eqref{notcon1}. Also let $N> (X_3)^8$ be an even integer and $0 < \alpha_3 < 1/6$ satisfying
\begin{equation*}
\frac{3^{10}N^{\alpha_3}}{ \log^{10} N }\ge \exp{\left(u_0\left(1+\frac{9\cdot 10^{-7}}{\log u_0} \right)\right)}.
\end{equation*}
We have
\begin{align*}
S(B,&P(y))< \frac{U_N N}{\log^2N}
 \Bigg\{1.00005(1+\varepsilon_2((X_3)^8,\delta))\left( 1+\frac{4.51}{\log^2N}\right)\\ & \cdot\left[\frac{2}{\frac{1}{2}-\alpha_3} e^{\gamma}+3\varepsilon C_{1}(\varepsilon)e^2h\left(\frac{3}{2}-3\alpha_3\right)\right] 
\left(1+\varepsilon_0+\frac{9}{\log N}\right)\\ & \cdot\left[\overline{c} +\frac{36}{\log^2 N} +\left(\log \frac{8}{3}+\frac{64}{\log^2 N} \right)\left( \frac{10\log (1+\varepsilon_0)}{\log N}+ \frac{27}{\log^2 N}\right)\right] 
 \\ &+\left(2 e^{\gamma}\prod_{p>2}\left( 1-\frac{1}{(p-1)^2}\right) \right)^{-1}\left(\frac{320\cdot m(X_3)(1+\varepsilon_0)}{3\log(1+\varepsilon_0)}+\frac{0.13(1+\varepsilon_0)\log^5 N}{N^{\frac{1}{8}}\log(1+\varepsilon_0)} \right)\Bigg\},
\end{align*}
with $U_N$ defined in \eqref{UNeq}, $\overline{c}$ defined in Lemma~\ref{lemma:Bover} below, $m(X_3)$ in Lemma~\ref{lemma:bf}, $\varepsilon_2$ in Lemma~\ref{e0lem}, and $C_1(\varepsilon)=106$ from Table \ref{tab:fF}.
\end{theorem}
Before we prove Theorem \ref{theo:B} we start by recalling that
\begin{equation*}
B=\left\{ N-p_1p_2p_3:z \le p_1<y \le p_2 \le p_3,~p_1p_2p_3 <N, (p_1p_2p_3,N)=1   \right\}
\end{equation*}
where $z=N^{1/8}$ and $y=N^{1/3}$ (Equations \eqref{Nyzeq} and \eqref{Bset}). With a view to apply Lemma~\ref{lemma:bf}, we now drop the restriction $(p_1,N)=1$ and relax the condition $p_1p_2p_3<N$, so that $p_1$ and $p_2p_3$ will range over independent intervals giving a bilinear form. In doing this we define
\begin{align}
B^{(j)}:=\{& N-p_1p_2p_3: z\le p_1<y\le p_2\le p_3,
 \notag\\ & \omega_jp_2p_3<N,(p_2p_3,N)=1,~\omega_j\le p_1<\omega_j(1+\varepsilon_0)\}, \label{bjdef}
\end{align}
where
\begin{equation}\label{j0eq}
\omega_j:=z(1+\varepsilon_0)^j~\text{for}~0 \le j \le j_0:=\frac{\log y/z}{\log(1+\varepsilon_0)},
\end{equation}
with $0<\varepsilon_0<1$. We see that 
\begin{equation*}
\left|B^{(j)}\right|=(\pi(Y_j)-\pi(Z_j))\sharp \{(p_2,p_3):y\le p_2\le p_3, ~\omega_jp_2p_3<N,(p_2p_3,N)=1\},
\end{equation*}
where 
\begin{equation}\label{zjyj}
    Z_j:=\omega_j\quad\text{and}\quad Y_j:=\min\left(\omega_j(1+\varepsilon_0),y\right).
\end{equation}
Defining  $\overline{B}:=\cup_jB^{(j)}$, we have
\begin{equation}
\label{eq:Bin}
B \subseteq \overline{B}\subseteq \left\{ N-p_1p_2p_3:z \le p_1<y \le p_2 \le p_3,~p_1p_2p_3 <(1+\varepsilon_0)N  \right\}
\end{equation}
and
\begin{equation}
\label{eq:numb}
S(B,P(y))\le S(\overline{B},P(y))= \sum_{j\le j_0}S(B^{(j)},P(y)).
\end{equation}
We now prove an explicit upper bound for the cardinality of $\overline{B}$ in a similar way as done by Nathanson in \cite[pp. 289--291]{Nathanson}.
\begin{lemma}
\label{lemma:Bover}
Keeping the notation and conditions of Theorem \ref{theo:B}, we have
\begin{align*} 
\left| \overline{B}\right| \le & \left(1+\varepsilon_0+\frac{9}{\log N}\right)\frac{N}{\log N}
\\ &\cdot \left[\overline{c} +\frac{36}{\log^2 N} +\left(\log \frac{8}{3} +\frac{64}{\log^2 N} \right)\left( \frac{10\log (1+\varepsilon_0)}{\log N}+ \frac{27}{\log^2 N}\right)\right],
\end{align*}
with $\overline{c}=\int_{1/8}^{1/3} \frac{\log(2-3\beta)}{\beta (1-\beta)}d \beta <0.363084$.
\end{lemma}
\begin{proof}
First note that since $p_1 < p_2 \le p_3$ and $p_1p_2p_3<(1+\varepsilon_0)N$, we have 
\begin{align}\label{p123ineq1}
\quad p_3 &< \frac{(1+\varepsilon_0)N}{p_1p_2} \quad\text{and}\\
p_1p_2^2&<(1+\varepsilon_0)N\label{p123ineq2}
\end{align}
Using \eqref{p123ineq1} and $0<\varepsilon_0<1$, we then obtain via \cite[Theorem 1]{R-S}
\begin{equation*}
\pi \left (\frac{(1+\varepsilon_0)N}{p_1 p_2} \right)\le \left(1+\varepsilon_0+\frac{9}{\log N}\right)\frac{N}{p_1 p_2 \log (N/p_1p_2)}.
\end{equation*}
Thus, from the definition \eqref{eq:Bin} of $\overline{B}$ and \eqref{p123ineq1},
\begin{align}
\left| \overline{B}\right| &\le\sum_{\substack{z \le p_1<y \le p_2 \le p_3 \\ p_1p_2p_3 <(1+\varepsilon_0)N}} 1 \le \sum_{\substack{z \le p_1<y \le p_2  \\ p_1p_2^2 <(1+\varepsilon_0)N}} \pi \left (\frac{(1+\varepsilon_0)N}{p_1 p_2} \right) \notag \\ & \le \left(1+\varepsilon_0+\frac{9}{\log N}\right)N\sum_{z \le p_1 <y}\frac{1}{p_1}\sum_{y\le p_2<w}\frac{1}{ p_2 \log (N/p_1p_2)},\label{overBfirstbound}
\end{align}
with $w=\sqrt{\frac{(1+\varepsilon_0)N}{p_1}}$. \par
We now introduce the functions $h_p(t)=\left(\log N/pt\right)^{-1}$ and
\begin{equation*}
I(u)=\int_y^{\sqrt{N/u}} h_u(t)\mathrm{d}\log\log t.
\end{equation*}
Noting that
\begin{equation*}
    y=N^{1/3}\ge X_2^{8/3}> 286,
\end{equation*}
we apply Lemma \ref{lemma:sump} with $f(t)=h_{p_1}(t)$, $g(t)=\log\log t$, 
\begin{equation*}
    c(n)=
    \begin{cases}
        1/n,&\text{if $n$ is prime},\\
        0,&\text{otherwise}
    \end{cases}
\end{equation*}
and $E=1/\log^2 y$ (as a consequence of Lemma \ref{1plem}) to obtain
\begin{align*}
\sum_{y\le p_2<w}\frac{1}{ p_2 \log (N/p_1p_2)}&\le \int_y^wh_{p_1}(t) d \log \log t+\frac{h_{p_1}(w)}{\log^2 y}\\ &= I(p_1)+\int_{\sqrt{\frac{N}{p_1}}}^wh_{p_1}(t) d \log \log t+ \frac{h_{p_1}(w)}{\log^2 y}\\ &
\le I(p_1)+\frac{10\log (1+\varepsilon_0)}{\log^2 N}+ \frac{27}{\log^3 N}.
\end{align*}
Where, in the last step, we substituted $y=N^{1/3}$ and applied the bounds
\begin{align}
\int_{\sqrt{\frac{N}{p_1}}}^wh_{p_1}(t) d \log \log t&\le \frac{10\log (1+\varepsilon_0)}{\log^2 N}\label{10logbound}\\
h_{p_1}(w)=\frac{2}{\log \left(\frac{N}{(1+\varepsilon_0)p_1}\right)}&\le \frac{3}{\log N}.\notag
\end{align} 
Here, \eqref{10logbound} is obtained by the change of variables $t=\sqrt{N/p_1}s$ as in \cite[p. 290]{Nathanson}. Therefore, also using Lemma \ref{1plem},
\begin{align}
\sum_{z \le p_1 <y}\frac{1}{p_1}&\sum_{y\le p_2<w}\frac{1}{ p_2 \log (N/p_1p_2)}\notag\\ &
\le \sum_{z \le p_1 <y}\frac{I(p_1)}{p_1}+\left(\log \frac{8}{3} +\frac{64}{\log^2 N} \right)\left( \frac{10\log (1+\varepsilon_0)}{\log^2 N}+ \frac{27}{\log^3 N}\right).\label{1p1bound}
\end{align}
Next we note that \cite[p. 291]{Nathanson}
\begin{equation*}
    \int_z^y I(u)\mathrm{d}\log \log u=\frac{\overline{c}}{\log N}
\end{equation*}
and upon using the substitution $t=N^{\tau}$,
\begin{align*}
    0=I(y)\le I(z)=\frac{1}{\log N}\int_{1/3}^{7/16}\frac{1}{(\frac{7}{8}-\tau)\tau}\mathrm{d}\tau\le\frac{0.56}{\log N}.
\end{align*}
We can thereby apply Lemma~\ref{lemma:sump} with $f(t)=I(t)$, $g(t)=\log\log t$, 
\begin{equation*}
    c(n)=
    \begin{cases}
        1/n,&\text{if $n$ is prime},\\
        0,&\text{otherwise}
    \end{cases}
\end{equation*}
and $E=1/\log^2 z$ (as a consequence of Lemma \ref{1plem}) to obtain
\begin{align}
\sum_{z \le p_1 <y}\frac{I(p_1)}{p_1}&<\int_z^y I(t)\mathrm{d}\log \log t+ \frac{I(z)}{\log^2 z}\notag\\ & \le \frac{\overline{c}}{\log N} +\frac{36}{\log^3 N}.\label{Iponpbound}
\end{align}
Using \eqref{Iponpbound} to bound \eqref{1p1bound} we can then bound $|\overline{B}|$ in \eqref{overBfirstbound}, which concludes the proof of the lemma.
\end{proof}
Equipped with Lemma~\ref{lemma:Bover}, we now prove Theorem~\ref{theo:B}.
\begin{proof}[Proof of Theorem~\ref{theo:B}]
From \eqref{eq:numb} we see that to bound $S(B,P(y))$ it suffices to bound each $S(B^{(j)},P(y))$ for $0\le j\le j_0$, with $B^{(j)}$ defined in \eqref{bjdef} and $j_0$ defined in \eqref{j0eq}. So, we begin by fixing a value of $j$ and consider the two cases $k_2< K_{\delta}(x_2(Y_j))$ and $k_2\ge K_{\delta}(x_2(Y_j))$. Here, $Y_j$ is as in \eqref{zjyj} and $k_2$, $K_{\delta}$, $x_2$ are as defined in Section \ref{notationsect} with $Y=Y_j$.
\newline

\noindent\textbf{Case 1:} $k_2< K_{\delta}(x_2(Y_j))$.

Let
\begin{align}
\label{eq:D*1}
Q(u):=\prod_{\substack{p< u\\p\nmid N}} p,\quad D^{(3)}:=N^{\frac{1}{2}-\alpha_3},\quad s_b:=\log D^{(3)}/\log y
\end{align}
and 
\begin{align*}
R^{(j)}:=\sum_{\substack{d<D^{(3)}Q(u_0)\\d|P(y)}} \left|r_d^{(j)}\right|
\end{align*}
where $r_d^{(j)}=\left|B_d^{(j)}\right|-\frac{|B^{(j)}|}{\varphi(d)}$, and
\begin{equation*}
    B_d^{(j)}=\sum_{\substack{p_1p_2p_3\equiv N\: (\text{mod }d) \\ z\le p_1<y\le p_2\le p_3,~\omega_j\le p_1<\omega_j(1+\varepsilon_0) \\\omega_jp_2p_3<N,~(p_2p_3,N)=1 }} 1.
\end{equation*}
Now, since $\alpha_3<1/6$, we have
\begin{equation*}
D^{(3)}\ge y=N^{1/3}.
\end{equation*} 
Therefore, we can apply Theorem~\ref{theo:JR} to the set $B^{(j)}$ with $g_n(p)=1/(p-1)$, $Q=Q(u_0)$ and $D=D^{(3)}$ in \eqref{suppereq} to give
\begin{equation*}
S(B^{(j)}, P(y)) <\left|B^{(j)}\right|V(y)(F(s)+\varepsilon C_{1}(\varepsilon) e^2h(s))+R^{(j)}.
\end{equation*}

By \eqref{vjybound} in Lemma~\ref{vjlem} we have $V(y)< 3\frac{U_N}{\log N}\left( 1+\frac{4.51}{\log^2N}\right)$. We also see by \eqref{eq:D*1} that $s_b=\frac{3}{2}-3\alpha_3< 3$ and therefore $F(s_b)=\frac{2e^{\gamma}}{s_b}$ by \eqref{eq: approx-F03}. Hence
\begin{align}\label{bjfirst}
S(B^{(j)},& P(y)) <\\  &\left|B^{(j)}\right|\frac{U_N}{\log N}\left( 1+\frac{4.51}{\log^2N}\right)\left[\frac{2}{\frac{1}{2}-\alpha_3} e^{\gamma}+3\varepsilon C_{1}(\varepsilon) e^2h\left(\frac{3}{2}-3\alpha_3\right)\right]+R^{(j)}.\notag
\end{align}

Now, from the definition of the sets $B^{(j)}$, we obtain
\begin{equation*}
r_d^{(j)}=\sum_{\substack{p_1p_2p_3\equiv N\: (\text{mod }d) \\ z\le p_1<y\le p_2\le p_3,~\omega_j\le p_1<\omega_j(1+\varepsilon_0) \\\omega_jp_2p_3<N,~(p_2p_3,N)=1 }} 1-\frac{1}{\varphi (d)} \sum_{\substack{z\le p_1<y\le p_2\le p_3\\\omega_j\le p_1<\omega_j(1+\varepsilon_0) \\\omega_jp_2p_3<N,~(p_2p_3,N)=1 }}1.
\end{equation*}
We now add the condition $(p_1p_2p_3,d)=1$ to the second sum above. This is equivalent to $(p_1,d)=1$, since the condition $(p_2p_3,d)=1$ already follows from the fact that $d$ divides $P(y)$ and $p_2,p_3\ge y$. This condition decreases the second term above by at most
\begin{equation*}
\frac{1}{\varphi(d)}\sum_{\substack{p_1p_2p_3<(1+\varepsilon_0)N\\p_1|d,p_1 \ge z}}1\le \frac{(1+\varepsilon_0)N}{\varphi(d)}\sum_{\substack{p_1|d,p_1 \ge z}}\frac{1}{p_1}\le \frac{(1+\varepsilon_0)N\omega(d)}{z\varphi(d)}\le \frac{(1+\varepsilon_0)N\log d}{z\varphi(d)\log 2},
\end{equation*}
where the last inequality uses Lemma \ref{omegalem}. We now put \label{def: a_N(n)} $a(n)=a_N(n)$ to be the characteristic function of the set of integers of the form $n=p_2p_3$ with $y\le p_2 \le p_3$ and $(N,p_2p_3)=1$. Then, for $\left|\theta\right|\le 1$ we see that
\begin{equation*}
r_d^{(j)}=\sum_{n<X_j}\sum_{\substack{Z_j \le p <Y_j\\np\equiv N\: (\text{mod }d)}} a(n)-\frac{1}{\varphi (d)}\sum_{n<X_j}\sum_{\substack{Z_j\le p<Y_j \\ (np,d)=1}} a(n)+\frac{(1+\varepsilon_0)\theta N\log d}{z\varphi(d)\log 2},
\end{equation*}
with
\begin{equation} \label{eq: def-final-X-Y-Z}
X_j=\frac{N}{w_j}, \quad
Y_j=\min\left(y,(1+\varepsilon_0)w_j\right) \quad\text{and}\quad
Z_j=w_j.
\end{equation}
With $X=X_j$, $Y=Y_j$ and $Z=Z_j$ the conditions in Lemma~\ref{lemma:bf} hold. Moreover, we also see that, from the condition
\begin{equation*}
    \frac{3^{10}N^{\alpha_3}}{\log^{10}N}\ge\exp\left(u\left(1+\frac{9\cdot 10^{-7}}{\log u}\right)\right) 
\end{equation*}
and Lemma \ref{thetauplem}, \label{def-part-D^*} $D^*:=\frac{\sqrt{X_jY_j}}{\log^{10}Y_j}\ge \frac{\sqrt{N}}{\log^{10} y}\ge D^{(3)}\cdot Q(u_0) $. Therefore, using Lemmas~\ref{muphilem} and \ref{lemma:bf}, and the bound $Y_j\ge z = N^{1/8}$, we obtain
\begin{align}
R^{(j)} \le & \sum_{\substack{d<D^*\\ d|P(y)}}\left|r_d^{(j)}\right|\notag\\
\le & \sum_{\substack{d<D^*\\ d|P(y)}}\left|\sum_{n<X_j}\sum_{\substack{Z_j \le p <X_j\\np\equiv N\: (\text{mod }d)}} a(n)-\frac{1}{\varphi (d)}\sum_{n<X_j}\sum_{\substack{Z_j\le p<X_j\\ (np,d)=1}}a(n)\right|+\sum_{\substack{d<D^*\\d\mid P(y)}}\left(\frac{(1+\varepsilon_0)N\log d}{z\varphi(d)\log 2} \right)\notag\\
\le&\frac{m(X_3)(1+\varepsilon_0)8^3N}{\log^3 N}+\frac{1.1}{\log 2}(1+\varepsilon_0)N^{\frac{7}{8}}\log^2 D^*\notag\\
\le & \frac{m(X_3)(1+\varepsilon_0)8^3N}{\log^3 N}+0.58(1+\varepsilon_0)N^{\frac{7}{8}}\log^2 N\label{rjfirst}
\end{align}
where in the last inequality we have used the bound $D^*\le\sqrt{X_jX_j}\le\sqrt{Ny/z}=N^{29/48}$. Substituting this into \eqref{bjfirst} gives us an upper bound for $S(B^{(j)},P(y))$ in terms of $B^{(j)}$. We now move onto the second case $k_2\ge K_{\delta}(x_2(Y_j))$.\newline

\noindent\textbf{Case 2:} $k_2\ge K_{\delta}(x_2(Y_j))$.

As in the proof of part (b) of Theorem \ref{theo:2.}, we avoid complications with the exceptional zero by working with $P^{(1)}(y)$ as opposed to $P(y)$. In particular, since $S(B^{(j)},P(y))\le S(B^{(j)},P^{(1)}(y))$ it suffices to bound the latter. We also let $Q^{(1)}(u)$ be as in \eqref{Q1eq} with $q_1$ now denoting the largest prime factor of $k_2$. Now, similar to the first case, we apply Theorem \ref{theo:JR} to the set $B^{(j)}$, with $g_n(p)=1/(p-1)$, $Q=Q^{(1)}(u_0)$ and $D=D^{(3)}$ in \eqref{suppereq} to give (cf.\ \eqref{bjfirst})
\begin{align}\label{bjsecond}
    S(B^{(j)},P^{(1)}(y))<\left|B^{(j)}\right|&\frac{U_N(1+\varepsilon_2((X_3)^8,\delta))}{\log N}\left( 1+\frac{4.51}{\log^2N}\right)\\
    &\cdot\left[\frac{2}{\frac{1}{2}-\alpha_3} e^{\gamma}+3\varepsilon C_{1}(\varepsilon) e^2h\left(\frac{3}{2}-3\alpha_3\right)\right]+R^{(1,j)},\notag
\end{align}
where
\begin{equation}\label{R1jeq}
    R^{(1,j)}:=\sum_{\substack{d<D^{(3)}Q^{(1)}(u_0)\\d\mid P^{(1)}(y)}}|r_d^{(j)}|.
\end{equation}
and we have used that
\begin{equation*}
    U_N^{(1)}\le U_N(1+\varepsilon_2((X_3)^8,\delta))
\end{equation*}
by Lemma \ref{e0lem}. We can then bound for $R^{(1,j)}$ in the same way as $R^{(j)}$ in the case $k_2<K_{\delta}(x_2(Y_j))$. However, since in the definition \eqref{R1jeq} of $R^{(1,j)}$ we have $d\mid P^{(1)}(y)$ and thus $k_2\nmid d$, we apply Lemma \ref{bflem2} (as opposed to \ref{lemma:bf}) to obtain
\begin{equation}\label{rjsecond}
    R^{(1,j)} \le  \sum_{\substack{d<D^*\\ d|P^{(1)}(y)}}\left|r_d^{(j)}\right|\le \frac{m^*(X_3)(1+\varepsilon_0)8^3N}{\log^3 N}+0.58(1+\varepsilon_0)N^{\frac{7}{8}}\log^2 N.
\end{equation}
Substituting this into \eqref{bjsecond} gives us an upper bound for $S(B^{(j)},P^{(1)}(y))$ and thus $S(B^{(j)},P(y))$ in terms of $|B^{(j)}|$ when $k_2\ge K_{\delta}(x_2(Y_j))$.\newline

We now combine our bounds for the two cases $k_2<K_{\delta}(x_2(Y_j))$ and $k_2\ge K_{\delta}(x_2(Y_j))$. In particular, by taking the maximum of our expressions for $S(B^{(j)},P(y))$ and $S(B^{(j)},P^{(1)}(y))$ in \eqref{bjfirst} and \eqref{bjsecond} with our bounds for $R^{(j)}$ and $R^{(1,j)}$ in \eqref{rjfirst} and \eqref{rjsecond}, we find that, for all values of $k_2$
\begin{align}\label{bjthird}
    S(B^{(j)},P(y))<\left|B^{(j)}\right|&\frac{U_N(1+\varepsilon_2((X_3)^8,\delta))}{\log N}\left( 1+\frac{4.51}{\log^2N}\right)\\
    &\cdot\left[\frac{2}{\frac{1}{2}-\alpha_3} e^{\gamma}+3\varepsilon C_{1}(\varepsilon) e^2h\left(\frac{3}{2}-3\alpha_3\right)\right]\\
    &+\frac{m(X_3)(1+\varepsilon_0)8^3N}{\log^3 N}+0.58(1+\varepsilon_0)N^{\frac{7}{8}}\log^2 N,\notag
\end{align}
noting that $m^*(X_3)\le m(X_3)$ by their respective definitions in Lemmas \ref{lemma:bf} and \ref{bflem2}. To finish off, we sum \eqref{bjthird} over $0\le j\le j_0$. Namely, using \eqref{eq:numb} along with the fact that $\overline{B}=\bigcup_j B_j$ is a disjoint union and $j_0=\frac{\log(y/z)}{\log(1+\varepsilon_0)}=\frac{5\log N}{24\log(1+\varepsilon_0)}$,
\begin{align*}
    S(B,P(y))&\le\sum_{j\le j_0} S(B^{(j)},P(y))\notag\\
    &\le \left|\overline{B}\right|\frac{U_N(1+\varepsilon_2((X_3)^8,\delta))}{\log N}\left( 1+\frac{4.51}{\log^2N}\right)\notag\\
    &\qquad\cdot\left[\frac{2}{\frac{1}{2}-\alpha_3} e^{\gamma}+3\varepsilon C_{1}(\varepsilon) e^2h\left(\frac{3}{2}-3\alpha_3\right)\right]\notag\\
    &\qquad+\frac{320\cdot m(X_3)(1+\varepsilon_0)N}{3\log(1+\varepsilon_0)\log^2 N}+\frac{0.13(1+\varepsilon_0)N^{\frac{7}{8}}\log^3 N}{\log(1+\varepsilon_0)}.
\end{align*}
Applying our bound for $|\overline{B}|$ in Lemma \ref{lemma:Bover} then completes the proof of the theorem.
\end{proof}

\section{Proof of Theorem~\ref{Theo:B.}}
\label{section:B}
Being now equipped with a lower bound on $S(A,P(z))$ (Theorem~\ref{theo:S>1}) and upper bounds on $\sum_{\substack{z \le q <y}}S(A_q,P(z))$ (Theorem \ref{theo:2.}) and $S(B,P(y))$ (Theorem \ref{theo:B}) we can prove our main result,  Theorem~\ref{Theo:B.}, by using the estimate on $\pi_2(N)$ given in Lemma~\ref{lem:>}. Specifically we need to select suitable values for $X_2$, $X_3$, $\delta$, $\alpha_1$, $\alpha_2$, $\alpha_3$ and $\varepsilon_0$ such that the conditions in Theorems~\ref{theo:S>1}, \ref{theo:2.} and \ref{theo:B} hold, and \eqref{pi2lower} is true for each possible range of $k_1$. This is obtained with $X_2=\exp(\exp(32.7))$, $X_3=\exp(\exp(30.62))$, $\delta=1.478$, $\alpha_1=\alpha_2=\alpha_3=10^{-5}$ and $\varepsilon_0=10^{-4}$. Note that when computing the lower bound from Theorem \ref{theo:S>1}, we can use \eqref{eq:A>} and replace $|A|$ by $\frac{N}{\log N}$ once we have ensured that the lower bound is positive. In particular, with these choices of parameters, we find that for $N\ge X_2\ge X_3^8$
\begin{equation*}
    \pi_2(N)> \frac{1}{10}\cdot\frac{U_N N}{\log^2N}
\end{equation*}
for $k_1<K_{\delta}(x_1(N))$ and
\begin{equation*}
    \pi_2(N)> 2\cdot 10^{-4}\cdot\frac{U_N N}{\log^2N}
\end{equation*}
for $k_1\ge K_{\delta}(x_1(N))$. 

In obtaining these parameters, we found that the most sensitive variable was $\delta$. So, we only roughly optimised over $\alpha_1$, $\alpha_2$, $\alpha_3$ and $\varepsilon_0$ before focusing on finding the value of $\delta$ which allowed us to take the lowest value of $X_2$. Note that increasing $\delta$ causes the bound $\beta_0$ on the Siegel zero to get very large (see \eqref{siegelmin}), whereas taking $\delta$ smaller causes $\varepsilon_i(X_2,\delta)$ to become too large (see Lemma \eqref{e0lem}). It therefore seems that the clearest way to improve our result would be to improve on the Siegel zero bounds we used from \cite{Bordignon1} and \cite{Bordignon2}.

It is however interesting to note that using the technique developed here it would be impossible\footnote{That is, unless some far-reaching result is proven, such as the non-existence of Siegel zeros.} to prove Theorem~\ref{Theo:B.} for $N\ge \exp( \exp(22))$. This is because our lower bound for $\pi_2(N)$ (accounting for the possibility of a large exceptional zero) is at best
\begin{equation*} \label{eq-fancy-F}
\mathcal{F}(X_2,\delta)=2\log 3- (1+\varepsilon_1(X_2,\delta))\log(6)-(1+\varepsilon_2(X_2^{1/8},\delta))\overline{c}
\end{equation*}
and $\mathcal{F}(\exp( \exp(22)),2)<0$. Since taking $\delta$ close to $2$ is very difficult without better bounds on the exceptional zero, it would be tough to even reach $N \ge \exp( \exp(30))$ with the current framework. As a result, a different approach would be required to obtain a substantial improvement to Theorem~\ref{Theo:B.}. In this regard, Cai \cite{Cai1}, Wu \cite{Wu1} and very recently Li \cite{li2024} give an alternate (albeit more complicated) proof of Chen's theorem which is asymptotically superior to the method we adapted from Nathanson \cite{Nathanson}. Therefore, it is likely that an explicit version of these methods (or similar) would give a better result than the one obtained here.

\section{Proof of Corollary \ref{distinctcor} and Theorem~\ref{Theo:B1.}}
\label{section:B1}
In this section we prove Corollary \ref{distinctcor} and Theorem~\ref{Theo:B1.} which follow readily from our main result (Theorem \ref{Theo:B.}). For Corollary \ref{distinctcor} we let $\pi_2(N)$ be as in Theorem \ref{Theo:B.}, and \label{def-pi2-*} $\pi_2^*(N)$ denote the number of representations of an even integer $N$ as the sum of a prime and a square-free number $\eta>1$ with at most two prime factors.
So, let $N>\exp(\exp(32.7))$ be an even integer and consider representations of the form
\begin{equation}\label{Nperep}
    N=p+\eta,
\end{equation}
where $p$ is prime and $\eta$ has at most two prime factors. If $\eta$ is not square-free there are two possible cases: either $\eta=1$, or $\eta$ has two identical prime factors.

For a fixed value of $N$, the case $\eta=1$ corresponds to at most one representation of the form \eqref{Nperep}. That is, either $N-1$ is prime and we set $p=N-1$, or $\eta=1$ does not give any valid representation.

On the hand if $\eta$ has two identical prime factors, $q_1$ and $q_2$, then $q_1=q_2<\sqrt{N}$. As a result, such values of $\eta$ correspond to at most $\sqrt{N}$ representations of the form \eqref{Nperep}. Combining these two cases, we have, by Theorem \ref{Theo:B.},
\begin{equation*}
    \pi_2^*(N)\ge\pi_2(N)-1-\sqrt{N}>2\cdot 10^{-4}\cdot\frac{U_N N}{\log^2 N}-1-\sqrt{N}>0,
\end{equation*}
which proves Corollary \ref{distinctcor}.

We now prove Theorem \ref{Theo:B1.}. For this, we require the following result proved in \cite{Dudek1}.
\begin{theorem}[Dudek]
\label{theo:D}
All integers greater than two can be written as the sum of a prime and a square-free number.
\end{theorem} 
Theorem~\ref{Theo:B1.} now follows from Theorem~\ref{Theo:B.} and, by Theorem~\ref{theo:D}, computing the largest $k$ such that
\begin{equation*}
\prod_{i \le k} p_i \le e^{e^{32.7}}.
\end{equation*}
By \cite[Theorems 3 \& 4]{R-S} we have $e^{29.2}<k< e^{29.3}$. It should be possible to exactly compute $k$, but we will not do so here. We also note that Theorem~\ref{theo:D} was improved by Lee and Francis in \cite{Forrest}, and Hathi and Johnston in \cite{H_J_21}. However, such improvements have a negligible impact on Theorem~\ref{Theo:B1.} unless Theorem \ref{Theo:B.} is substantially improved.\par

\section{Notation index}\label{sectnot}
{As this paper contains a lot of different notation, below we have added page and equation numbers for the definitions of different pieces of notation used. If there is no equation number we will instead state the theorem, lemma, proof etc.\ where the notation first appears. Note that some notations are defined twice: in general (gen.) and then in a particular way to be applied to the proof of the main result (appl.). Note also that throughout the paper, letters $p$, $q$ and any subscripts thereof (e.g.\ $p_i$ and $q_i$) will always denote prime numbers.}

\small{
\begin{align*}
    &\alpha,\textit{p}\pageref{eq:H11},\textit{eq.}\eqref{eq:H11} &
    &\alpha_1,\textit{p}\pageref{theo:S>1},\textit{Thm.}\ref{theo:S>1} \\
    &\alpha_2,\textit{p}\pageref{theo:2.},\textit{Thm.}\ref{theo:2.} &
    &\alpha_3,\textit{p}\pageref{theo:B},\textit{Thm.}\ref{theo:B} \\
    &\beta_k,\textit{p}\pageref{lemma:PNTPAP1},\textit{Lem.}\ref{lemma:PNTPAP1} &
    &\beta_0(x),\textit{p}\pageref{siegelmin},\textit{eq.}\eqref{siegelmin}\\
    &\beta_0^*(x),\textit{p}\pageref{betabound2},\textit{eq.}\eqref{betabound2} &
    &\gamma_{3},\textit{p}\pageref{lem:hsm1},\textit{Lem.}\ref{lem:hsm1}\\
    &\gamma_{s_0},\textit{p}\pageref{lem:hsm1},\textit{Lem.}\ref{lem:hsm1}&
    &\delta,\textit{p}\pageref{notationsect},\textit{Sec.} \ref{notationsect}\\
    &\varepsilon,\textit{p}\pageref{eq:cond1},\textit{eq.}\eqref{eq:cond1} &
    &\varepsilon_i(X_2,\delta),\textit{p}\pageref{eq: def-epsilon0-X2-delta},\textit{eq.}\eqref{eq: def-epsilon0-X2-delta} \\
    &\theta(x),\textit{p}\pageref{eq:theta},\textit{eq.}\eqref{eq:theta}&
    &\kappa_{s_0},\textit{p}\pageref{lemma:forg},\textit{Lem.}\ref{lemma:forg}\\
    &\tilde{\kappa},\textit{p}\pageref{lemma:forg},\textit{Lem.}\ref{lemma:forg}&
    &\nu(x),\textit{p}\pageref{siegelmin},\textit{eq.}\eqref{siegelmin}\\
    &\xi_{s_0},\textit{p}\pageref{eq:tau},\textit{aft.eq.}\eqref{eq:tau}&
    &\tilde{\xi},\textit{p}\pageref{eq:tau},\textit{aft.eq.}\eqref{eq:tau}\\
    &\pi_2(N),\textit{p}\pageref{Theo:B.},\textit{Thm.}\ref{Theo:B.} &
    &\pi^*_2(N),\textit{p}\pageref{def-pi2-*},\textit{Sec.} \ref{section:B1}\\
    &\tau_n,\textit{p}\pageref{eq:tau},\textit{eq.}\eqref{eq:tau}&
    &\tau'_n,\textit{p}\pageref{eq:taup},\textit{eq.}\eqref{eq:taup}\\
    &\varphi^*(n),\textit{p}\pageref{def-varphi-*-n},\textit{pf Lem.}\ref{lem: new-case-b-main} &
    &\chi^*,\textit{p}\pageref{induceeq},\textit{pf Lem.}\ref{small313}\\
    &\chi_0,\textit{p}\pageref{notrivialPNTAP},\textit{Lem.}\ref{notrivialPNTAP}&
    &\omega(n),\textit{p}\pageref{omegalem},\textit{Lem.}\ref{omegalem} \\
    &\omega(n;q,a),\textit{p}\pageref{AAqeq},\textit{Sec.} \ref{sec:ps}&
    &\omega_j,\textit{p}\pageref{j0eq},\textit{eq.}\eqref{j0eq}\\
    &A,\textit{p}\pageref{theo:JR},\textit{Thm.}\ref{theo:JR} \textit{(gen.)} &
    &A,\textit{p}\pageref{Aapeq},\textit{eq.}\eqref{Aapeq} \textit{(appl.)}\\
    &A^{(j)},\textit{p}\pageref{eq: def-mj-Aj-Pj(x)},\textit{eq.}\eqref{eq: def-mj-Aj-Pj(x)}&
    &A_d,\textit{p}\pageref{theo:JR},\textit{Thm.}\ref{theo:JR} \textit{(gen.)} \\
    &A_d,\textit{p}\pageref{Aapeq},\textit{eq.}\eqref{Aapeq} \textit{(appl.)}&
    &a(n)=a_N(n), \textit{p}\pageref{def: a_N(n)},\textit{pf Thm.}\ref{theo:B}\\
    &a(x),\textit{p}\pageref{eq: def-a(X_2)},\textit{eq.}\eqref{eq: def-a(X_2)} &
    &a_1(x),\textit{p}\pageref{eq: def-a_1(X_2)},\textit{eq.}\eqref{eq: def-a_1(X_2)}\\
    &B,\textit{p}\pageref{Bset},\textit{eq.}\eqref{Bset}&
    &\overline{B},\textit{p}\pageref{eq:Bin},\textit{bef.eq.}\eqref{eq:Bin}\\
    &B_{\mu, \varphi}, \textit{p}\pageref{muphilem},\textit{Lem.}\ref{muphilem}&
    &B^{(j)},\textit{p}\pageref{bjdef},\textit{eq.}\eqref{bjdef}\\
    &B_d^{(j)},\textit{p}\pageref{eq:D*1},\textit{aft.eq.}\eqref{eq:D*1}&
    &C_1(\varepsilon),\textit{p}\pageref{tab:fF}, \textit{Table } \ref{tab:fF}\\
    &C_2(\varepsilon),\textit{p}\pageref{tab:fF}, \textit{Table } \ref{tab:fF}&
    &\overline{C}(\varepsilon),\textit{p}\pageref{def-C-upper-bar},\textit{pf Thm.}\ref{theo:S>1}\\
    &\overline{c},\textit{p}\pageref{lemma:Bover},\textit{Lem.}\ref{lemma:Bover}&
    &c(x),\textit{p}\pageref{c1eq},\textit{eq.}\eqref{c1eq}\\
    &c_1(x),\textit{p}\pageref{c1eq},\textit{eq.}\eqref{c1eq}&
    &c_2(x),\textit{p}\pageref{eq: def-c-2},\textit{eq.}\eqref{eq: def-c-2}\\
    &c_3(x),\textit{p}\pageref{eq: def-c-3},\textit{eq.}\eqref{eq: def-c-3}&
    &c_4(x),\textit{p}\pageref{eq: def-c-4},\textit{eq.}\eqref{eq: def-c-4}\\
    &c_4^*(x),\textit{p}\pageref{eq: def-c-4-*},\textit{eq.}\eqref{eq: def-c-4-*}&
    &c_{\alpha,x},\textit{p}\pageref{eq: def-c-alpha1-X2},\textit{eq.}\eqref{eq: def-c-alpha1-X2}\\
    &c_n,\textit{p}\pageref{tab:cn1}, \textit{Table } \ref{tab:cn1}&
    &D,\textit{p}\pageref{eq:cond1},\textit{Thm.}\ref{theo:JR}\\
    &D^*,\textit{p}\pageref{lemma:bf},\textit{Lem.}\ref{lemma:bf} \textit{(gen.)}&
    &D^*,\textit{p}\pageref{def-part-D^*},\textit{pf Thm.}\ref{theo:B} \textit{(appl.)}\\
    &D_0,\textit{p}\pageref{def-D_0},\textit{pf Lem.}\ref{lemma:bf}&
    &D^{(1)},\textit{p}\pageref{eq: TheoA-def-D-s-Q},\textit{eq.}\eqref{eq: TheoA-def-D-s-Q}\\
    &D^{(1)}_j,\textit{p}\pageref{linearlem},\textit{Lem.}\ref{linearlem}&
    &D^{(2)},\textit{p}\pageref{eq: def-Q-D2-D2q-sq},\textit{eq.}\eqref{eq: def-Q-D2-D2q-sq}\\
    &D^{(2)}_q,\textit{p}\pageref{eq: def-Q-D2-D2q-sq},\textit{eq.}\eqref{eq: def-Q-D2-D2q-sq}&
    &D^{(3)},\textit{p}\pageref{eq:D*1},\textit{eq.}\eqref{eq:D*1}\\
    &D_{f}(x; q_1, q_2, l),\textit{p}\pageref{eq: def-D-f-x-q1-q2-l},\textit{eq.}\eqref{eq: def-D-f-x-q1-q2-l}&
    &E_j,\textit{p}\pageref{linearlem},\textit{Lem.}\ref{linearlem}\\
    &\mathcal{E}(x),\textit{p}\pageref{eq: def-mathcal-E},\textit{eq.}\eqref{eq: def-mathcal-E}&
    &E_f(x;k,l),\textit{p}\pageref{eq: E_f(x;k,l)},\textit{eq.}\eqref{eq: E_f(x;k,l)}\\
    &F(s),\textit{p}\pageref{def-f(s)-F(s)},\textit{eq.}\eqref{def-f(s)-F(s)},\eqref{eq: approx-F03},\eqref{eq: approx-F34}&
    &\mathcal{F}(x,\delta),\textit{p}\pageref{eq-fancy-F}, \textit{Sec.}\ref{section:B}\\
    &f(s),\textit{p}\pageref{def-f(s)-F(s)},\textit{eq.}\eqref{def-f(s)-F(s)},\eqref{eq: approx-f02},\eqref{eq: approx-f24}&
    &f(x,\chi),\textit{p}\pageref{eq: f(x,chi)},\textit{eq.}\eqref{eq: f(x,chi)}\\
    &f_n(s),\textit{p}\pageref{eq:f22},\textit{eq.}\eqref{eq:f11}-\eqref{eq:odd}&
    &G(z, \lambda^{\pm}),\textit{p}\pageref{theo:G},\textit{Prop.}\ref{theo:G}\\
    &g_n(d),\textit{p}\pageref{theo:JR},\textit{Thm.}\ref{theo:JR} \textit{(gen.)}& 
    &g_n(d),\textit{p}\pageref{subsec: expl-epsilon},\textit{Sec.} \ref{subsec: expl-epsilon} \textit{(appl.)}\\
    &H := H(N),\textit{p}\pageref{small313},\textit{Lem.}\ref{small313}&
    &H(s),\textit{p}\pageref{eq:H11},\textit{eq.}\eqref{eq:H11}\\
    &h(s),\textit{p}\pageref{eq-def-h(s)},\textit{eq.}\eqref{eq-def-h(s)}&
    &h_n(s)\text{(Sec.\ref{section:LS})},\textit{p}\pageref{eq:hnf},\textit{eq.}\eqref{eq:hnf}\\
    &h_p(t) \text{(Sec.\ref{section:3})},\textit{p}\pageref{overBfirstbound}, \textit{pf Lem.} \ref{lemma:Bover}&
    &I(u),\textit{p}\pageref{overBfirstbound},\textit{pf Lem.} \ref{lemma:Bover}\\
    &\textnormal{Ind}_k,\textit{p}\pageref{lemma:PNTPAP1},\textit{Lem.} \ref{lemma:PNTPAP1}&
    &J(k),\textit{p}\pageref{lemma:exprec},\textit{Lem.} \ref{lemma:exprec}\\
    &j_0,\textit{p}\pageref{j0eq},\textit{eq.}\eqref{j0eq}&
    &K,\textit{p}\pageref{eq:epsilonK},\textit{eq.}\eqref{eq:epsilonK}\\
    &K_{\delta}(x),\textit{p}\pageref{notcon2},\textit{eq.}\eqref{notcon2}&
    &k_i,\textit{p}\pageref{eq:k11},\textit{eq.}\eqref{eq:k11}\\
    &k_{x},\textit{p}\pageref{eq: def-k-x},\textit{eq.}\eqref{eq: def-k-x}&
    &k_0(x_i),\textit{p}\pageref{eq:k11}, \textit{bef.} \eqref{eq:k11}\\
    &\ell,\textit{p}\pageref{eq: def-mj-Aj-Pj(x)}, \textit{bef.} \eqref{eq: def-mj-Aj-Pj(x)}&
    &l(x),\textit{p}\pageref{elleq},\textit{eq.}\eqref{elleq}\\
    &l^*(x),\textit{p}\pageref{ellstareq},\textit{eq.}\eqref{ellstareq}&
    &l_1(x),\textit{p}\pageref{ell1eq},\textit{eq.}\eqref{ell1eq}\\
    &l^*_1(x),\textit{p}\pageref{ell1stareq},\textit{eq.}\eqref{ell1stareq}&
    &l_2(x),\textit{p}\pageref{ell2eq},\textit{eq.}\eqref{ell2eq}\\
    &\text{li}(x),\textit{p}\pageref{eq: def-li(x)},\textit{eq.}\eqref{eq: def-li(x)}&
    &M,\textit{p}\pageref{eq:MM},\textit{eq.}\eqref{eq:MM}\\
    &\overline{m}_{\alpha_1,x},\textit{p}\pageref{eq: def-m-bar-alpha1-X2},\textit{eq.}\eqref{eq: def-m-bar-alpha1-X2}&
    &m_j,\textit{p}\pageref{eq: def-mj-Aj-Pj(x)},\textit{eq.}\eqref{eq: def-mj-Aj-Pj(x)}\\
    &m(x),\textit{p}\pageref{lemma:bf}, \textit{Lem.} \ref{lemma:bf}&
    &m^*(x),\textit{p}\pageref{betabound3}, \textit{Lem.} \ref{bflem2}\\
    &\mathbb{P},\textit{p}\pageref{theo:JR},\textit{Thm.} \ref{theo:JR} \textit{(gen.)}&
    &\mathbb{P},\textit{p}\pageref{eq: TheoA-def-D-s-Q},\textit{pf Thm.} \ref{theo:S>1}\textit{(appl.)}\\
    &P(x),\textit{p}\pageref{theo:JR},\textit{Thm.} \ref{theo:JR} \textit{(gen.)}&
    &P(x),\textit{p}\pageref{notcon2},\textit{eq.}\eqref{notcon2} \textit{(appl.)}\\
    &P^{(j)}(x),\textit{p}\pageref{eq: def-mj-Aj-Pj(x)},\textit{eq.}\eqref{eq: def-mj-Aj-Pj(x)}&
    &p(x),\textit{p}\pageref{peq},\textit{eq.}\eqref{peq}\\
    &p^*(x),\textit{p}\pageref{large533}, \textit{Lem.} \ref{large533}&
    &p_i(x),\textit{p}\pageref{peq},\textit{eq.}\eqref{peq}\\
    &Q(u),\textit{p}\pageref{eq: TheoA-def-D-s-Q},\textit{eq.}\eqref{eq: TheoA-def-D-s-Q}&
    &Q^{(1)}(u),\textit{p}\pageref{Q1eq},\textit{eq.}\eqref{Q1eq}\\
    &Q_1(x),\textit{p}\pageref{notcon2},\textit{eq.}\eqref{notcon2}&
    &q_j,\textit{p}\pageref{eq: def-mj-Aj-Pj(x)}, \textit{bef.} \eqref{eq: def-mj-Aj-Pj(x)}\\
    &R,\textit{p}\pageref{Rdef},\textit{eq.}\eqref{Rdef}&
    &R_i,\textit{p}\pageref{theo:kadiri},\textit{Thm.} \ref{theo:kadiri}\\
    &R^{(j)},\textit{p}\pageref{eq:D*1},\textit{aft.eq.}\eqref{eq:D*1}&
    &R^{(1,j)},\textit{p}\pageref{R1jeq},\textit{eq.}\eqref{R1jeq}\\
    &r(d),\textit{p}\pageref{theo:JR},\textit{Thm.} \ref{theo:JR} \textit{(gen.)}&
    &r(d),\textit{p}\pageref{eq: def-r(d)-rk(d)},\textit{eq.} \eqref{eq: def-r(d)-rk(d)} \textit{(appl.)}\\
    &r_k(d),\textit{p}\pageref{eq: def-r(d)-rk(d)},\textit{eq.} \eqref{eq: def-r(d)-rk(d)}&
    &r_d^{(j)},\textit{p}\pageref{eq:D*1},\textit{aft.eq.}\eqref{eq:D*1}\\
    &S(A,n),\textit{p}\pageref{eq: def-S-A-n},\textit{eq.}\eqref{eq: def-S-A-n}&
    &S(A, \mathbb{P}, z),\textit{p}\pageref{theo:JR},\textit{Thm.} \ref{theo:JR}\\
    &s,\textit{p}\pageref{eq-def-h(s)},\textit{Thm.} \ref{theo:JR}&
    &s^{(1)},\textit{p}\pageref{eq: TheoA-def-D-s-Q},\textit{eq.}\eqref{eq: TheoA-def-D-s-Q}\\
    &s_b,\textit{p}\pageref{eq:D*1},\textit{eq.}\eqref{eq:D*1}&
    &s^{(1)}_j,\textit{p}\pageref{linearlem},\textit{Lem.}\ref{linearlem}\\
    &s^{(2)}_q,\textit{p}\pageref{eq: def-Q-D2-D2q-sq},\textit{eq.}\eqref{eq: def-Q-D2-D2q-sq}&
    &T_N(D,z),\textit{p}\pageref{eq: def-T-n-D-z},\textit{eq.}\eqref{eq: def-T-n-D-z}\\
    &U_N,\textit{p}\pageref{UNeq},\textit{eq.}\eqref{UNeq}&
    &U^{(j)}_N,\textit{p}\pageref{eq: def-UNj},\textit{eq.}\eqref{eq: def-UNj}\\
    &u_0,\textit{p}\pageref{lem: prod-u>u_0}, \textit{Lem.} \ref{lem: prod-u>u_0}&
    &V(x),\textit{p}\pageref{vzdef},\textit{eq.}\eqref{vzdef}\textit{(gen.)}\\ 
    &V(x),\textit{p}\pageref{eq: def-V(x)-Vj(x)},\textit{eq.}\eqref{eq: def-V(x)-Vj(x)}\textit{(appl.) }&
    &V^{(j)}(x),\textit{p}\pageref{eq: def-V(x)-Vj(x)},\textit{eq.}\eqref{eq: def-V(x)-Vj(x)}\\
    &v_0,\textit{p}\pageref{lemma:bf}, \textit{Lem.} \ref{lemma:bf}&
    &v_k(x),\textit{p}\pageref{veq},\textit{eq.}\eqref{veq}\\
    &v_k'(x),\textit{p}\pageref{eq:v1},\textit{eq.}\eqref{eq:v1}&
    &w,\textit{p}\pageref{overBfirstbound}, \textit{aft.eq.} \eqref{overBfirstbound}\\
    &X,\textit{p}\pageref{XYZdef},\textit{eq.}\eqref{XYZdef} \textit{(gen.)}&
    &X,\textit{p}\pageref{eq: def-final-X-Y-Z},\textit{eq.}\eqref{eq: def-final-X-Y-Z} \textit{(appl.)}\\
    &X_1, \textit{p}\pageref{cor:chisum}, \textit{Lem.} \ref{cor:chisum}&
    &X_2,\textit{p}\pageref{notationsect}, \textit{Sec.} \ref{notationsect}\\
    &X_3,\textit{p}\pageref{notationsect}, \textit{Sec.} \ref{notationsect}&
    &X_A, \textit{p}\pageref{eq: def-X_A},\textit{eq.}\eqref{eq: def-X_A}\\
    &X_j,\textit{p}\pageref{eq: def-final-X-Y-Z},\textit{eq.}\eqref{eq: def-final-X-Y-Z}&
    &x_1(N),\textit{p}\pageref{notcon1},\textit{eq.}\eqref{notcon1}\\
    &x_2(Y),\textit{p}\pageref{notcon1},\textit{eq.}\eqref{notcon1}&
    &x^{(n)}_k,\textit{p}\pageref{eq:f4}, \textit{aft.eq.} \eqref{eq:f4}\\
    &Y,\textit{p}\pageref{XYZdef},\textit{eq.}\eqref{XYZdef} \textit{(gen.)}&
    &Y,\textit{p}\pageref{eq: def-final-X-Y-Z},\textit{aft.eq.}\eqref{eq: def-final-X-Y-Z} \textit{(appl.)}\\
    &Y_j,\textit{p}\pageref{zjyj},\textit{eq.}\eqref{zjyj}&
    &y,\textit{p}\pageref{eq: def-z-y},\textit{eq.}\eqref{eq: def-z-y}\\
    &y_n,\textit{p}\pageref{eq: def-y_n}, \textit{Sec.} \ref{ELS}&
    &Z,\textit{p}\pageref{XYZdef},\textit{eq.}\eqref{XYZdef} \textit{(gen.)}\\
    &Z,\textit{p}\pageref{eq: def-final-X-Y-Z},\textit{aft.eq.}\eqref{eq: def-final-X-Y-Z} \textit{(appl.)}&
    &Z_j,\textit{p}\pageref{zjyj},\textit{eq.}\eqref{zjyj}\\
    &z,\textit{p}\pageref{theo:JR},\textit{Thm.} \ref{theo:JR} \textit{(gen.)}&
    &z,\textit{p}\pageref{eq: def-z-y},\textit{eq.}\eqref{eq: def-z-y} \textit{(appl.)}
\end{align*}
}


\newpage

\begin{thebibliography}{99}

\bibitem{Akbary}
A. Akbary and K. Hambrook,
\textit{A variant of the Bombieri-Vinogradov theorem with
explicit constants and applications}
Math. Comp.: 84(294): 1901--1934, 2015.
		
\bibitem{Bordignon1}
M. Bordignon, 
\textit{Explicit bounds on exceptional zeroes of Dirichlet L-functions},
J. Number Theory, 201:68--76, 2019.

\bibitem{Bordignon2}
M. Bordignon, 
\textit{Explicit bounds on exceptional zeroes of Dirichlet L-functions II},
J. Number Theory, 210: 481--487, 2020.

\bibitem{Bordignon4}
M. Bordignon, 
\textit{Medium-sized values for the prime number
theorem for primes in arithmetic
progressions}
New York J. Math., 27:1415--1438, 2021.

\bibitem{Bro}
K. G. Borodzkin,
\textit{On the problem of I. M. Vinogradov's constant},
Proc. Third All-Union Math. Conf., 1956.

\bibitem{Broad}
S. Broadbent, H. Kadiri, A. Lumley, N. Ng and K. Wilk, 
\textit{Sharper bounds for the Chebyshev function $\theta (x)$},
Math. Comp.: 90(331):2281--2315, 2021. 

\bibitem{Buthe1}
J. B\"uthe
\textit{A Brun-Titchmarsh inequality for weighted sums over prime numbers}
Acta Arith. 166.3 289--299, 2014.

\bibitem{Buthe3}
J. B\"uthe
\textit{An analytic method for bounding $\psi(x)$.}
Math. Comp. 87.312: 1991--2009, 2018.

\bibitem{Cai2}
Y. Cai and M. Lu, 
\textit{Chen's theorem in short intervals},
Acta Arith., 91(4):311--323, 1999. 

\bibitem{Cai3}
Y. Cai, 
\textit{Chen's theorem with small primes},
Acta Math. Sin., 18(3):597--604, 2002. 

\bibitem{Cai}
Y. Cai and M. Lu,
\textit{On Chen's theorem},
 In: Jia C., Matsumoto K. (eds) Analytic Number Theory. Developments in Mathematics, vol 6. Springer, Boston, MA, 2002. 

\bibitem{Cai1}
Y. Cai,
\textit{On Chen's theorem. II.}
J. Number Theory, 128(5):1336--1357, 2008.

\bibitem{Cai4}
Y. Cai, 
\textit{A remark on Chen's theorem with small primes},
Taiwanese J. Math., 19(4):1183--1202, 2015.

\bibitem{Car}
M. Car, 
\textit{Le théorème de Chen pour $F_q[X]$},
Dissertationes Math. (Rozprawy Mat.) 223, 54 pp., 1984.

\bibitem{Chen2}
J. R. Chen,
\textit{On the representation of a large even integer as the sum of a prime and the product of at most two primes},
Kexue Tongbao, 17:385--386, 1966.

\bibitem{Chen1}
J. R. Chen,
\textit{On the representation of a large even integer as the sum of a prime and the product of at most two primes},
Sci. Sinica, 16:157--176, 1973.

\bibitem{Chen}
J. R. Chen,
\textit{On the representation of a large even integer as the sum of a
              prime and the product of at most two primes. {II}},
Sci. Sinica, 21(4):421--430, 1978.

\bibitem{Chen3}
J. R. Chen,
\textit{Further improvement on the constant in the proposition ‘1+2’: On the represen-tation of a large even integer as the sum of a prime and the product of at most two primes (II)},
Sci. Sinica, 21(4):477--49, 1978.

\bibitem{Chu}
N. G. Chudakov,
\textit{Introduction to the Theory of Dirichlet $L$-Functions},
OGIZ, Moscow-Leningrad, 1947. 



\bibitem{Davenport}
H. Davenport,
\textit{Multiplicative Number Theory, Third Edition},
Graduate Texts in Mathematics, 74. Springer-Verlag, New York, 2000.

\bibitem{Dudek1}
A. W. Dudek, 
\textit{On the sum of a prime and a square-free number},
Ramanujan J., 42:233--240, 2017.

\bibitem{Forrest}
F. J. Francis and E. S. Lee,
\textit{Additive Representations of Natural Numbers},
Integers, 22 (\#A14), 2022.

\bibitem{Opera}
J.B. Friedlander and H. Iwaniec,
\textit{Opera de Cribro},
American Mathematical Society, Providence RI,  2010.

\bibitem{Greaves}
G. Greaves, 
\textit{Sieves in Number Theory},
Springer-Verlag, Berlin, 2001.

\bibitem{Halberstam}
H. Halberstam,
\textit{A proof of Chen's theorem},
Journées Arithmétiques de Bordeaux, Astérisque, (24-25):281--293, 1975.

\bibitem{H_J_21}
S. Hathi and D. R. Johnston, 
\textit{On the sum of a prime and a square-free number with divisibility conditions},
J. Number Theory, 256:354--372, 2024.

\bibitem{Helfgott}
H. Helfgott,
\textit{The ternary Goldbach problem},
to appear in Ann. of Math. Studies.

\bibitem{Hinz}
J. G. Hinz, 
\textit{Chen's theorem in totally real algebraic number fields},
Acta Arith., 58(4):335--361, 1991. 

\bibitem{Iwaniec}
H. Iwaniec,
\textit{Sieve Methods},
Graduate Course, Rutgers university, New Brunswick, NJ, unpublished notes, 1996.

\bibitem{JohnstonYang}
D. R. Johnston and A. Yang,
\textit{Some explicit estimates for the error term in the prime number theorem}
J. Math. Anal. Appl., 527(2): Paper No. 127460, 23 pp., 2023.

\bibitem{JR}
W. B. Jurkat and H.-E. Richert, 
\textit{An improvement on Selberg's sieve method I},
Acta Arith., 11:207--216, 1965.

\bibitem{Kadiri}
H. Kadiri, 
\textit{An explicit zero-free region for the Dirichlet L-functions},
arXiv:math/0510570v1, 2005.

\bibitem{Kadiri2}
H. Kadiri and A. Lumley
\textit{Short effective intervals containing primes}
Integers, 14 (\#A61), 2014.

\bibitem{Li}
Y. Li and Y. Cai, 
\textit{Chen's theorem with small primes},
Chin. Ann. Math. Ser. B, 32(3):387--396, 2011. 

\bibitem{li2024}
R. Li,
\textit{On Chen's theorem, Goldbach's conjecture and almost prime twins},
https://arxiv.org/abs/2405.05727, 2024.


\bibitem{Lu1}
M. Lu and Y. Cai,
\textit{Chen's theorem in short intervals},
Chinese Sci. Bull., 43(16):1401--1403, 1998. 

\bibitem{Lu}
M. Lu and Y. Cai, 
\textit{Chen's theorem in arithmetical progressions},
Sci. China Ser. A, 42(6):561--569, 1999. 

\bibitem{M-V}
H. L. Montgomery and R. C. Vaughan, 
\textit{The large sieve},
Matematika, 20(40):119--134, 1973.

\bibitem{M-V2}
H. L. Montgomery and R. C. Vaughan,
\textit{Multiplicative number theory. I. Classical
theory}
Cambridge University Press, Cambridge, 2007.

\bibitem{Nathanson}
B. M. Nathanson,
\textit{Additive Number Theory, The classical bases},
 Graduate Texts in Mathematics, 164, Springer-Verlag, New York, 1996.


\bibitem{TOS}
T. Oliveira e Silva, 
\textit{Goldbach Conjecture verification}~,~http://sweet.ua.pt/tos/goldbach.html. 

\bibitem{O_H_P_14}
T. Oliveria e Silva, S. Herzog and S. Pardi
\textit{Empirical verification of the even Goldbach conjecture and computation of prime gaps up to $4\cdot 10^{18}$}
Math. Comp., 83 (288): 2033--2060, 2014.

\bibitem{Guy}
G. Robin, 
\textit{Estimation de la fonction de Tchebychef $\theta$ sur le $k$-ième nombre premier et grandes valeurs de la fonction $\omega(n)$ nombre de diviseurs premiers de $n$},
Acta Arith., 42(4):367--389, 1983.

\bibitem{Renyi}
A. A. R\'enyi,
\textit{On the representation of an even number as the sum of a prime and an
almost prime}, Izv. Akad. Nauk. SSSR 12:57--78 (in Russian), 1948.

\bibitem{Ross}
P. M. Ross, 
\textit{On Chen's theorem that each large even number has the form p1+p2 or p1+p2p3},
J. London Math. Soc. (2), 10(4):500--506, 1975. 

\bibitem{R-S}
J. B. Rosser and L. Schoenfeld, 
\textit{Approximate formulas for some functions of prime numbers},
Illinois J. Math., 6:64--94, 1962.

\bibitem{Vinogradov}
I. M. Vinogradov, 
\textit{Representation of an odd number as a sum of three primes},
Dokl.Akad. Nauk. SSR, 15:291–294, 1937.

\bibitem{Wu}
J. Wu, 
\textit{Chen's double sieve, Goldbach's conjecture and the twin prime problem},
Acta Arith., 114(3):215--273, 2004. 

\bibitem{Wu1}
J. Wu, 
\textit{Chen's double sieve, Goldbach's conjecture and the twin prime problem. II.}
Acta Arith., 131(4):367--387, 2008.

\bibitem{Yamada1}
T. Yamada,
\textit{Explicit Chen's theorem},
https://arxiv.org/abs/1511.03409, 2015.





\end{thebibliography}
\end{document}